\PassOptionsToPackage{table}{xcolor}
\documentclass{article}

\RequirePackage[square,numbers]{natbib}
\RequirePackage{amsthm}
\RequirePackage{amsmath}
\RequirePackage[colorlinks,citecolor=blue,urlcolor=blue]{hyperref}
\RequirePackage[utf8]{inputenc}
\RequirePackage{bussproofs}
\RequirePackage{amssymb}
\RequirePackage{MnSymbol}
\RequirePackage{cmll} 
\RequirePackage{xspace}
\RequirePackage{leftindex}
\RequirePackage{relsize}
\RequirePackage{array}
\newcolumntype{C}[1]{>{\centering\arraybackslash}p{#1}}
\RequirePackage{multirow}
\RequirePackage{graphicx}
\RequirePackage{float}
\RequirePackage{comment}
\RequirePackage{longtable}
\RequirePackage{tikz-qtree}
\RequirePackage{tikz}
\usetikzlibrary{positioning,shapes.multipart,trees,calc}
\usepackage[lmargin=2.5cm,rmargin=2.5cm,bmargin=3.8cm]{geometry}

\makeatletter
\newtheorem*{rep@theorem}{\rep@title}
\newcommand{\newreptheorem}[2]{%
\newenvironment{rep#1}[1]{%
 \def\rep@title{#2 \ref{##1}}%
 \begin{rep@theorem}}%
 {\end{rep@theorem}}}
\makeatother

\theoremstyle{plain}
\newtheorem{theorem}{Theorem}[section]
\newreptheorem{theorem}{Theorem}
\newtheorem{proposition}[theorem]{Proposition}
\newreptheorem{proposition}{Proposition}
\newtheorem{corollary}[theorem]{Corollary}

\newtheorem{lemma}[theorem]{Lemma}
\newreptheorem{lemma}{Lemma}
\newtheorem{example}[theorem]{Example}
\newreptheorem{example}{Example}

\theoremstyle{definition}
\newtheorem{definition}[theorem]{Definition}

\newcommand{\bigrotatedbackslash}{\rotatebox[origin=c]{5}{\ensuremath{\mathsmaller{\backslash}}}}

\newcommand{\bigrotatedslash}{\rotatebox[origin=c]{-5}{\ensuremath{\mathsmaller{/}}}}

\newcommand{\bigdblbackslash}{\mathbin{{\bigrotatedbackslash}\mspace{-4mu}{\bigrotatedbackslash}}}

\newcommand{\bigvvee}{\raisebox{0.5pt}{\ensuremath{\,\mathop{\mathlarger{\mathlarger{\bigdblbackslash\hspace{-0.20ex}\bigrotatedslash}}}}}\xspace}

\newcommand{\bigrotatedbackslashd}{\rotatebox[origin=c]{5}{\ensuremath{\backslash}}}

\newcommand{\bigrotatedslashd}{\rotatebox[origin=c]{-38}{\ensuremath{\backslash}}}

\newcommand{\bigdblbackslashd}{\mathbin{{\bigrotatedbackslashd}\mspace{-5mu}{\bigrotatedbackslashd}}}

\newcommand{\bigvveed}{\raisebox{2.2pt}{\ensuremath{\mathop{\bigdblbackslashd\hspace{-0.8ex}\bigrotatedslashd}}}\xspace}

\DeclareMathOperator*{\intd}{\scalebox{0.51}{\bigvveed}}

\def\disji{\raisebox{-0.9pt}{\rotatebox[origin=c]{-90}{$\!{\mathlarger{\geqslant}}$}}}
\newcommand{\lori}{\,\disji\,}
\newcommand{\vvee}{\lori}

\newcommand{\dep}[2]{=\hspace{-0.1cm}(#1,#2)}

\newcommand{\PL}{\ensuremath{\mathbf{CPL}}\xspace}

\newcommand{\PLVVEE}{\ensuremath{\PL(\!\vvee\!)}\xspace}

\newcommand{\Prop}{\ensuremath{\mathsf{Prop}}\xspace}
\newcommand{\De}{\Delta}
\newcommand{\Ga}{\Gamma}

\newcommand{\Lam}{\Lambda}
\newcommand{\Sig}{\Sigma}

\renewcommand{\phi}{\varphi}
\newcommand{\seq}{\Rightarrow}
\newcommand{\ben}{\bigwedge}
 
\newcommand{\RRR}{\mathcal{R}}
\newcommand{\PPP}{\mathcal{P}}

\makeatletter
\newcommand{\leqnomode}{\tagsleft@true}
\newcommand{\reqnomode}{\tagsleft@false}
\makeatother

\begin{document}

\renewcommand*{\thefootnote}{*}

\title{A Deep-Inference Sequent Calculus for Basic Propositional Team Logic (Without Delving Too Deep)\footnote{The first and third author's research leading to this paper was supported in part by grant 336283 of Academy of Finland. The first author also received funding from the European Research Council (ERC) under the European Union's Horizon 2020 research and innovation programme (grant agreement No 101020762). Part of the first author's research was conducted while he was affiliated with the Department of Mathematics and Statistics, University of Helsinki, Finland. The second author was supported by the Netherlands Organisation for Scientific Research under grant 639.073.807 as well as by the MOSAIC project (EU H2020-MSCA-RISE-2020 Project 101007627). The authors would like to thank Josef Doyle, Marianna Girlando, Søren Brinck Knudstorp, Juha Kontinen, and Elio La Rosa for helpful discussions on the contents of this paper.}}
\author{Aleksi Anttila\textsuperscript{a}, Rosalie Iemhoff\textsuperscript{b}, Fan Yang\textsuperscript{b}}

\maketitle
\renewcommand*{\thefootnote}{\arabic{footnote}}
\setcounter{footnote}{0}

\begin{center}
\noindent
\emph{\footnotesize \textsuperscript{a} Institute for Logic, Language and Computation, University of Amsterdam, Science Park 107, Amsterdam, 1098 XG, The
Netherlands}

\noindent
\emph{\footnotesize \textsuperscript{b} Department of Philosophy and Religious Studies, Utrecht University, Janskerkhof 13, Utrecht, 3512 BL, The Netherlands}

\end{center}

\begin{abstract}
\noindent
We introduce a sequent calculus for the propositional team logic with both the split disjunction and the inquisitive disjunction consisting of a Gentzen-style system ($\mathsf{G3}$-like) for classical propositional logic together with two deep-inference rules for the inquisitive disjunction. We show that the system satisfies various desirable properties: it admits height-preserving weakening, contraction and inversion; it supports a procedure for constructing cutfree proofs and countermodels similar to that for $\mathsf{G3cp}$; and cut elimination holds as a corollary of cut elimination for the $\mathsf{G3}$-style subsystem together with a normal form theorem for cutfree derivations. We also prove a sequent interpolation theorem for the system that yields a novel Lyndon's interpolation theorem for the logic as a corollary.
\end{abstract}

\section{Introduction}\label{di:section:introduction}

Logics such as \emph{dependence logic} and \emph{inquisitive logic} are usually interpreted using \emph{team semantics}: formulas are interpreted with respect to sets of evaluation points (valuations/assignments/possible worlds) called \emph{teams}, rather than single evaluation points as in the usual Tarskian semantics. Team semantics was originally introduced by Hodges \cite{hodges1997,hodges1997b} to provide a compositional semantics for Hintikka and Sandu's \emph{independence-friendly logic} \cite{hintikka1989, hintikka1996}; the idea was further developed by Väänänen in his work on dependence logic \cite{vaananen2007,vaananen2008,vaananen2010}. Independently, Ciardelli, Groenendijk, and Roelofsen developed inquisitive logic \cite{groenendijk,CiardelliRoelofsen2011,inqsembook,ciardellibook} which also essentially employs team semantics (see \cite{yang2014,ciardelli2016dependency}). The use of teams allows for simple and natural ways of formalizing notions such as question meaning (the question as to whether $p$ is the case can be formalized in propositional inquisitive logic as $ p\intd \lnot p$, where $\vvee$ is the \emph{inquisitive disjunction}) and dependence (``the value of $q$ functionally depends on the value of $p$'' can be formalized in propositional dependence logic \cite{yangvaananen2016} using a \emph{dependence atom} $\dep{p}{q}$). We will refer to logics primarily intended to be interpreted using team semantics as \emph{team logics}.

In this paper, we focus on propositional team logics. While there are a great number of natural deduction- and Hilbert-style axiomatizations of propositional team logics in the literature \cite{CiardelliRoelofsen2011,sanovirtema,puncochar,yangvaananen2016,ciardelli2016dependency,yang2017,luck2018,ciardelliIemhoffYang,yang2022}, the development of sequent calculus systems and of proof theory in general for these logics has been slower. The sequent calculi that have been constructed have all been for variants of inquisitive logic; these include multiple labelled systems \cite{sano,ChenMa,Muller,muller2024,litak2025boundedinquisitivelogicssequent} as well as the multi-type display calculus in \cite{frittella2016}. There is additionally one natural deduction system with normalization \cite{Muller}.

One of the main difficulties in providing standard sequent calculi for team logics is that usually these logics are not \emph{closed uniform substitution}. Due to the failure of uniform substitution, axiomatizations for team logics typically feature rules that may only be applied to some subclass of formulas, and these axiomatizations do not admit the usual uniform substitution rule. Many proof-theoretic techniques depend on the universal applicability of the rules, so it is not immediately obvious how to apply these techniques to most team logics (we discuss the difficulties with cut elimination in a setting with restricted rules in more detail in Section \ref{di:section:cut_elimination})---often some specialized machinery has to be introduced to handle this issue. For instance, the construction of the multi-type display calculus in \cite{frittella2016} involves the introduction of a new language featuring two types of formulas for the team logic axiomatized, with closure under substitution holding within each of these types. It is also not even immediate how a sequent should be interpreted in the setting of team semantics---there are, for instance, multiple disjunctions available to interpret the commas in the succedent $\De$ of a sequent $\Ga\seq \De$.

In this paper, we introduce a sequent calculus for the propositional team logic $\PLVVEE$ (studied, for instance, in \cite{yangvaananen2016,sanovirtema,yang2024negation}) that has cut elimination---that is, cut is admissible in the cutfree fragment. This logic, which we will refer to as \emph{basic propositional team logic} in this paper, features both the \emph{split disjunction} $\vee$ (also known as the \emph{tensor} or \emph{local disjunction})---this is the canonical disjunction employed in dependence logic and other logics in the lineage of dependence logic---and the \emph{inquisitive disjunction} $\vvee$ already mentioned above, which is used in inquisitive logic to model the meanings of questions. Basic propositional team logic is a conservative extension of classical propositional logic $\PL$, with the split disjunction $\vee$ extending the classical disjunction; it is also an extension of propositional dependence logic (the extension of $\PL$ with dependence atoms) in the sense that dependence atoms are definable in $\PLVVEE$. Notably, the calculus we introduce is, to our knowledge, the first sequent calculus for a team logic featuring the split disjunction that is based on classical logic. In related work, \cite{Muller} provides a labelled calculus for an intuitionistic variant of propositional inquisitive logic with the split disjunction.

The calculus we construct features a standard Gentzen-style system for $\PL$ with some syntactic restrictions to the effect that certain active formulas and context sets must be classical ($\vvee$-free). Adapting an idea from the natural deduction systems in \cite{yangvaananen2016,yang2017}, this Gentzen-style system is supplemented with two \emph{deep-inference} (see, e.g., \cite{schutte,bull, kashima,pym,guglielmi,brunnler2009,poggiolesi2009,poggiolesi}) \footnote{Some authors (e.g. \cite{brunnler2010nestedsequents}) use `deep inference' to refer in particular to the \emph{calculus of structures} \cite{guglielmi}. We use the term to refer, more broadly, to any approach in which rules are applicable deep within formulas.} rules for the inquisitive disjunction $\vvee$---that is, rules which allow one to introduce the inquisitive disjunction (almost) anywhere within a formula, rather than only as its main connective. The deep-inference rules allow for cutfree completeness of the system and for many standard proof-theoretic techniques to be applied despite the limited applicability of the restricted rules. Essentially, cutfree proofs can be constructed by first constructing cutfree classical proofs, and then introducing inquisitive disjunctions as required; and procedures involving the commuting of sequents which depend on the universal applicability of the rules (such as cut elimination) can be conducted in such a way that they only involve commuting sequents in the classical part of the calculus, in which the rules \emph{are} universally applicable. The cutfree fragment of the calculus has a weak subformula property; we define the relevant notion of weak subformula by generalizing the notion of \emph{resolutions} from inquisitive logic \cite{ciardellibook}.

Our aim was to develop a tractable, transparent sequent calculus for a team logic that has informative proof-theoretic properties. In our case, this means a simple system with cut elimination that departs as little as possible from a Gentzen-style calculus. Our \emph{pro re nata} deep-inference approach accomplishes this in the following ways.

First, we avoid both importing the semantics into the system in the form of labels (like the labelled systems in \cite{ChenMa,Muller}) and extending the syntax of the logic (like the multi-type display calculus in \cite{frittella2016}).

Second, our system consists of only a single pair of rules for each connective: an introduction (right) rule and an elimination (left) rule. This is in contrast with the frequently complex natural deduction systems for team logics, including the natural deduction system for $\PLVVEE$ \cite{yangvaananen2016}.

Third, the system is a natural extension of a well-known Gentzen-style system for $\PL$ (a variant of $\mathsf{G3cp}$ without the implication rules---see, e.g., \cite[pp. 77--78]{troelstra}) extended with rules for the inquisitive disjunction $\vvee$---this means that the fact, mentioned above, that $\PLVVEE$ is an extension of $\PL$ with $\vvee$ is directly reflected in a straightforward way in the calculus, and the calculus allows us to see immediately and transparently exactly what is required to be added to an axiomatization of $\PL$ to axiomatize $\PLVVEE$.

It should be pointed out that this structure is possible due to a key design decision concerning the interpretation of sequents and their structural components. In order to discuss this decision, let us informally introduce the semantics of the split disjunction $\vee$ and the inquisitive disjunction $\vvee$. Below, $t$ is a propositional team---a set of valuations.
\begin{align*}
&t\models \phi \vee \psi &&\text{iff} &&\text{there exist $s,u\subseteq t$ such that $t=s\cup u$, $s\models\phi$, and $u\models\psi$}\\
    &t\models \phi \vvee \psi &&\text{iff} && t\models \phi \text{ or }t\models \psi
\end{align*}
That is, a split disjunction $\phi\vee\psi$ is true in a team just in case the team can be split into two subteams, with each disjunct being true in at least one of the subteams; and an inquisitive disjunction is true in a team just in case one of the disjuncts is. Now, instead of the standard interpretation whereby we take a sequent $\Ga\seq \De$ to be valid just in case whenever each formula in $\Ga$ is true in a team $t$, at least one formula of $\De$ is true in $t$---this is is equivalent to interpreting $\Ga\seq \De$ as $\bigwedge \Ga \models \bigvvee \De$, with the inquisitive disjunction $\vvee$ interpreting the commas in the succedent---we interpret a sequent $\Ga\seq \De$ as $\bigwedge \Ga \models \bigvee \De$, with the split disjunction $\vee$ interpreting the commas in $\De$ (i.e., we take $\Ga\seq \De$ to be valid just in case whenever each formula in $\Ga$ is true in a team $t$, there is, for each $\phi\in \Delta$, a $t_\phi$ such that $t_\phi\models \phi$; and $t=\bigcup_{\phi\in \Delta}t_\phi$). While the inquisitive disjunction $\vvee$ has the standard disjunction semantics (with respect to teams), it is the split disjunction $\vee$ that extends the classical disjunction of $\PL$; therefore, in order to employ, as we have done, a calculus for $\PL$ in the team setting without making extensive changes, succedent commas must be interpreted as $\vee$. The labelled systems in \cite{ChenMa,Muller}, in contrast, interpret the succedent comma as $\vvee$; the multi-type display calculus in \cite{frittella2016} essentially interprets the succedent comma as $\vee$ for one type of the system, and as $\vvee$ for the other.

Fourth and finally, due to the third point above, we are able to show many proof-theoretic results for our system as extensions or corollaries of the analogous results for the classical Gentzen-style system. The following results in this paper are examples of this: the admissibility (with some restrictions) of height-preserving weakening, contraction and inversion (Section \ref{di:section:inversion_etc}); the $\mathsf{G3cp}$-style proof of cutfree completeness and countermodel construction in Section \ref{di:section:countermodels}; the cut elimination procedure (Section \ref{di:section:cut_elimination}); and the sequent interpolation theorem (Section \ref{di:section:interpolation}).

The system also allows us to prove---via the interpolation theorem---some novel results concerning the logic $\PLVVEE$. The literature on interpolation in propositional team logics \cite{dagostino,ciardelliIemhoffYang,yang2022} has mainly focused on uniform interpolation---there are (as far as we know) no results on constructing Craig's interpolants that are not also uniform interpolants, and no results on Lyndon's interpolation. Our sequent interpolation theorem yields as corollaries both a Lyndon's interpolation theorem as well as a Craig's interpolation theorem that does not rely on the construction of (the comparatively more complex) uniform interpolants.

One further interesting feature of the system is that, due to the fact that succedent commas are interpreted with the split disjunction, certain structural rules correspond directly to two important \emph{team-semantic closure properties}: the \emph{empty team property} and \emph{union closure}. The logic $\PLVVEE$ has the empty team property (meaning that the empty team satisfies all $\PLVVEE$-formulas), and its classical fragment is union closed (meaning that the truth of a classical formula in a collection of teams implies its truth in the union of the collection); $\PLVVEE$ as a whole is not union closed. The empty team property corresponds to the soundness of weakening on the right, and union closure corresponds to the soundness of contraction on the right. Therefore, weakening on the right is sound for all $\PLVVEE$-formulas (and indeed admissible in the cutfree fragment of our system for all $\PLVVEE$-formulas), whereas contraction on the right is only guaranteed to be sound (and admissible in the cutfree fragment) for classical formulas.

The paper is structured as follows. In Section \ref{di:section:syntax_etc}, we define the syntax and semantics of the logic $\PLVVEE$ and recall some basic facts about team semantics and this logic. In Section \ref{di:section:resolutions}, we adapt the notion of \emph{resolutions} from inquisitive logic to our setting---resolutions turn out to be a convenient tool for describing how our system functions. In Section \ref{di:section:double_disjunction_system}, we introduce our deep-inference sequent calculus $\mathsf{GT}$ for $\PLVVEE$. We also provide motivation for the deep-inference approach by discussing why a straightforward sequent calculus-translation of the natural deduction system for $\PLVVEE$ would not be cutfree complete, and how the deep-inference rules provide for a natural cutfree extension of this translation. In Section \ref{di:section:inversion_etc}, we prove some basic properties of the system: the height-preserving admissibility (with some restrictions) of  weakening, contraction, and inversion in the cutfree fragment. In Section \ref{di:section:countermodels}, we make use of the (semantic) invertibility of the rules in our system to provide a procedure for constructing cutfree proofs and countermodels that is similar to the analogous procedure for $\mathsf{G3cp}$. This yields a semantic proof of cut elimination. In Section \ref{di:section:partial_resolutions}, we generalize the notion of resolutions to what we call \emph{partial resolutions}, and use partial resolutions to define a weak subformula property for the cutfree fragment of our system as well as to provide a second semantic proof of cut elimination. In Section \ref{di:section:cut_elimination}, we give a syntactic proof of cut elimination---a cut elimination procedure. This makes use of a \emph{normal form theorem} for derivations (similar to what is known in as a \emph{decomposition theorem} in the literature on the calculus of structures): each derivation in the cutfree fragment of $\mathsf{GT}$ can be transformed into a derivation in which one first applies only rules within the classical subsystem of $\mathsf{GT}$, then applies only the right deep-inference rule for $\vvee$, and finally applies only the left deep-inference rule for $\vvee$. Using this theorem, each cut in $\mathsf{GT}$ can be transformed into cuts within the classical subsystem of $\mathsf{GT}$, which can then be eliminated using the cut elimination procedure for this subsystem. In Section \ref{di:section:interpolation}, we utilize the cutfree completeness of the system to prove a sequent interpolation theorem (employing an adaptation of Maehara's method); this also yields Craig's and Lyndon's interpolation as corollaries. In Section \ref{di:section:structural_rule_variant}, we define a variant of $\mathsf{GT}$ featuring independent-context rules instead of shared-context rules with syntactic restrictions. Section \ref{di:section:conclusion} concludes with some discussion concerning the applicability of the deep-inference approach to other team logics.

A preliminary version of this paper appeared in the PhD thesis of the first author \cite{anttila2025}.

\section{Preliminaries} \label{di:section:preliminaries}

\subsection{Syntax, Semantics, and Closure Properties} \label{di:section:syntax_etc}

We fix a (countably infinite) set $\Prop$ of propositional variables.
\begin{definition}[Syntax]
The set of formulas $\alpha$ of \emph{classical propositional logic} \PL is generated by:
\[\alpha::=p\mid\bot \mid\neg\alpha\mid \alpha\wedge\alpha\mid\alpha\vee\alpha\]
where $p\in \Prop$.

The set of formulas $\phi$ of \emph{basic propositional team logic} \PLVVEE is generated by:
\[\phi::= \alpha\mid  \phi\wedge\phi\mid\phi\vee\phi\mid\phi\vvee\phi\]
where $\alpha \in \PL$ (that is, $\alpha$ is a formula of $\PL$).
\end{definition}

We let $\bigvee \emptyset:=\bot$, and $\bigwedge \emptyset:=\lnot \bot$. The connective $\vee$ is the \emph{split disjunction} (also known as the \emph{tensor disjunction}) and $\vvee$ is the \emph{inquisitive disjunction}. We write $\mathsf{P}(\phi)$ for the set of propositional variables appearing in $\phi$, and $\phi(p_1,\ldots, p_n)$ if $\{p_1,\ldots, p_n\}\subseteq \mathsf{P}(\phi)$. For a set/multiset of formulas $\Gamma$, we let $\mathsf{P}(\Gamma):=\bigcup_{\phi\in \Gamma}\mathsf{P}(\phi)$. We also call formulas of $\PL$ \emph{classical formulas}. We reserve the first lowercase Greek letters $\alpha,\beta$ for classical formulas, and the uppercase Greek letters $\Xi,\Lambda,\Theta,\Omega$ for multisets of classical formulas. The set of \emph{subformulas} of a formula is defined in the standard way.

Let $\mathsf{X}\subseteq \Prop$. A \emph{team with domain $\mathsf{X}$} is a set $t\subseteq 2^\mathsf{X}$ of valuations over $\mathsf{X}$. 

\begin{definition}[Semantics]
For any formula $\phi$ and any team $t$ with domain $\supseteq \mathsf{P}(\phi)$, the satisfaction relation $t\models\phi$ is defined inductively by:
\begin{align*}
    &t\models p &&\text{iff} &&\text{for all $v\in t$, $ v(p)=1$;}\\
    &t\models \bot &&\text{iff} &&t=\emptyset;\\
    &t\models \neg\alpha &&\text{iff} &&\text{for all $v\in t$, $\{v\}\nmodels\alpha$;}\\
    &t\models\phi\wedge\psi &&\text{iff} &&\text{$t\models\phi$ and $t\models\psi$;}\\
    &t\models\phi\vee\psi &&\text{iff} &&\text{there exist $s,u\subseteq t$ such that $t=s\cup u$, $s\models\phi$, and $u\models\psi$;}\\
    &t\models\phi\vvee\psi &&\text{iff} &&\text{$t\models\phi$ or $t\models\psi$.}
\end{align*}
For any multiset $\Gamma\cup\{\phi\}$ of formulas, we write $t\models \Gamma$ if $t\models\psi$ for all $\psi\in\Gamma$, and we write $\Gamma\models\phi$ if $t\models \Gamma$ implies $t\models \phi$. We write $\phi\models \psi$ for $\{\phi\}\models \psi$. We write $\models\phi$ if $t\models\phi$ for all teams $t$ with domain $\supseteq \mathsf{P}(\phi)$.
\end{definition}

We define standard \emph{closure properties}. We say that $\phi$ has the
\begin{align*}
&\text{\textbf{empty team property}} &&\text{iff}  &&\emptyset \models \phi;\\
&\text{\textbf{downward closure property}
}  &&\text{iff} &&[t\models \phi\text{ and } s\subseteq t]\implies s\models \phi;\\
&\text{\textbf{union closure property}} &&\text{iff} &&[t\models \phi \text{ for all }t\in T\neq \emptyset]\implies \bigcup T\models \phi;\\
&\text{\textbf{flatness property}} &&\text{iff} &&  t\models \phi \iff [\{v\}\models\phi\text{ for all } v\in t]. 
\end{align*}

It is easy to see that $\phi$ is flat iff it has the empty team, downwards-closure, and union-closure properties.

\begin{proposition}
    All formulas of $\PLVVEE$ satisfy the empty team and downward-closure properties. All formulas of $\PL$ additionally satisfy the union-closure property and hence also the flatness property.
\end{proposition}

We also have that the team semantics of formulas of $\PL$ on singletons coincide with their standard single-valuation semantics: $\{v\}\models \alpha \iff v\models \alpha$. We therefore have that for any $\alpha\in \PL$:
\[t\models \alpha \iff \{v\}\models \alpha \text{ for all }v\in t \iff v\models \alpha \text{ for all }v\in t  .\]

\begin{corollary}\label{di:coro:CPL_entailment_eq_team_entailment}
For any set/multiset $\Lambda\cup\{\alpha\}$ of classical formulas, $\Lambda\models^c\alpha \iff \Lambda\models\alpha$, where $\models^c$ stands for the usual (single-valuation) entailment relation.
\end{corollary}

To see why formulas including $\vvee$ might fail to be union closed, consider the formula $p\vvee \lnot p$. If $v_p\models p$ and $v_{\bar{p}}\models \lnot p$, then $\{v_p\}\models p\vvee \lnot p$ and $\{v_{\bar{p}}\}\models p\vvee \lnot p$, but $\{v_p,v_{\bar{p}}\}\nmodels p\vvee \lnot p$.

$\PLVVEE$ does not admit uniform substitution: writing $\phi(\chi/p)$ for the result of replacing all occurrences of $p$ in $\phi(p)$ with $\chi$, it is not the case that $\phi(p)\models \psi(p)$ implies $\phi(\chi/p)\models \psi(\chi/p)$. For instance, we have $p\vee p\models p$, but $(p\vvee \lnot p)\lor (p\vvee \lnot p)\nmodels p\vvee \lnot p$.

\subsection{Resolutions and Normal Form} \label{di:section:resolutions}

In this section, we adapt the notion of \emph{resolutions} from inquisitive logic \cite{ciardellibook}, and list pertinent results. The proofs of these results are analogous to their proofs in propositional inquisitive logic; see \cite{ciardellibook} for details.

As discussed in Section \ref{di:section:introduction}, the inquisitive disjunction $\vvee$ is used in inquisitive logic as a question-forming connective---for instance, the question as to whether or not $p$ is the case is represented by $p\vvee\lnot p$. It is helpful to think of the resolutions of a formula $\phi$ as representing the answers to $\phi$ (the different possible ways of resolving $\phi$) whenever $\phi$ represents a question. If $\phi$ does not contain instances of $\vvee$, it does not represent a question, and its only resolution is $\phi$ itself. In Section \ref{di:section:partial_resolutions}, we generalize the notion of resolutions to \emph{partial resolutions}, and use partial resolutions to prove a weak subformula property for and the cutfree completeness of our sequent calculus. 

\begin{definition}[Resolutions]
The set $\RRR(\phi)$ of \emph{resolutions} of a formula $\phi$ is defined recursively as follows:
\begin{itemize}
    \item[-] $\mathcal{R}(p):=\{p\}$;
    \item[-] $\mathcal{R}(\bot):=\{\bot\}$;
    \item[-] $\mathcal{R}(\lnot \alpha):=\{\lnot \beta\mid \beta\in \RRR(\alpha)\}$;
    \item[-] $\mathcal{R}(\phi \land \psi):=\{\alpha\land \beta\mid \alpha\in \RRR(\phi)$ and $\beta\in \RRR(\psi)\}$;
    \item[-] $\mathcal{R}(\phi \vee \psi):=\{\alpha\vee \beta\mid \alpha \in \RRR(\phi)$ and $\beta\in \RRR(\psi)\}$;
    \item[-] $\mathcal{R}(\phi \vvee \psi):=\RRR(\phi)\cup\RRR(\psi)$.
\end{itemize}
\end{definition}

Clearly for $\phi\in \PLVVEE$, $\RRR(\phi)\subseteq \PL$; and for $\alpha\in \PL$, $\RRR(\alpha)=\{\alpha\}$.

\begin{definition}[Resolutions for multisets]
A \emph{resolution function for a multiset of formulas $\Gamma\subseteq \PLVVEE$} is a map $f:\Gamma\to \PL$ such that for each $\phi \in \Gamma$, $f(\phi)\in \RRR(\phi)$. The set $\RRR(\Gamma)$ (of multisets)  of \emph{resolutions} of $\Gamma$ is defined as the set
$$\RRR(\Gamma):=\{f[\Gamma]\mid f\text{ if a resolution function for }\Gamma\}$$
of images $f[\Gamma]$ of $\Gamma$ under all resolution functions $f$.
\end{definition}
As above, for $\Gamma\subseteq \PLVVEE$, for each $\Xi\in\RRR(\Gamma)$, $\Xi\subseteq \PL$; and for $\Lambda\subseteq \PL$, $\RRR(\Lambda)=\{\Lambda\}$.

Each formula is equivalent to the inquisitive disjunction of its resolutions (and hence to an inquisitive disjunction of classical formulas), and similarly for multisets of formulas.
\begin{proposition}[Normal form] \label{di:prop:normal_form}
    $\phi \equiv \bigvvee \RRR(\phi)$ and $$t\models \Gamma \iff t\models \Xi \text{ for some }\Xi\in \RRR(\Ga).$$
\end{proposition}
\begin{proof}
Follows from the fact that $\land$ and $\vee$ distribute over $\vvee$: \begin{equation*}
    \phi\vee(\psi\vvee\chi)\equiv(\phi\vee\psi)\vvee(\phi\vee\chi)\text{ and }\phi\land(\psi\vvee\chi)\equiv(\phi\land\psi)\vvee(\phi\land\chi).\qedhere
\end{equation*}
\end{proof}

The \emph{split property} is a generalization of the disjunction property (with respect to $\vvee$): 

\begin{proposition}[Split property] \label{di:prop:split}For $\Lambda\subseteq \PL$ and $\phi,\psi\in \PLVVEE$: $$\Lambda\models \phi \vvee\psi \iff [\Lambda \models \phi \text{ or }\Lambda\models \psi].$$
\end{proposition}
\begin{proof}
    By the union closure of $\PL$ and the downward closure of $\PLVVEE$.
\end{proof}

Finally, a multiset of formulas $\Gamma$ entails $\psi$ just in case each resolution of $\Gamma$ entails some resolution of $\psi$:
\begin{theorem}[Resolution theorem]\label{di:theorem:resolution}
    $$\Gamma \models \psi \iff \text{ for each } \Lambda \in \RRR(\Gamma)\text{ there is some } \alpha \in \RRR(\psi)\text{ such that } \Lambda \models \alpha.$$
\end{theorem}
\begin{proof}
    By Propositions \ref{di:prop:normal_form} and \ref{di:prop:split}.
\end{proof}

\section{The System $\mathsf{GT}$} \label{di:section:double_disjunction_system}

We introduce the sequent calculus $\mathsf{GT}$ for $\PL$ (`G' stands for Gentzen, `T' for team). This is a $\mathsf{G3}$-style system (structural rules are incorporated into the logical rules; see Section \ref{di:section:structural_rule_variant} for an alternative system featuring some explicit structural rules) with deep-inference style rules for the inquisitive disjunction $\vvee$.

Given formulas $\phi_1,\ldots, \phi_n$ and multisets of formulas $\Gamma$ and $\Delta$, we use the notation $\phi_1,\ldots, \phi_n,\Ga$ to denote the multiset $\{\phi_1,\ldots,\phi_n\}\cup\Ga $, and the notation $\Ga,\De$ to denote the multiset $\Ga\cup \De$ (in both cases, $\cup$ denotes the multiset union operation). A \emph{sequent} is an expression of the form $\Gamma \seq \Delta$, where $\Gamma$ and $\Delta$ are finite multisets of formulas. In $\Gamma \seq \Delta$, the multiset $\Gamma$ is the \emph{antecedent} and $\Delta$ is the \emph{succedent}. The intended interpretation of $\Gamma \seq \Delta$ is $\ben \Ga \models \bigvee \De$ (i.e., $ \Ga \models \bigvee \De$); we say that $\Gamma \seq \Delta$ is \emph{valid} if $\ben \Ga \models \bigvee \De$. Let us emphasize that the commas in the succedent of a sequent $\phi_1,\ldots,\phi_n\seq \psi_1,\ldots ,\phi_m$ are interpreted as the split disjunction $\vee$ (while the commas in the antecedent are interpreted as conjunction $\land$): the sequent $\phi_1,\ldots,\phi_n\seq \psi_1,\ldots ,\phi_m$ is valid just in case for any team $t$, if $t\models \phi_1\land \ldots \land \phi_n$, then there are $t_1,\dots,t_n$ such that $t=t_1\cup \ldots \cup t_m$ and $t_1\models \psi_1, \ldots, t_m\models \psi_m$.

We make use of subformula occurrence substitutions to implement our deep-inference rules. We use $\phi[\psi/\chi]$ to denote the result of replacing a particular occurrence of the subformula $\chi$ in $\phi$ with $\psi$ (if $\chi$ does not occur in $\phi$, we let $\phi[\psi/\chi]:=\phi$) . In practice, we will usually use the abbreviation $\phi\{\psi\}$ for $\phi[\psi/\chi]$---$\chi$ will always be either arbitrary (with some restrictions), as in our rules, or clear from the context. We use this notation in the following manner: given a formula such as $\phi=p \land (q \vvee r)$, we write $\phi\{q \vvee r\}$ to denote $\phi$ and simultaneously indicate that $\phi$ has a subformula occurrence $q\vvee r$; i.e., $\phi\{q \vvee r\}$ is shorthand for $\phi[q \vvee r/q \vvee r]$. Then, in the same context, any $\phi\{\eta\}$ will refer to $\phi[\eta/q\vvee r]$. For instance, $\phi\{q\}$ will denote
$\phi[q/q\vvee r]$, i.e., $p\land q$.

\begin{definition}[The sequent calculus $\mathsf{GT}$]\
    \\
    
\noindent
\begin{longtable}{|C{0.50\textwidth} C{0.44\textwidth} |}
\hline
\multicolumn{2}{|c|}{\it Axioms} \\
&\\
 \AxiomC{$\Ga,p \seq p,\De$}
 \DisplayProof { \footnotesize $\mathsf{At}$} & \AxiomC{$\Ga,\bot\seq \De $} \DisplayProof {\it\footnotesize $\mathsf{L}\bot$}\\
 &\\
 \multicolumn{2}{|c|}{\it Logical rules} \\
 &\\
  \AxiomC{$\Gamma\Rightarrow \alpha,\Delta$}
 \RightLabel{\footnotesize$\mathsf{L}\lnot$}
 \UnaryInfC{$\Gamma,\lnot \alpha\Rightarrow \Delta$}
 \DisplayProof &
 \AxiomC{$\Ga,\alpha\Rightarrow \De$}
 \RightLabel{\footnotesize$\mathsf{R}\lnot$}
 \UnaryInfC{$\Ga\Rightarrow \lnot \alpha, \De$}
 \DisplayProof\\
&\\
 \AxiomC{$\Ga, \phi, \psi \seq \De$}
 \RightLabel{{\footnotesize $\mathsf{L}\land$}}
 \UnaryInfC{$\Ga, \phi\land \psi \seq \De$}
 \DisplayProof & 
 \AxiomC{$\Ga\seq \phi,\Lambda$}
 \AxiomC{$\Ga\seq \psi,\Lambda$}
 \RightLabel{{\footnotesize $\mathsf{R}\land$}}
 \BinaryInfC{$\Ga \seq \phi \land \psi,\Lambda,\De$}
 \DisplayProof \\
  &\\
 \AxiomC{$\Ga,\phi\seq \Lambda$}
 \AxiomC{$\Ga,\psi\seq \Lambda$}
 \RightLabel{{\footnotesize $\mathsf{L}\lor$}}
 \BinaryInfC{$\Ga,\phi\lor \psi\seq \Lambda,\De$}
 \DisplayProof & 
 \AxiomC{$\Ga\seq \phi,\psi,\De$}
 \RightLabel{{\footnotesize $\mathsf{R}\lor$}}
 \UnaryInfC{$\Ga\seq \phi\lor\psi,\De$}
 \DisplayProof \\
  &\\
 \AxiomC{$\Ga,\chi\{\phi_L\}\seq \De$}
 \AxiomC{$\Ga,\chi\{\phi_R\}\seq \De$}
 \RightLabel{{\footnotesize $\mathsf{L}\vvee^*$}}
 \BinaryInfC{$\Ga,\chi\{\phi_L\vvee \phi_R\}\seq \De$}
 \DisplayProof & 
 \AxiomC{$\Ga\seq \chi\{\phi_i\},\De$}
 \RightLabel{{\footnotesize $\mathsf{R}\vvee^*$}}
 \UnaryInfC{$\Ga\seq \chi\{\phi_L\vvee\phi_R\},\De$}
 \DisplayProof\\
 &\\
 \multicolumn{2}{|c|}{\it Cut} \\
 &\\
 \multicolumn{2}{|c|}{
  \AxiomC{$\Ga \seq \phi,\De$}
 \AxiomC{$\Pi,\phi \seq \Sig$}
 \RightLabel{{\footnotesize $\mathsf{Cut}$}}
 \BinaryInfC{$\Pi,\Ga\seq \De,\Sig$}
 \DisplayProof }\\
 &\\
 \multicolumn{2}{|c|}{\footnotesize $(*)$ $\chi\{\eta\}$ abbreviates $\chi[\eta/\psi]$, where the occurrence $\psi$ does not occur in the scope of a negation $\lnot$ in $\chi$; and $i\in \{L,R\}.$}\\
  \hline
\end{longtable}
\end{definition}
The sequent(s) above the line in each rule (application) are called the \emph{premise(s)} of that rule (application); the sequent below the line is the \emph{conclusion} of the rule. The formulas $\phi,\psi,\alpha,\chi\{\phi_i\}$ in the premise(s) are the \emph{active formulas} of the rule; the newly introduced formula in the conclusion is the \emph{principal formula} of the rule. The multisets $\Gamma$ ($\Delta,\Lambda)$ are the left (right) context of a given rule; both are also called the \emph{side formulas} of the rule. In $\mathsf{L}\vvee$ and $\mathsf{R}\vvee$, we call the subformula(s) $\phi_i$ ($i\in \{L,R\}$) in the premise(s) the \emph{active subformulas} of the rule and the subformula $\phi_L\vvee\phi_R$ in the conclusion the \emph{principal subformula} of the rule; for all other rules, their active subformula(s) are the same as their active formula(s), and similarly for their principal subformula(s). The formula $\phi$ in $\mathsf{Cut} $ is the \emph{cutformula}, and we say $\mathsf{Cut} $ is a \emph{cut on $\phi$}. We often refer to formula occurrences simply as formulas for simplicity, and disambiguate when it aids clarity.

The rules $\mathsf{L}\neg$, $\mathsf{R}\neg$, $\mathsf{R}\land$, and $\mathsf{L}\vee$ are standard Gentzen-style rules with syntactic restrictions: recall that $\alpha$ ranges over classical formulas whence the active formulas of $\mathsf{L}\lnot$ and $\mathsf{R}\lnot$ must be classical; and that $\Lam$ ranges over multisets of classical formulas, whence the premise right contexts of $\mathsf{R}\land $ and $\mathsf{L}\vee$ must consist of classical formulas. The restrictions on the negation rules are due to the corresponding restriction in the syntax of $\PLVVEE$. The restrictions that the premise right context in $\mathsf{R}\land$ and $\mathsf{L}\vee$ be classical are required for soundness since, for instance, $\mathsf{L}\vee$ without the restriction would not be sound: $p\models p\vvee \lnot p$ and $\lnot p \models p\vvee \lnot p$, but $p\vee \lnot p\nmodels p\vvee \lnot p$. The failure of soundness in general is essentially due to the failure of union closure: restricting to classical (and hence union-closed) premise right contexts ensures soundness. These two rules also feature implicit weakening ($\De$ is added to the right context) to make up for the loss in strength incurred due to the syntactic restrictions---they allow us to prove that right weakening is admissible in the system (see Lemma \ref{di:lemma:weakening}).

Similarly to the logical rules, the structural rules of weakening and contraction (implicit in our $\mathsf{G3}$-style system) are only sound for sequents with formulas with specific closure properties. Specifically,  right weakening $[ \Ga\seq \De,\nu $ implies $ \Ga\seq \De,\nu,\nu]$ is sound for $\nu$ if $\nu$ has the empty team property (or if $\nu$ is any formula in our setting); and right contraction $[ \Ga\seq \De,\nu,\nu $ implies $ \Ga\seq \De,\nu]$ is sound for $\nu$ if $\nu$ is union closed (or if $\nu$ is classical). To see why right contraction is not sound in general with respect to formulas which are not union closed, note that $(p\vvee \lnot p)\vee(p\vvee \lnot p) \seq (p\vvee \lnot p),(p\vvee \lnot p)$ is valid whereas $(p\vvee \lnot p)\vee(p\vvee \lnot p) \seq (p\vvee \lnot p)$ is not. We show corresponding admissibility results for the structural rules in the next section.

The rules $\mathsf{L}\vvee$ and $\mathsf{R}\vvee$ are our deep-inference rules. They are standard classical disjunction rules, strengthened to allow the active/principal subformula(s) to appear within any position (not in the scope of a negation) in the active/principle formula(s). The soundness of these rules follows from the fact that $\land $, $\lor $, and $\vvee$ distribute over $\vvee$. Below is an example application of the rule $\mathsf{L}\vvee$.
\begin{center}\AxiomC{$\Ga,p\land (q\vee r) \seq \Delta$}\AxiomC{$\Ga,p\land (s \vvee (q \land \lnot p))\seq \De$} \RightLabel{\footnotesize$\mathsf{L}\vvee$}\BinaryInfC{$\Ga,p\land ((q\vee r)\vvee (s \vvee (q \land \lnot p)) \seq \De $}\DisplayProof\end{center}

Given any set of sequent calculus rules $\mathsf{C}$, we let $\mathsf{C}$ also denote the system consisting of the rules in $\mathsf{C}$. Let $\mathsf{G3cp}$ (for $\PL$) be the system $\mathsf{GT}\setminus \{\mathsf{L}\vvee,\mathsf{R}\vvee\}$. This is the usual $\mathsf{G3cp}$-system (see, e.g., \cite[pp. 77--78]{troelstra}) without the implication rules, but with the standard rules for the negation. The system $\mathsf{G3cp}$ for $\PL$ is a $\mathsf{G3}$-style cutfree complete system for $\PL$---the soundness (given the standard semantics for classical propositional logic) and cutfree completeness of this system follow easily from that for the usual $\mathsf{G3cp}$. We will refer to $\mathsf{G3cp}$ for $\PL$ simply as `$\mathsf{G3cp}$' since distinguishing between the two systems is unnecessary in the context of this paper. We will often cite results concerning the usual $\mathsf{G3cp}$; that they also hold for $\mathsf{G3cp}$ for $\PL$ is easily confirmed. We write $\mathsf{C}^{-}$ for the $\mathsf{Cut}$-free fragment of $\mathsf{C}$, i.e., $\mathsf{C}^{-}:=\mathsf{C}\setminus \mathsf{Cut}$. We write $\vdash_{\scriptsize \mathsf{C}} \Gamma \seq \Delta$ (or simply $\vdash  \Gamma \seq \Delta$ when the system is clear from the context) if there is a derivation of $\Gamma\seq \Delta$ in calculus $\mathsf{C}$. If $\mathcal{D}$ is a derivation of $ \Gamma\seq \Delta$ in calculus $\mathsf{C}$, we say that $\mathcal{D}$ \emph{witnesses} $\vdash_{\scriptsize \mathsf{C}} \Gamma \seq \Delta$. We define the \emph{height} of a derivation $\mathcal{D}$ inductively: if $\mathcal{D}$ consists of a single axiom, the height of $\mathcal{D}$ is one; otherwise the height of $\mathcal{D}$ is $1$ plus the maximum of the heights of the subderivations ending in the premises of the final rule application in $\mathcal{D}$.

We now verify that the system is sound given the team semantics for $\PLVVEE$.

\begin{theorem}[Soundness of $\mathsf{GT}$] \label{di:theorem:GT_soundness}
    $\vdash \Ga \seq \De$ implies $\Ga\models\bigvee\De$.
\end{theorem}
\begin{proof}
By induction on the height of derivations; we separate into different cases based on the final rule applied in the derivation. We provide proof details only for the non-trivial cases.
\begin{itemize}
    \item[-] $\mathsf{L}\lnot$: Assume $ \Ga \models \alpha \vee\bigvee\De$. We want to show $ \lnot \alpha ,\Ga\models \bigvee \De$, so let $t\models \lnot \alpha \land \ben \Ga$. By $ \Ga \models \alpha \vee\bigvee\De$, we have $t\models \alpha \vee\bigvee\De$, so there are $s,u$ with $t=s\cup u$, $s\models \alpha$, and $u\models \bigvee \De$. By $s\models \alpha$, $t\models \lnot \alpha$, and the flatness of $\alpha$ and $\lnot \alpha$, we have that for each $v\in s$, $v\models \alpha\land \lnot \alpha$; therefore $s=\emptyset$, whence $t=u\models \bigvee \De$.
    \item[-] $\mathsf{R}\lnot$: Assume $\alpha,\Ga\models\bigvee \De$. We want to show $ \Ga\models \lnot \alpha\vee \bigvee\De$, so let $t\models  \Ga$. Let $s:=\{v\in t\mid v\models \alpha\}$ and $u:=t\setminus s$. By the flatness of $\alpha$ we have $s \models \alpha$, and by downward closure we have $s\models \Ga$; therefore by $\alpha, \Ga\models\bigvee \De$, we have $s\models \bigvee \De$. By the flatness of $\lnot \alpha$, we have $u\models \lnot \alpha$, so that by $t=s\cup u$, we have $t\models \lnot \alpha \vee \bigvee \De$.
    \item[-] $\mathsf{R}\land$: Assume $\Ga \models \phi \vee\bigvee \Lambda$ and $ \Ga\models \psi\vee\bigvee \Lambda$. We want to show $\Ga \models  (\phi\land \psi)\vee\bigvee \Lambda$, so let $t\models  \Ga$. Then $t\models \phi \vee\bigvee \Lambda$ and $t\models \psi \vee\bigvee \Lambda$, so $t=s_1\cup u_1$ where $s_1\models \phi$ and $u_1\models \bigvee \Lambda$, and $t=s_2\cup u_2$ where $s_2\models \psi$ and $u_2\models \bigvee \Lambda$. Then also $t= (s_1 \cap s_2)\cup u_1\cup u_2$. By downward closure, $s_1 \cap s_2\models  \phi \land \psi$, and by the union closure of $\bigvee \Lambda$, we have $u_1\cup u_2\models \bigvee \Lambda$. Therefore, $t\models (\phi\land \psi)\vee\bigvee \Lambda$.
    \item[-] $\mathsf{L}\vee$: Assume $\phi, \Ga\models \bigvee \Lambda$ and $\psi,\Ga\models \bigvee \Lambda$. We want to show $\phi \vee\psi, \Ga\models \bigvee \Lambda$, so let $t\models (\phi \vee\psi)\land \ben \Ga$. Then $t=s\cup u$ where $s \models \phi$ and $u\models \psi$, and $t\models \Ga$. By downward closure, $s \models \Ga$ so that $s \models\bigvee \Lambda$; similarly $u\models \bigvee\Lambda$. By the union closure of $\bigvee \Lambda$, we have $t\models\bigvee \Lambda$.
    \item[-] $\mathsf{L}\vvee$: Assume $\chi \{\phi_L\}, \Ga  \models \bigvee \De$ and $\chi \{\phi_R\},\Ga \models \bigvee \De$. We want to show $ \chi \{\phi_L\vvee \phi_R\},\Ga \models \bigvee \De$, so let $t\models \chi \{\phi_L\vvee \phi_R\}\land \ben \Ga  $. Since $\land $, $\vee$, and $\vvee$ distribute over $\vvee$ (i.e., $\phi\vee(\psi\vvee\chi)\equiv(\phi\vee\psi)\vvee(\phi\vee\chi)$, $\phi\land(\psi\vvee\chi)\equiv(\phi\land\psi)\vvee(\phi\land\chi)$, and $\phi\vvee(\psi\vvee\chi)\equiv(\phi\vvee\psi)\vvee(\phi\vvee\chi)$), $\chi\{\phi_L\vvee\phi_R\}\equiv \chi \{\phi_L\}\vvee \chi \{\phi_R\}$, so $t\models  \chi \{\phi_L\}$ or $ \chi \{\phi_R\}$. In either case, by our assumptions, $t\models \bigvee \De$.
    \item[-] $\mathsf{R}\vvee$: Follows immediately from the induction hypothesis and the fact that $\land $, $\vee$, and $\vvee$ distribute over $\vvee$, whence $\chi\{\phi_L\vvee\phi_R\}\equiv \chi \{\phi_L\}\vvee \chi \{\phi_R\}$.

    \item[-] $\mathsf{Cut}$: Assume $ \Ga \models \phi \vee \bigvee \De$ and $\phi, \Pi \models \bigvee \Sigma$. We want to show $ \Ga ,\Pi \models \bigvee \De \vee \bigvee \Sigma$, so let $t\models \ben \Ga \land \ben \Pi$. By $\Ga \models \phi \vee \bigvee \De$, we have $t\models \phi \vee \bigvee \De$, whence $t=s\cup u$, where $s\models \phi$ and $u\models\bigvee \Delta $. By downward closure, $s\models \Pi$, whence $s\models \bigvee \Sigma$. Therefore, $t\models \bigvee \Delta\vee \bigvee \Sigma$. \qedhere
\end{itemize}
\end{proof}

Observe that we employed downward closure in our proof of the soundness of $\mathsf{Cut}$---$\mathsf{Cut}$ is not sound in general if the comma on the right is interpreted as the split disjunction $\vee$.

\begin{theorem}[Cut elimination for $\mathsf{G3cp}$] \label{di:theorem:classical_cut_elimination} \cite[pp. 94--98]{troelstra} If $\mathcal{D}$ witnesses $\vdash_{\scriptsize \mathsf{G3cp}} \Xi \seq \Lambda$, there is an effective procedure for transforming $\mathcal{D}$ into a derivation $\mathcal{D}'$ witnessing $\vdash_{\scriptsize \mathsf{G3cp}^-} \Xi \seq \Lambda$.
\end{theorem}

\begin{corollary}[Cutfree classical completeness] \label{di:coro:cutfree_classical_completeness} $\Xi \models \bigvee \Lambda$ implies $\vdash_{\scriptsize \mathsf{G3cp}^-} \Xi \seq \Lambda$.
\end{corollary}

Let us end this section by commenting on why it is necessary to include deep-inference style rules in our system. One may attempt to obtain a sequent calculus for $\PLVVEE$ by simply extending the system $\mathsf{G3cp}$ with the following seemingly natural translations of natural deduction rules for $\PLVVEE$ \cite{yangvaananen2016} into sequent calculus rules:

\begin{tabular}{C{0.50\textwidth} C{0.44\textwidth} }
 \AxiomC{$\Ga,\phi_L\seq \De$}
 \AxiomC{$\Ga,\phi_R\seq \De$}
 \RightLabel{{\footnotesize $\mathsf{L}\vvee^-$}}
 \BinaryInfC{$\Ga,\phi_L\vvee \phi_R\seq \De$}
 \DisplayProof & 
 \AxiomC{$\Ga\seq \phi_i,\De$}
 \RightLabel{{\footnotesize $\mathsf{R}\vvee^-$}}
 \UnaryInfC{$\Ga\seq \phi_L\vvee\phi_R,\De$}
 \DisplayProof\\
 &\\
\AxiomC{$\Ga,\phi\vee\psi_L \seq \De$}\AxiomC{$\Ga,\phi\vee\psi_R\seq \De$} \RightLabel{\footnotesize$\mathsf{LDstr}$}\BinaryInfC{$\Ga,\phi\vee(\psi_L\vvee\psi_R)\seq \De $}\DisplayProof &
 \AxiomC{$\Ga\seq \phi\vee(\psi_L\vvee\psi_R),\De$}
 \RightLabel{{\footnotesize $\mathsf{R Dstr}$}}
 \UnaryInfC{$\Ga\seq (\phi\vee\psi_L)\vvee(\phi\vee\psi_R),\De$}
 \DisplayProof\\
\end{tabular}\\

This system is, however, not cutfree complete. For instance, the following valid sequent is not derivable:
$$ (r\land x)\vee(( (p\land x)\vvee (q \land x))\vee(y\land x)) \seq (x\land (r\vee (p\vee y)))\vvee (x\land (r\vee(q\vee y))).$$
There is no rule of this system that any putative derivation of this sequent could have concluded with. Any putative derivation could not have concluded with $\mathsf{LDstr}$ or $\mathsf{RDstr}$ because the principal formula of such a rule application would not be of the right form in either case. It could not have concluded with $\mathsf{R}\vvee^-$ or $\mathsf{L}\vee$ because none of the possible premises of such a rule application are valid. For instance, one of the two possible premises of a concluding application of $\mathsf{R}\vvee^-$ is
$$ (r\land x)\vee(( (p\land x)\vvee (q \land x))\vee(y\land x)) \seq x\land (r\vee (p\vee y)),$$
and this sequent is not valid. Finally, it could not have concluded with any other rule because the main connective of the principal formula of such a rule application would not be the connective of that rule.

The problem is that the distributivity rules do not reach sufficiently deep into the active/principal formula (we conjecture that one can formulate a notion of subformula occurrence depth and show that a system with the distributivity rules is cutfree complete for formulas in which $\vvee$ only appears up to certain depth). The deep-inference rules $\mathsf{L}\vvee$ and $\mathsf{R}\vvee$ fix this issue and have the added benefit of allowing for a system consisting only of a single introduction (right) and single elimination (left) rule for each connective.

\section{Weakening, Inversion, Contraction}\label{di:section:inversion_etc}

We mentioned in Section \ref{di:section:double_disjunction_system} that in our setting, right weakening is sound for all formulas, whereas right contraction is only (guaranteed to be) sound for union-closed formulas. In this section, we further show that (left and right) weakening is height-preserving admissible in $\mathsf{GT}^-$; that inversion for all rules except $\mathsf{R}\vvee$ is height-preserving admissible; that left contraction is height-preserving admissible with respect to all formulas; and that right contraction is admissible with respect to classical formulas.

Let $\vdash^n_{\mathsf{GT}^-} \Ga\seq \De$ (or simply $\vdash^n \Ga\seq \De$) denote the fact that there is a $\mathsf{GT}^-$-derivation of $\Ga\seq \De$ of height at most $n$.

\sloppy
\begin{lemma}[Height-preserving weakening] \label{di:lemma:weakening}
Both left and right weakening are height-preserving admissible:
\begin{itemize}
    \item[-] Left weakening: $\vdash^n\Gamma\seq \Delta$ implies $\vdash^n\Gamma,\nu\seq \Delta$.
    \item[-] Right weakening: $\vdash^n\Gamma\seq \Delta$ implies $\vdash^n\Gamma\seq \nu,\Delta$.
\end{itemize}
\end{lemma}
\fussy
\begin{proof}
By induction on $n$. For both items, if $n=1$, the derivation $\mathcal{D}$ witnessing the sequent on the left of the item consists of a single axiom, and then the sequent on the right of the item follows by the same axiom. For the inductive step, we assume both items hold for $n$ and prove that they hold for $n+1$. Letting $R$ denote the final rule applied in $\vdash^{n+1}\Gamma\seq \Delta$, if $R$ is not $\mathsf{R}\land $ or $\mathsf{L}\vee$, both items follow by applying the induction hypothesis to the (subderivation(s) ending in the) premise(s) of $R$, and then applying $R$. If $R$ is $\mathsf{R}\land $ or $\mathsf{L}\vee$, left weakening follows by the induction hypothesis and $R$, as  above, and right weakening follows immediately by the implicit weakening in $R$. For instance, if $R$ is $\mathsf{R}\land$, it is of the form:
\[ 
 \AxiomC{$\Ga\seq\phi,\Lambda$}
 \AxiomC{$\Ga\seq\psi ,\Lambda$}
 \RightLabel{{\footnotesize $\mathsf{R}\land$}}
 \BinaryInfC{$\Ga\seq \phi\land\psi ,\Lambda,\De'$}
 \DisplayProof 
\]
Left weakening: By the induction hypothesis, $\vdash^n\Ga,\nu \seq \phi,\Lambda$ and $\vdash^n\Ga, \nu \seq \psi,\Lambda$, whence by $\mathsf{R}\land$, we have $\vdash^{n+1}\Ga,\nu \seq \phi \land \psi,\Lambda,\De'$.

Right weakening: By $\mathsf{R}\land$ applied (with implicit weakening right context $\Delta',\nu$) to the premises, $\vdash^{n+1}\Ga \seq \nu,\phi \land \psi,\Lambda,\De'$.
\end{proof}

Height-preserving inversion is admissible for all rules except for $\mathsf{R}\vvee$; $\mathsf{R}\vvee$ instead has an inversion-like property which corresponds to the split property (Proposition \ref{di:prop:split}): if a sequent $\Xi\seq \chi\{\phi_L\vvee\phi_R\}$ with classical antecedent $\Xi$ is derivable, then so is either $\Xi\seq \chi\{\phi_L\}$ or $\Xi\seq \chi\{\phi_R\}$ (inversion in the standard sense is not sound for $\mathsf{R}\vvee$ as, for instance, $p\vvee \lnot p\seq p\vvee \lnot p$ is valid whereas $p\vvee \lnot p\seq p$ is not). Observe that for $\mathsf{L}\vee$ and $\mathsf{R}\land$, the inverted rules, unlike the original rules, have no syntactic restrictions and feature no implicit weakening.

We use the following convention. If $\psi$ and $\chi$ are subformula occurrences of $\phi$ such that neither is a subformula occurrence of the other, we write $\phi=\phi\{\psi\}\{\chi\}$ (contrast with $\phi\{\psi\{\chi\}\}$).

\begin{lemma}[Height-preserving inversion]
 \label{di:lemma:inversion}\
\begin{description}
\item[$(\mathsf{L}\lnot)$] $\vdash^{n} \Ga ,\lnot \alpha\seq \De$ implies $\vdash^{n} \Ga\seq\alpha , \De$; 
\item[$(\mathsf{R}\lnot)$] $\vdash^{n} \Ga \seq \lnot \alpha,\De$ implies $\vdash^{n} \Ga,\alpha \seq \De$; 
\item[$(\mathsf{L}\land)$] $\vdash^{n} \Ga, \phi\land\psi \seq \De$ implies $\vdash^{n} \Ga, \phi,\psi \seq \De$; 
\item[$(\mathsf{R}\land)$] $\vdash^{n} \Ga \seq \phi\land\psi,\De$ implies $\vdash^{n} \Ga \seq \phi,\De$ and $\vdash^{n} \Ga \seq \psi,\De$; 
\item[$(\mathsf{L}\lor)$] $\vdash^{n} \Ga, \phi\lor\psi \seq \De$ implies $\vdash^{n} \Ga, \phi\seq \De$ and $\vdash^{n} \Ga, \psi\seq \De$;
\item[$(\mathsf{R}\lor)$] $\vdash^{n} \Ga \seq \phi \vee \psi,\De$ implies $\vdash^{n} \Ga \seq \phi,\psi,\De$; 
\item[$(\mathsf{L}\vvee)$] $\vdash^{n} \Ga, \chi\{\phi_L\vvee\phi_R\} \seq \De$ implies $\vdash^{n} \Ga, \chi\{\phi_L\}\seq \De$ and $\vdash^{n} \Ga, \chi\{\phi_R\}\seq \De$; 
\item[$(\mathsf{R}\vvee)$] $\vdash^{n} \Xi \seq  \chi\{\phi_L\vvee\phi_R\},\De$ implies $\vdash^{n} \Xi\seq \chi\{\phi_L\}, \De$ or $\vdash^{n} \Xi\seq  \chi\{\phi_R\},\De$. 
\end{description}
\end{lemma}
\begin{proof} 
    We say that the displayed formula on the sequent on the left of each item is the active formula of the inversion, and those on the right are its principal formulas; the principal/active subformulas of the inversion are defined in a similar way. Each item is proved by induction on $n$. For each item, if $n=1$, the derivation $\mathcal{D}$ witnessing the sequent on the left of the item consists of a single axiom, with the active formula of the inversion being a side formula of the axiom. In all cases, the sequent(s) on the right then follow by the same axiom. We now assume that each item holds for $n$ and prove that it holds for $n+1$. We separate subcases for each item as follows (note that not all subcases exist for each item):

    Case 1: The active formula of the inversion (e.g., the formula $\lnot \alpha$ in the item $\mathsf{L}\lnot$) is not principal in the final rule $R$ applied in $\mathcal{D}$.

    Case 1.1: The active formula of the inversion is in the left/right context of each premise of $R$. In this case, the result follows by the induction hypothesis applied to the (subderivation(s) ending in the) premise(s) of $R$ followed by application(s) of $R$. For instance, for the item $\mathsf{R}\land $, if $R$ is also $\mathsf{R}\land $, it is, in this case, of the form:
                \[ 
     \AxiomC{$\Ga\seq \eta,\phi\land\psi,\Lambda'$}
     \AxiomC{$\Ga\seq \nu,\phi\land\psi,\Lambda'$}
 \RightLabel{{\footnotesize $\mathsf{R}\land$},}
 \BinaryInfC{$\Ga\seq \eta\land \nu,\phi \land\psi,\Lambda',\De'$}
 \DisplayProof
 \]
  where $\phi\land\psi,\Lambda'=\Lambda$ is the right context of each premise of $R$. Applying the induction hypothesis to the (subderivations ending in the) premises yields $\vdash^{n} \Ga\seq \eta,\phi ,\Lambda'$; $\vdash^{n} \Ga\seq \eta,\psi ,\Lambda'$; $\vdash^{n} \Ga\seq \nu,\phi ,\Lambda'$; and $\vdash^{n} \Ga\seq \nu,\psi ,\Lambda'$. Then two applications of $\mathsf{R}\land$ yield $\vdash^{n+1} \Ga\seq \eta\land \nu,\phi ,\Lambda',\De'$ and $\vdash^{n+1} \Ga\seq \eta\land \nu,\psi ,\Lambda',\De'$.

  Case 1.2: The active formula of the inversion is introduced by the implicit weakening in $R$. In this case, the right context of the conclusion of $R$ is of the form $\Lambda,\phi_a,\De'$, where $\phi_a$ is the active formula of the inversion, and the right context of each premise is $\Lambda$. The result follows by an application of $R$ to the premises with a conclusion right context of the form $\Lambda,\phi_p,\De'$, where $\phi_p$ is a principal formula of the inversion. For instance, for the item $\mathsf{R}\land $, if $R$ is also $\mathsf{R}\land $, it is, in this case, of the form:
    \[ 
     \AxiomC{$\Ga\seq \eta,\Lambda$}
     \AxiomC{$\Ga \seq \nu ,\Lambda$}
 \RightLabel{{\footnotesize $\mathsf{R}\land$}}
 \BinaryInfC{$\Ga\seq \eta\land \nu,\Lambda,\phi \land\psi,\De'$}
 \DisplayProof
 \]
By applying $\mathsf{R}\land$ to the premises, we get $\vdash^{n+1}\Ga\seq \eta\land \nu,\Lambda,\phi,\De' $ and  $\vdash^{n+1}\Ga\seq \eta\land \nu,\Lambda,\psi,\De' $.

Case 2: The active formula of the inversion is principal in $R$.

Case 2.1: The active subformula of the inversion is the principal subformula of $R$. In this case, the (subderivation(s) ending in the) premise(s) of $R$ already yield the desired result (if $R$ is $\mathsf{R}\land $ or $\mathsf{L}\vee$, right weakening must also be applied to the premises). For instance, for the item $\mathsf{R}\land $, $R$ must in this case also be $\mathsf{R}\land $, and it is of the form:
        \[ 
     \AxiomC{$\Ga\seq\phi, \Lambda$}
     \AxiomC{$\Ga\seq\psi, \Lambda$}
 \RightLabel{{\footnotesize $\mathsf{R}\land$}}
 \BinaryInfC{$\Ga\seq \phi\land \psi,\Lambda,\De'$}
 \DisplayProof 
\]
We get $\vdash^{n}\Gamma \seq \phi,\Lambda,\De'$ and $\vdash^{n}\Gamma \seq \psi,\Lambda,\De'$ (whence also $\vdash^{n+1}\Gamma \seq \phi,\Lambda,\De'$ and $\vdash^{n+1}\Gamma \seq \psi,\Lambda,\De'$) by applying height-preserving right weakening (Lemma \ref{di:lemma:weakening}) to the premises.

Case 2.2:  The active subformula of the inversion is not the principal subformula of $R$. In this case, the result follows by the induction hypothesis applied to the premises of $R$ followed by application(s) of $R$. For instance, for the item $\mathsf{L}\vee $, $R$ must in this case be $\mathsf{L}\vvee $, and we may assume without loss of generality that it is of the form:
       \[
     \AxiomC{$\Ga,\phi\{\phi_L\} \vee\psi\seq \De$}
     \AxiomC{$\Ga,\phi\{\phi_R\} \vee\psi\seq \De$}
 \RightLabel{{\footnotesize $\mathsf{L}\vvee$}}
 \BinaryInfC{$\Ga,\phi\{\phi_L\vvee\phi_R\} \vee\psi\seq \De$}
 \DisplayProof 
\]
Applying the induction hypothesis to the premises yields $\vdash^{n} \Ga,\phi\{\phi_L\} \seq \De$; $\vdash^{n} \Ga,\psi \seq \De$; and $\vdash^{n} \Ga,\phi\{\phi_R\} \seq \De$. Then we already have $\vdash^{n+1} \Ga,\psi \seq \De$, and an application of $\mathsf{L}\vvee$ yields $\vdash^{n+1} \Ga,\phi\{\phi_L\vvee\phi_R\} \seq \De$.

For the items $\mathsf{L}\vvee$ and $\mathsf{R}\vvee$, the subcases of Case 2.2 in which $R$ is one of the deep-inference rules follow essentially as above, but this is less easy to see due to the complicated structure of the formulas involved, and there are some minor additional complications. We now provide some further details on these subcases. In each of these subcases, the active formula of the inversion and the principal formula of $R$ is the same formula $\chi$, and the active subformula $\phi_L\vvee\phi_R$ of the inversion is not the principal subformula $\nu_L\vvee \nu_R$ of $R$. The formula $\chi$ is then of one of the following forms:  (i) $\chi\{\phi_L\vvee\phi_R\}\{\nu_L\vvee\nu_R\}$; (ii) $\chi\{(\phi_L\vvee\phi_R)\{\nu_L\vvee\nu_R\}\}$; or (iii) $\chi\{(\nu_L\vvee\nu_R)\{\phi_L\vvee\phi_R\}\}$.

Case 2.2 of item $\mathsf{L}\vvee$; $R$ is $\mathsf{L}\vvee$. In all of the subcases (i), (ii), (iii), the result follows essentially, as above, by applying the induction hypothesis to the premises and then applying $R$. We provide the details for (iii). We may assume without loss of generality that $\chi\{(\nu_L\vvee\nu_R)\{\phi_L\vvee\phi_R\}\}=\chi\{\nu_L\{\phi_L\vvee\phi_R\}\vvee\nu_R\}$. Then $R$ is of the form:  
\[ 
        \AxiomC{$\Ga ,\chi\{\nu_L\{\phi_L\vvee\phi_R\}\}\seq \De$}
            \AxiomC{$\Ga,\chi\{\nu_R\} \seq \De$}
            \RightLabel{{\footnotesize $\mathsf{L}\vvee$}}
            \BinaryInfC{$\Ga, \chi\{\nu_L\{\phi_L\vvee\phi_R\}\vvee\nu_R\} \seq \De$}
            \DisplayProof 
            \]
             By the induction hypothesis applied to the left premise, $\vdash^{n}\Ga ,\chi\{\nu_L\{\phi_L\}\}\seq \De$ and $\vdash^{n}\Ga ,\chi\{\nu_L\{\phi_R\}\}\seq \De$, and from the right premise we have $\vdash^{n}\Ga ,\chi\{\nu_R\}\seq \De$. Then by $\mathsf{L}\vvee$, we have $\vdash^{n+1}\Ga ,\chi\{\nu_L\{\phi_L\}\vvee\nu_R\}\seq \De$ and $\vdash^{n+1}\Ga ,\chi\{\nu_L\{\phi_R\}\vvee\nu_R\}\seq \De$, as required.
  
Case 2.2 of item $\mathsf{R}\vvee$; $R$ is $\mathsf{R}\vvee$. In all of the subcases (i), (ii), (iii), the result follows essentially, as above, by applying the induction hypothesis to the premises and then applying $R$. We provide the details for (iii). We may assume without loss of generality that $\chi\{(\nu_L\vvee\nu_R)\{\phi_L\vvee\phi_R\}\}=\chi\{\nu_L\{\phi_L\vvee\phi_R\}\vvee\nu_R\}$. There are two subcases. In the first, $R$ is of the form:
   \[ 
     \AxiomC{$\Xi \seq\chi\{\nu_L\{\phi_L\vvee\phi_R\}\},\De$}
 \RightLabel{{\footnotesize $\mathsf{R}\vvee$}}
 \UnaryInfC{$\Xi\seq\chi\{\nu_L\{\phi_L\vvee\phi_R\}\vvee\nu_R\}, \De$}
 \DisplayProof 
\]
 By the induction hypothesis applied to the premise, we have $\vdash^{n} \Xi \seq\chi\{\nu_L\{\phi_L\}\},\De$ or $\vdash^{n} \Xi \seq\chi\{\nu_L\{\phi_R\}\},\De$. If the former, then $\vdash^{n+1} \Xi \seq\chi\{\nu_L\{\phi_L\}\vvee\nu_R\},\De$ by $\mathsf{R}\vvee$; if the latter, then $\vdash^{n+1} \Xi \seq\chi\{\nu_L\{\phi_R\}\vvee\nu_R\},\De$ by $\mathsf{R}\vvee$.

 In the second subcase, $R$ is of the form
                       \[ 
     \AxiomC{$\Xi \seq\chi\{\nu_R\},\De$}
 \RightLabel{{\footnotesize $\mathsf{R}\vvee$}}
 \UnaryInfC{$\Xi\seq\chi\{\nu_L\{\phi_L\vvee\phi_R\}\vvee\nu_R\}, \De$}
 \DisplayProof 
\]
Then by $\mathsf{R}\vvee$ applied to the premise, $\vdash^{n+1} \Xi\seq \chi\{\nu_L\{\phi_L\}\vvee\nu_R\}, \De$.
\end{proof}

Let us make a few additional remarks on why the induction in the above proof goes through for item $\mathsf{R}\vvee$ in the cases in which $R$ has two premises. Note first that $R$ cannot be $\mathsf{L}\vvee$ since the antecedent of the conclusion of $\mathsf{L}\vvee$ is always nonclassical, whereas the antecedent of the premise of inverted $\mathsf{R}\vvee$---as per the restriction prescribed in the lemma---must be classical. So $R$ must be $\mathsf{R}\land $ or $\mathsf{L}\vee$. If the active formula $\chi\{\phi_L\vvee\phi_R\}$ of the inversion is a side formula of $R$, it must then have been introduced via the implicit weakening in $R$, and we proceed as in Case 1.2 (instead of introducing $\chi\{\phi_L\vvee\phi_R\}$, we introduce  $\chi\{\phi_L\}$ or  $\chi\{\phi_R\}$). If, on the other hand, $\chi\{\phi_L\vvee\phi_R\}$ is an active formula of $R$, $R$ must be $\mathsf{R}\land$. We may assume without loss of generality that $\chi\{\phi_L\vvee\phi_R\}=\chi_1\{\phi_L\vvee\phi_R\}\land \chi_2$ and that $R$ is of the form:
 \[ 
     \AxiomC{$\Xi \seq\chi_1\{\phi_L\vvee\phi_R\},\Lambda$}
     \AxiomC{$\Xi \seq\chi_2,\Lambda$}
 \RightLabel{{\footnotesize $\mathsf{R}\land$},}
 \BinaryInfC{$\Xi \seq\chi_1\{\phi_L\vvee\phi_R\}\land \chi_2,\Lambda,\De'$}
 \DisplayProof 
\]
We proceed essentially as in Case 2.2, by first applying the induction hypothesis to the left premise and then applying $R$.

\begin{lemma}[Height-preserving contraction] \label{di:lemma:contraction}
Left contraction is height-preserving admissible, and right contraction is height-preserving admissible with respect to classical formulas:
\begin{itemize}
    \item[-] Left contraction: $\vdash^n\Gamma,\nu,\nu\seq \Delta$ implies $\vdash^n\Gamma,\nu\seq \Delta$.
    \item[-] Right contraction: $\vdash^n\Gamma\seq \alpha,\alpha, \Delta$ implies $\vdash^n\Gamma\seq \alpha,\Delta$.
\end{itemize}
\end{lemma}
\begin{proof}
By simultaneous induction on $n$. For both items, if $n=1$, the derivation $\mathcal{D}$ witnessing the sequent on the left of the item consists of a single axiom, and then the sequent on the right of the item follows by the same axiom. We now assume that both items hold for $n$ and prove that they hold for $n+1$. Let $R$ denote the final rule applied in $\mathcal{D}$. We separate subcases for each item as follows: 

Case 1: The formula (occurrence) to be contracted is not the principal formula of $R$.

Case 1.1: Both of the formula occurrences $\nu$/$\alpha$ to be contracted are in the left/right context of each premise of $R$ (that is, neither occurrence is introduced via implicit weakening). In this case, the result follows by the induction hypothesis applied to the (subderivation(s) ending in the) premise(s) of $R$ followed by an application of $R$. For instance, for right contraction, if $R$ is $\mathsf{L}\lnot$, it is of the form:
\[
 \AxiomC{$\Ga'\seq \beta,\alpha,\alpha,\De$}
 \RightLabel{{\footnotesize $\mathsf{L}\lnot$}}
 \UnaryInfC{$\Ga',\lnot \beta \seq \alpha,\alpha,\De$}
 \DisplayProof 
\]
By the induction hypothesis applied to the premise followed by an application of $\mathsf{L}\lnot$, we have $\vdash^{n+1}\Ga',\lnot \beta\seq \alpha,\De$.

Case 1.2: At least one of the formula occurrences $\alpha$ to be contracted is introduced by the implicit weakening in $R$. In this case, we replace the application of $R$ by another application in which one fewer occurrence is introduced. For instance, if $R$ is $\mathsf{R}\land$, it is of one of the following forms:
\[ 
 \AxiomC{$\Ga\seq \phi, \Lambda$}
 \AxiomC{$\Ga\seq \psi, \Lambda$}
 \RightLabel{{\footnotesize $\mathsf{R}\land$}}
 \BinaryInfC{$\Ga \seq \alpha,\alpha,\phi\land\psi,\Lambda,\De'$}
 \DisplayProof \hspace{0.5cm} \AxiomC{$\Ga\seq \alpha, \phi, \Lambda$}
 \AxiomC{$\Ga\seq \alpha,\psi, \Lambda$}
 \RightLabel{{\footnotesize $\mathsf{R}\land$}}
 \BinaryInfC{$\Ga \seq \alpha,\alpha,\phi\land\psi,\Lambda,\De'$}
 \DisplayProof 
\]
In either case, by $\mathsf{R}\land$ applied to the premises, $\vdash^{n+1} \Ga \seq \alpha,\phi\land\psi,\Lambda,\De'$.

Case 2: The formula occurrence to be contracted is the principal formula of $R$. In this case, the result follows by application(s) of the inversion lemma to the (subderivations ending in the) premise(s), followed by application(s) of the induction hypothesis (for one of the two items), followed by an application of $R$. We give three examples.

For left contraction, if $R$ is $\mathsf{L}\lnot$, we have that $\nu=\lnot \beta$ and $R$ is of the form:
 \[
 \AxiomC{$\Ga,\lnot \beta \seq \beta,\De$}
 \RightLabel{{\footnotesize $\mathsf{L}\lnot$}}
 \UnaryInfC{$\Ga,\lnot \beta,\lnot \beta \seq \De$}
 \DisplayProof 
\]
By the $\mathsf{L}\lnot$-case of Lemma \ref{di:lemma:inversion} applied to the premise, $\vdash^{n}\Ga\seq \beta,\beta,\De$. Then by the induction hypothesis (for right contraction), we have
$\vdash^{n}\Ga\seq \beta, \De$, whence by $\mathsf{L}\lnot$, we have $\vdash^{n+1}\Ga,\lnot \beta \seq \De$.

For right contraction, if $R$ is $\mathsf{R}\vee$, we have that $\alpha=\beta_1\vee\beta_2$ and $R$ is of the form:
\[
 \AxiomC{$\Ga\seq \beta_1,\beta_2,\beta_1\vee\beta_2, \De$}
 \RightLabel{{\footnotesize $\mathsf{R}\vee$}}
 \UnaryInfC{$\Ga \seq \beta_1\vee\beta_2,\beta_1 \vee \beta_2,\De$}
 \DisplayProof 
\]
 By the $\mathsf{R}\vee$-case of Lemma \ref{di:lemma:inversion} applied to the premise, $\vdash^{n}\Ga \seq \beta_1,\beta_2,\beta_1,\beta_2,\De$. Then, by applying the induction hypothesis twice, we obtain first $\vdash^{n}\Ga \seq \beta_1,\beta_1,\beta_2,\De$ and then $\vdash^{n}\Ga \seq \beta_1,\beta_2,\De$, whence by $\mathsf{R}\vee$, we have $\vdash^{n+1}\Ga \seq \beta_1\vee\beta_2,\De$.

For left contraction, if $R$ is $\mathsf{L}\vvee$, we have that $\nu=\chi\{\phi_L\vvee\phi_R\}$ and $R$ is of the form:
 \[
 \AxiomC{$\Ga,\chi\{\phi_L\vvee\phi_R\},\chi\{\phi_L\}\seq \De$}
 \AxiomC{$\Ga,\chi\{\phi_L\vvee\phi_R\},\chi\{\phi_R\}\seq \De$}
 \RightLabel{{\footnotesize $\mathsf{L}\vvee$}}
 \BinaryInfC{$\Ga,\chi\{\phi_L\vvee\phi_R\},\chi\{\phi_L\vvee\phi_R\} \seq \De$}
 \DisplayProof 
 \]
By the $\mathsf{L}\vvee$-case of Lemma \ref{di:lemma:inversion} applied to the premises, $\vdash^{n}\Ga,\chi\{\phi_L\},\chi\{\phi_L\}\seq \De$ and $\vdash^{n}\Ga\seq\chi\{\phi_R\},\chi\{\phi_R\} \seq \De$. Then by the induction hypothesis, $\vdash^{n}\Ga,\chi\{\phi_L\} \seq \De$ and $\vdash^{n}\Ga,\chi\{\phi_R\} \seq \De$, whence by $\mathsf{L}\vvee$, we have $\vdash^{n+1}\Ga,\chi\{\phi_L\vvee \phi_R\} \seq \De$. 

Note that for right contraction, if $R$ is $\mathsf{R}\vvee$, Case 2 does not apply since the principal formula of $R$ is not classical and hence cannot be the target formula of the contraction.
\end{proof}

\section{Countermodel Construction}\label{di:section:countermodels}

Due to the (semantic) invertibility of the rules, together with the split property (Proposition \ref{di:prop:split}), there is a simple way of proving the cutfree completeness of $\mathsf{GT}$ that is similar to one for $\mathsf{G3cp}$ \cite[pp. 104--105]{troelstra}: roughly, given a sequent $S$, one applies inverted versions of the rules until one arrives at \emph{atomic sequents}---sequents consisting solely of propositional variables and $\bot$. Given semantic invertibility and split, if $S$ is valid, so are (some of) these atomic sequents, whence they are axioms; one will then have constructed a derivation of $S$ consisting of these axioms and the rules of $\mathsf{GT}$ corresponding to the inverted rules applied in the process. Conversely, if $S$ is invalid, some of these atomic sequents must also be invalid, and these invalid atomic sequents then enable one to construct a countermodel for $S$.

The following lemma yields an effective procedure for constructing a countermodel to an invalid sequent using countermodels to atomic sequents derived from the original invalid sequent via inverted rules---it is the (constructive) semantic counterpart to the inversion lemma (Lemma \ref{di:lemma:inversion}). Observe that---in accordance with the split property---the step in the procedure corresponding to inverted $\mathsf{R}\vvee$ requires that the antecedent of the relevant sequent be classical. The steps corresponding to the inverted versions of the restricted-context rules $\mathsf{R}\land$ and $\mathsf{L}\vee$ omit the implicit weakening of the original rules and work for any succedent context (whether classical or non-classical). We call a team $t\in 2^\mathsf{X}$ such that $t\models \bigwedge\Pi$ and $t\not \models \bigvee \Sigma$ a \emph{countermodel to $\Pi\seq \Sigma$ on $\mathsf{X}$.}

\begin{lemma} Given a sequent $\Pi\seq \Sigma$ and any $\mathsf{X}\subseteq \mathsf{Prop}$ such that $\mathsf{P}(\Pi)\cup \mathsf{P}(\Sigma)\subseteq \mathsf{X}$,\label{di:lemma:countermodels}\
\begin{itemize}
    \item[-] if $\Pi\seq \Sigma$ is atomic, there is an effective procedure for constructing a countermodel to $\Pi\seq \Sigma$ on $\mathsf{X}$;
    \item[-] if $\Pi\seq \Sigma$ is of the form of the  conclusion of any logical rule application (apart from $\mathsf{R}\land$, $\mathsf{L}\vee$, or $\mathsf{R}\vvee$), then there is an effective procedure for constructing a countermodel to $\Pi\seq \Sigma$ on $\mathsf{X}$ from a countermodel on $\mathsf{X}$ to any premise of that application;
    \item[-] $\mathsf{R}\land$: if $\Pi\seq \Sigma$ is of the form $\Gamma\seq \phi\land \psi,\De$, then there is an effective procedure for constructing a countermodel to $\Pi\seq \Sigma$ on $\mathsf{X}$ from a countermodel on $\mathsf{X}$ to either $\Gamma\seq \phi,\De$ or $\Gamma\seq \psi,\De$;
    \item[-] $\mathsf{L}\vee$: if $\Pi\seq \Sigma$ is of the form $\Gamma, \phi\vee\psi\seq \De$, then there is an effective procedure for constructing a countermodel to $\Pi\seq \Sigma$ on $\mathsf{X}$ from a countermodel on $\mathsf{X}$ to either $\Gamma,\phi\seq \De$ or $\Gamma,\psi\seq \De$;
    \item[-] $\mathsf{R}\vvee$: if $\Pi\seq \Sigma$ is of the form $\Xi\seq \chi\{\psi_L\vvee\psi_R\}$, then there is an effective procedure for constructing a countermodel to $\Pi\seq \Sigma$ on $\mathsf{X}$ given countermodels on $\mathsf{X}$  to both $\Xi\seq \chi\{\psi_L\},\De$ and $\Xi\seq \chi\{\psi_R\},\De$.
\end{itemize}
\end{lemma}
\begin{proof} We omit some easy cases in which a countermodel to the premise of a rule is also a countermodel to its conclusion.
       \begin{itemize}
       \item[-] Atomic sequent: If $\Pi \seq \Sigma$ is an invalid atomic sequent, define $v\in 2^{\mathsf{X}}$ by $v(p)=1$ iff $p\in \Pi$. We show that $\{v\}$ is a countermodel to $\Pi \seq \Sigma$. Clearly $\{v\}\models \bigwedge \Pi $. Since the atomic sequent $\Pi \seq \Sigma$ is invalid, $\Pi$ and $\Sigma$ must be disjoint, whence $v\not\models \bigvee \Sigma $, so that also $\{v\}\not\models \bigvee \Sigma $. (Note that in the special case in which $\mathsf{X}=\emptyset$, $v$ is the empty valuation $\emptyset$, and the countermodel $\{v\}$ is the nonempty team $\{\emptyset\}$ containing the empty valuation.)
        \item[-] $\mathsf{L}\lnot$: Assume $t$ is a countermodel to the premise of the rule $\mathsf{L}\lnot$, i.e., $t\models \bigwedge \Ga$ and $t\nmodels \alpha\vee \bigvee \De$. Let $s=\{v\in t\mid v\not\models\alpha\}\subseteq t$. We show that $s$ is a countermodel to the conclusion of the rule. Clearly, $s\models\neg\alpha$, and  by downward closure, $s\models \bigwedge \Ga$. Finally, to see $s\not\models \bigvee \De$, assume otherwise. Since $t\setminus s\models\alpha$  holds, we have $(t\setminus s)\cup s=t\models\alpha\vee \bigvee \De$, which is a contradiction.
        \item[-]  $\mathsf{R}\lnot$: Assume $t\models \bigwedge \Ga \land \alpha$ and $t\nmodels \bigvee \De$. We show that $t$ is a countermodel for the conclusion of the rule, that is, $t\not\models \lnot\alpha\vee\bigvee \De$. Since $t\models \alpha$, no nonempty subset of $t$ satisfies $\neg\alpha$. Thus, if $t\models \lnot\alpha\vee\bigvee \De$ were to hold, it must be that $t\models \bigvee \De$, which contradicts the assumption.

        \item[-]  $\mathsf{L}\vee$: Assume $t\models \bigwedge \Ga\land \phi$ and $t\nmodels \bigvee \De$. Clearly, $t$ is our countermodel, since $t\models \bigwedge \Ga\land (\phi\vee\psi)$ by the empty team property. Similarly for the other premise.
        \item[-]  $\mathsf{L}\vvee$: Assume $t\models \bigwedge \Ga\land \chi\{\phi_L\}$ and $t\nmodels \bigvee \De$. The team $t$ is our countermodel: clearly $t\models \bigwedge \Ga\land \chi\{\phi_L\}\vvee \chi\{\phi_R\}$, whence $t\models \bigwedge \Ga\land \chi\{\phi_L\vvee\phi_R\}$ since $\vee $, $\land$, and $\vvee$ distribute over $\vvee$.
        \item[-]  $\mathsf{R}\vvee$: Assume $t_1\models \bigwedge \Xi$ and $t_1\nmodels \chi\{\phi_L\}\vee\bigvee\De$ and $t_2\models \bigwedge \Xi$ and $t_2\nmodels \chi\{\phi_R\}\vee\bigvee\De$. The team $t_1\cup t_2$ is our countermodel: by the union closure of $\bigwedge \Xi$, we have $t_1\cup t_2\models \bigwedge \Xi$. By downward closure, $t_1\cup t_2\nmodels \chi\{\phi_L\}\vee\bigvee\De$ and $t_1\cup t_2\nmodels \chi\{\phi_R\}\vee\bigvee\De$ so $t_1\cup t_2\nmodels \chi\{\phi_L\vvee\phi_R\}\vee\bigvee\De$ since $\vee $, $\land$, and $\vvee$ distribute over $\vvee$.
        \qedhere
    \end{itemize}
\end{proof}

We introduce some metalanguage notation for convenience. We use $\doublewedge$ as a metalanguage `and', and $\doublevee$ as a metalanguage `or'. We call a collection of sequents joined together by the metalanguage connectives a \emph{Boolean combinations of sequents}. For instance, $(\Ga \seq \phi,\Lambda \mathrel{\doublewedge} \Ga \seq \psi,\Lambda)\doublevee \Ga, \chi \seq \Lambda$ is a Boolean combination of sequents (we also allow combinations that contain parentheses that govern the order of operations). We will use these combinations to express derivability and semantic invertibility facts in the following manner:
\begin{description}
    \item[($\mathsf{R}\land $)] $(\Gamma\seq \phi,\Lambda \mathrel{\doublewedge} \Gamma\seq \psi,\Lambda) \leftmodels\vdash_{\scriptsize \mathsf{\mathsf{R}\land}} \Gamma\seq \phi\land \psi, \Lambda$
    \item[($\mathsf{R}\vvee$)] $(\Xi\seq\chi\{\phi_L\},\De \mathrel{\doublevee}  \Xi\seq\chi\{\phi_R\},\De)\leftmodels\vdash_{\scriptsize \mathsf{\mathsf{R}\vvee}} \Xi\seq \chi\{\phi_L\vvee \phi_R\},\De$
\end{description}
The left-to-right direction of the first line above reformulates the rule $\mathsf{R}\land$ (without implicit weakening): given both of the sequents $\Gamma\seq \phi,\Lambda $ and $\Gamma\seq \psi,\Lambda $, one can derive the sequent $\Gamma\seq \phi\land \psi, \Lambda$ using only $\mathsf{R}\land$. The right-to-left direction expresses the semantic invertibility of $\mathsf{R}\land$ (restricted to classical right contexts): if $\Gamma\seq \phi\land \psi, \Lambda$ is valid, then both $\Gamma\seq \phi,\Lambda $ and $\Gamma\seq \psi,\Lambda $ are valid. Similarly, the second line expresses that one can derive $ \Xi\seq \chi\{\phi_L\vvee \phi_R\},\De$ via $\mathsf{R}\vvee$ using either $\Xi\seq\chi\{\phi_L\},\De$ or $\Xi\seq\chi\{\phi_R\},\De$, and that if $ \Xi\seq \chi\{\phi_L\vvee \phi_R\},\De$ is valid, so is either $\Xi\seq\chi\{\phi_L\},\De$ or $\Xi\seq\chi\{\phi_R\},\De$.

We say that a Boolean combination $\mathcal{S}$ of sequents is \emph{valid} if the Boolean expression given by replacing each valid sequent in $\mathcal{S}$ with $1$ and each invalid sequent with $0$ is valid. For instance, the Boolean combination $S_1 \mathrel{\doublewedge} (S_2 \mathrel{\doublevee} S_3)$ of sequents $S_1,S_2,S_3$ is valid iff $S_1 \mathrel{\doublewedge} (S_2 \mathrel{\doublevee} S_3)$ read as a Boolean expression is valid, namely, the sequent $S_1$ is valid and one of the sequents $S_2$ and $S_3$ is valid. We write $\mathcal{S}_1 \models \mathcal{S}_2$ if the validity of $\mathcal{S}_1$ implies the validity of $\mathcal{S}_2$. We say that a subcombination of $\mathcal{S}$ is a \emph{witnessing combination} for $\mathcal{S}$ if the validity of the subcombination implies the validity of $\mathcal{S}$ (e.g., $\mathcal{S}_1$ is a witnessing combination for $\mathcal{S}_1 \mathrel{\doublevee}\mathcal{S}_2$). We write $\mathcal{S}_1 \vdash_{\scriptsize \mathsf{C}} \mathcal{S}_2$ to denote the fact that there is some witnessing combination for $\mathcal{S}_1$ such that given derivation(s) in $\mathsf{GT}^-$ for the sequents in $\mathcal{S}_1$, one can extend these derivations using rules in $\mathsf{C}$ to construct derivation(s) of a witnessing combination for $\mathcal{S}_2$. We write $\mathcal{S}_1 \leftmodels\vdash_{\scriptsize \mathsf{C}} \mathcal{S}_2$ if $\mathcal{S}_1 \vdash_{\scriptsize \mathsf{C}} \mathcal{S}_2$ and $\mathcal{S}_2 \models \mathcal{S}_1$. We write $\vdash_{\mathsf{C}} \mathcal{S}$ if there are derivations for the sequents in a witnessing combination for $\mathcal{S}$ (from axioms) using rules in $\mathsf{C}$. We write simply $\vdash$ for $\vdash_{\scriptsize \mathsf{GT}^-}$. We say that a derivation $\mathcal{D}$ \emph{witnesses} $\vdash_{\scriptsize \mathsf{C}_1} \mathcal{S}_1\vdash_{\scriptsize \mathsf{C}_2}\mathcal{S}_2\vdash_{\scriptsize \mathsf{C}_3}\ldots\vdash_{\scriptsize \mathsf{C}_n}\mathcal{S}_{n}$ if $\mathcal{D}$ consists of derivations of the sequents in a witnessing combination for $\mathcal{S}_1$ using rules in $\mathsf{C}_1$, extensions of these derivations using rules in $\mathsf{C}_2$ such that the combined derivations constitute derivations of the sequents in a witnessing combination for $\mathcal{S}_2$, and so on ($\mathcal{D}$ \emph{witnessing} $ \mathcal{S}_1\vdash_{\scriptsize \mathsf{C}_2}\mathcal{S}_2\vdash_{\scriptsize \mathsf{C}_3}\ldots\vdash_{\scriptsize \mathsf{C}_n}\mathcal{S}_{n}$ is defined in a similar way).

Our cutfree completeness/countermodel construction theorem below (Theorem \ref{di:theorem:countermodels}) follows from the fact for each rule $R$ of $\mathsf{GT}^-$, there is a derivability/semantic invertibility result of the form $\mathcal{S}_1\leftmodels \vdash_R \mathsf{S}_2$, and if a result for a given rule involves context restrictions, there is always some other such result that can be applied first in order to meet these restrictions. We list the required results of this form below. Note that in the item for $\mathsf{R}\land $, the right contexts of the sequents in the relevant Boolean combinations are classical---this is to allow for a pair of combinations $\mathcal{S}_1$ and $\mathcal{S}_2$ such that $\mathcal{S}_1\leftmodels \vdash_\mathsf{R\land} \mathsf{S}_2$ (similarly for $\mathsf{L}\vee$).
\begin{description}
    \item[($\mathsf{L}\lnot $)] $\Gamma\seq \alpha,\De\leftmodels\vdash_{\scriptsize \mathsf{\mathsf{L}\lnot}} \Gamma,\lnot \alpha\seq \De$
    \item[($\mathsf{R}\lnot$)] $\Gamma,\alpha\seq \De\leftmodels\vdash_{\scriptsize \mathsf{\mathsf{R}\lnot}} \Gamma\seq \lnot\alpha,\De$
    \item[($\mathsf{L}\land $)] $\Gamma,\phi,\psi\seq \De\leftmodels\vdash_{\scriptsize \mathsf{\mathsf{L}\land}} \Gamma,\phi\land \psi\seq \De$
    \item[($\mathsf{R}\land $)] $(\Gamma\seq \phi,\Lambda \mathrel{\doublewedge} \Gamma\seq \psi,\Lambda) \leftmodels\vdash_{\scriptsize \mathsf{\mathsf{R}\land}} \Gamma\seq \phi\land \psi, \Lambda$
    \item[($\mathsf{L}\vee $)] $(\Gamma,\phi\seq \Lambda \mathrel{\doublewedge} \Gamma,\psi\seq \Lambda) \leftmodels\vdash_{\scriptsize \mathsf{\mathsf{L}\vee}} \Gamma,\phi\vee\psi\seq \Lambda$
    \item[($\mathsf{R}\vee$)] $\Gamma\seq\phi,\psi, \De\leftmodels\vdash_{\scriptsize \mathsf{\mathsf{R}\vee}} \Gamma\seq \phi\vee \psi,\De$
    \item[($\mathsf{L}\vvee $)] $(\Gamma,\chi\{\phi_L\}\seq \De \mathrel{\doublewedge} \Gamma,\chi\{\phi_L\}\seq \De) \leftmodels\vdash_{\scriptsize \mathsf{\mathsf{L}\vvee}} \Gamma,\chi\{\phi_L\vvee\phi_R\}\seq \De$
    \item[($\mathsf{R}\vvee$)] $(\Xi\seq\chi\{\phi_L\},\De \mathrel{\doublevee}  \Xi\seq\chi\{\phi_R\},\De)\leftmodels\vdash_{\scriptsize \mathsf{\mathsf{R}\vvee}} \Xi\seq \chi\{\phi_L\vvee \phi_R\},\De$
\end{description}

\begin{theorem}[Cutfree completeness and countermodel construction] \label{di:theorem:countermodels}
    There is an effective procedure that, given a sequent $\Ga\seq \De$, yields a cutfree derivation of $\Ga\seq \De$ if $\Ga\seq \De$ is valid, and yields a countermodel to $\Ga\seq \De$ if $\Ga\seq \De$ is not valid.
\end{theorem}
\begin{proof}
By $\mathsf{L}\vvee$ and by the semantic invertibility of $\mathsf{L}\vvee$, we can find classical $\Xi_1;\ldots;\Xi_n$ such that
$$\bigdoublewedge_{1\leq i \leq n}(\Xi_i\seq \De) \leftmodels\vdash_{\scriptsize \mathsf{L}\vvee}\Gamma \seq \De.$$
By $\mathsf{R}\vvee$ and by the split property (Proposition \ref{di:prop:split}), we can find classical $\Lambda_{11};\ldots;\Lambda_{nm_n}$ such that
$$\bigdoublewedge_{1\leq i \leq n}\bigdoublevee_{1\leq j \leq m_i}(\Xi_i\seq \Lambda_{ij}) \leftmodels\vdash_{\scriptsize \mathsf{R}\vvee}\bigdoublewedge_{1\leq i \leq n}(\Xi_i\seq \De).$$
Finally, by the rules of $\mathsf{G3cp}^-$ together with the semantic invertibility of $\mathsf{G3cp}^-$, we can find atomic sequents $\Xi_{111}\seq \Lambda_{111};\ldots;\Xi_{nm_n q_{nm_n}}\seq \Lambda_{nm_nq_{nm_n}}$ such that 
$$ \bigdoublewedge_{1\leq i \leq n}\bigdoublevee_{1\leq j \leq m_i}\bigdoublewedge_{1 \leq k \leq q_{ij}}(\Xi_{ijk}\seq \Lambda_{ijk})\leftmodels\vdash_{\scriptsize \mathsf{G3cp}^-}\bigdoublewedge_{1\leq i \leq n}\bigdoublevee_{1\leq j \leq m_i}(\Xi_i\seq \Lambda_{ij}).$$
Therefore, 
$$ \bigdoublewedge_{1\leq i \leq n}\bigdoublevee_{1\leq j \leq m_i}\bigdoublewedge_{1 \leq k \leq q_{ij}}(\Xi_{ijk}\seq \Lambda_{ijk})\leftmodels\vdash \Gamma \seq \De.$$
Then if $\Gamma \seq \Delta$ is valid, so is $$\bigdoublewedge_{1\leq i \leq n}\bigdoublevee_{1\leq j \leq m_i}\bigdoublewedge_{1 \leq k \leq q_{ij}}(\Xi_{ijk}\seq \Lambda_{ijk}),$$ whence by by Corollary \ref{di:coro:cutfree_classical_completeness}, $$\vdash \bigdoublewedge_{1\leq i \leq n}\bigdoublevee_{1\leq j \leq m_i}\bigdoublewedge_{1 \leq k \leq q_{ij}}(\Xi_{ijk}\seq \Lambda_{ijk}),$$ and therefore $\vdash \Ga\seq \De$. Conversely, if $\Gamma \seq \Delta$ is invalid, then by the soundness of $\mathsf{GT}^-$ and Corollary \ref{di:coro:cutfree_classical_completeness}, we also have that $$\bigdoublewedge_{1\leq i \leq n}\bigdoublevee_{1\leq j \leq m_i}\bigdoublewedge_{1 \leq k \leq q_{ij}}(\Xi_{ijk}\seq \Lambda_{ijk})$$ is invalid. Lemma \ref{di:lemma:countermodels} then yields a procedure for constructing countermodels on $\mathsf{P}(\Ga)\cup \mathsf{P}(\De)$ to the invalid atomic sequents in this combination, as well as for constructing a countermodel to $\Ga\seq \De$ from these countermodels.
\end{proof}

\begin{example} We write
               \[ 
     \AxiomC{$\Ga_1\seq \De_1$}
     \AxiomC{$\Ga_2\seq \De_2$}
 \RightLabel{{\footnotesize R}}
 \BinaryInfC{$\Ga_3\seq \De_3$}
 \DisplayProof
 \]
 to denote $(\Ga_1\seq \De_1 \mathrel{\doublewedge}\Ga_2\seq \De_2 ) \leftmodels\vdash_{\scriptsize R} \Ga_3\seq \De_3 $ for all rules $R$ other than $\mathsf{R}\vvee$ (and similarly for single-premise rules), and we write
               \[ 
     \AxiomC{$\Xi\seq \chi\{\phi_L\},\De$}
     \AxiomC{$\Xi\seq \chi\{\phi_R\},\De$}
 \RightLabel{{\footnotesize $\mathsf{R}\vvee$}}
 \doubleLine
 \BinaryInfC{$\Xi\seq \chi\{\phi_L\vvee\phi_R\},\De$}
 \DisplayProof
 \]
 to denote $(\Xi\seq \chi\{\phi_L\},\De \mathrel{\doublevee}\Xi\seq \chi\{\phi_R\},\De) \leftmodels\vdash_{\scriptsize \mathsf{L}\vvee} \Xi\seq \chi\{\phi_L\vvee\phi_R\},\De$. Valid sequents are blue; invalid sequents are red; countermodels are written above the sequent arrows. The following demonstrates how the procedure yields a countermodel $\{v_p,v_{\bar{p}}\}$ (where $v_p \models p$ and $v_{\bar{p}}\not \models p$) to the invalid sequent $p \vvee(p\vee\lnot p)\seq p\vvee\lnot p$.
               \[ 
        \AxiomC{{\color{blue} $p\seq p$}}
            \AxiomC{{\color{red} $p,p\overset{\scriptsize \{v_p\}}{\seq} $}}
            \RightLabel{{\footnotesize $\mathsf{R}\lnot $}}
            \UnaryInfC{{\color{red} $p\overset{\scriptsize \{v_p\}}{\seq}\lnot p$}}
        \RightLabel{{\footnotesize $\mathsf{R}\vvee$}}
        \doubleLine
    \BinaryInfC{{{\color{blue} $p\seq p\vvee\lnot p$}}}
        \AxiomC{{\color{blue}$p\seq p$}}
        \AxiomC{{\color{red}$\overset{\scriptsize \{v_{\bar{p}}\}}{\seq} p, p$}}
    \RightLabel{{\footnotesize $\mathsf{L}\lnot$}}
        \UnaryInfC{{\color{red}$\lnot p\overset{\scriptsize \{v_{\bar{p}}\}}{\seq} p$}}
    \RightLabel{{\footnotesize $\mathsf{L}\vee$}}
        \BinaryInfC{{\color{red}$p\vee\lnot p\overset{\scriptsize \{v_{\bar{p}}\}}{\seq} p$}}
        \AxiomC{{\color{red}$p,p\overset{\scriptsize \{v_p\}}{\seq} $}}
        \RightLabel{{\footnotesize $\mathsf{R}\lnot$}}
        \UnaryInfC{{\color{red}$p\overset{\scriptsize \{v_p\}}{\seq} \lnot p$}}
        \AxiomC{{\color{blue}$p\seq p$}}
        \RightLabel{{\footnotesize $\mathsf{L}\lnot$}}
        \UnaryInfC{{\color{blue}$\lnot p,p\seq $}}
        \RightLabel{{\footnotesize $\mathsf{R}\lnot$}}
        \UnaryInfC{{\color{blue}$\lnot p\seq \lnot p$}}
    \RightLabel{{\footnotesize $\mathsf{L}\vee$}}
        \BinaryInfC{{\color{red}$p\vee\lnot p\overset{\scriptsize \{v_p\}}{\seq} \lnot p$}}
    \RightLabel{{\footnotesize $\mathsf{R}\vvee$}}
    \doubleLine
    \BinaryInfC{{\color{red}$p\vee\lnot p\overset{\scriptsize \{v_p,v_{\bar{p}}\}}{\seq} p\vvee\lnot p$}}
 \RightLabel{{\footnotesize $\mathsf{L}\vvee$}}
 \BinaryInfC{{\color{red}$ p \vvee(p\vee\lnot p)\overset{\scriptsize \{v_p,v_{\bar{p}}\}}{\seq} p\vvee\lnot p$}}
 \DisplayProof
 \]   
 Or, in our inline notation,
 
 \begin{centering}
     $(({\color{blue}p\seq p } \mathrel{\doublevee} {\color{red}p,p\seq })\mathrel{\doublewedge}(({\color{blue} p\seq p} \mathrel{\doublewedge} {\color{red} \seq p,p}) \mathrel{\doublevee} ({\color{red}p ,p\seq } \mathrel{\doublewedge} {\color{blue} p\seq p} ))$\\
     $\leftmodels\vdash_{\scriptsize \mathsf{G3cp}^-}$\\
     $(({\color{blue}p\seq p } \mathrel{\doublevee} {\color{red}p\seq \lnot p})\mathrel{\doublewedge}( {\color{red}p\vee \lnot p \seq p} \mathrel{\doublevee} {\color{red}p\vee \lnot p \seq \lnot p}))$\\
$\leftmodels\vdash_{\scriptsize \mathsf{R}\vvee}$\\
$( {\color{blue}p\seq p\vvee\lnot p}\mathrel{\doublewedge} {\color{red}p\vee \lnot p \seq p \vvee\lnot p})$\\
$\leftmodels\vdash_{\scriptsize \mathsf{L}\vvee}$\\
${\color{red}p \vvee(p\vee\lnot p)\seq p\vvee\lnot p}$\\
 \end{centering}
 \end{example}

\section{Partial Resolutions and the Subformula Property} \label{di:section:partial_resolutions}
$\mathsf{GT}^-$ has a weak subformula property. We can formulate this property in terms of what we will call \emph{partial resolutions}---a generalization of resolutions (Section \ref{di:section:resolutions}). Partial resolutions are also a helpful tool for understanding how our system functions: the deep-inference rules transform partial resolutions of a formula of a higher degree (those of the highest degree being the resolutions of the formula) to partial resolutions of a lower degree (the only partial resolution of the lowest degree being the formula itself); in conjunction with the resolution theorem (Theorem \ref{di:theorem:resolution}) and cutfree classical completeness (Corollary \ref{di:coro:cutfree_classical_completeness}), this yields another way of showing the cutfree completeness of $\mathsf{GT}$.

In order to define partial resolutions, it is helpful to first define some more notation for the type of substitutions carried out via the deep inference rules. The reader may wish to examine Example \ref{di:ex:partial_resolutions} as they read the following definitions.

A \emph{$\vvee$-labelled formula (occurrence)} is a formula occurrence $\phi$ in which each occurrence of the disjunction symbol $\vvee$ is assigned a natural number. We indicate the label of each $\vvee$ by writing the corresponding number as a subscript. Examples: (i) $p\vvee_0 (q\vvee_1 r)$ (ii) $(p \vvee_2 x)\vvee_2 (q\vvee_1 r)$. Example (i) is also an example of what we call the \emph{$\vvee$-labelling} of a formula. The $\vvee$-labelling of a formula $\phi$ is denoted $\phi_{\scriptsize\vvee}$, and it is obtained  by labelling each occurrence of $\vvee$ according to its position in the formula (from left to right).

Let $|\phi|_{\scriptsize\vvee}$ denote the number of occurrences of $\vvee$ in $\phi$. We extend the notation for subformula occurrence substitutions to $\vvee$-labelled formulas in the obvious way, and let $$\phi[i_1/j_1]\ldots[i_n/j_n]:=\phi_{\scriptsize\vvee}[\chi_{1i_1}/\chi_{1L}\vvee_{j_1}\chi_{1R}]\ldots[\chi_{ni_n}/\chi_{nL}\vvee_{j_n}\chi_{nR}],$$ where $\{i_x\}_{1 \leq x \leq n}\subseteq\{L,R\}$, $\{j_x\}_{1 \leq x \leq n}\subseteq \{0,\ldots,|\phi|_{\scriptsize\vvee}-1 \}$, and the $j_x$ are distinct. For instance, for $\phi_{\scriptsize\vvee}=p\vvee_0(q\vvee_1 r)$, we have $\phi[L/0]=p$, $\phi[R/0]=q\vvee_1 r$, $\phi[L/1]=p \vvee_0 q$, $\phi[L/0][L/1]=p$, $\phi[R/0][L/1]=q$, and $\phi[L/1][R/0]=q$. 

\begin{definition}[Partial resolutions]
The set $\PPP\RRR_n (\phi)$ of \emph{partial resolutions of $\phi$ of degree $n\in \{0,\ldots,|\phi|_{\scriptsize\vvee}\}$} is defined by
\begin{align*}
    \PPP\RRR_n(\phi):=\{\phi[i_1/j_1]\ldots[i_n/j_n] \mid  \{i_x\}_{1 \leq x \leq n}\subseteq  \{L,R\},\;\{j_x\}_{1 \leq x \leq n}\subseteq \{0,\ldots,|\phi|_{\scriptsize\vvee}-1 \},\; x\neq y \text{ implies }j_x\neq j_y\}.
\end{align*}

The set $\PPP\RRR (\phi)$ of 
\emph{partial resolutions of $\phi$} is defined by $$\PPP\RRR(\phi):=\bigcup_{0\leq n\leq |\phi|_{\tiny\vvee}}\PPP\RRR_n(\phi).$$

\end{definition}

\begin{example} \label{di:ex:partial_resolutions}
For $\phi_{\scriptsize\vvee}=p\vvee_0 (q\vvee_1 r)$,

{\footnotesize
\begin{tikzpicture}[every text node part/.style={align=center},
level 1/.style={sibling distance=40mm},
level 2/.style={sibling distance=20mm}]

\fill[gray,rounded corners,fill opacity=0.1]  (-7.95,-0.6) node[black, above right,opacity=1] {\tiny{$\PPP\RRR_0(\phi)=\{\phi\}$}}   rectangle  (7.95,0.6)  ;

\fill[gray,rounded corners,fill opacity=0.2]  (-7.95,-2.1) node[black,above right,opacity=1] {\tiny{$\PPP\RRR_1(\phi)$}}   rectangle  (7.95,-0.9)  ;

\fill[gray,rounded corners,fill opacity=0.3]  (-7.95,-3.8) node[black,above right,opacity=1] {\tiny{$\PPP\RRR_2(\phi)=\RRR(\phi)$}}   rectangle  (7.95,-2.4)  ;

  \node {$\phi$\\$p\vvee_0 (q\vvee_1 r)$} [edge from parent fork down]
    child {node {$\phi[L/0]$\\$p$} 
      child {node {$\phi[L/0][L/1]$\\$p$} edge from parent node[left,draw=none] {\tiny$1$} }
      child {node {$\phi[L/0][R/1]$\\$p$} edge from parent node[right,draw=none] {\tiny$1$} } edge from parent node[below left,draw=none] {\tiny$0$} }
    child {node {$\phi[R/0]$\\$q\vvee_1 r$}
      child {node {$\phi[R/0][L/1]$\\$q$} edge from parent node[left,draw=none] {\tiny$1$} }
      child {node {$\phi[R/0][R/1]$\\$r$} edge from parent node[right,draw=none] {\tiny$1$} } edge from parent node[below left,draw=none] {\tiny$0$} }
    child {node {$\phi[L/1]$\\$p\vvee_0 q$}
      child {node {$\phi[L/1][L/0]$\\$p$} edge from parent node[left,draw=none] {\tiny$0$} }
      child {node {$\phi[L/1][R/0]$\\$q$} edge from parent node[right,draw=none] {\tiny$0$} } edge from parent node[below right,draw=none] {\tiny$1$} }
    child {node {$\phi[R/1]$\\$p\vvee_0 r$}
      child {node {$\phi[R/1][L/0]$\\$p$} edge from parent node[left,draw=none] {\tiny$0$} }
      child {node {$\phi[R/1][R/0]$\\$r$} edge from parent node[right,draw=none] {\tiny$0$}  } edge from parent node[below right,draw=none] {\tiny$1$}
    };
\end{tikzpicture}
}

It is easy to see that $\PPP\RRR_{|\phi|_{\tiny\vvee}}(\phi)=\RRR(\phi)$; we generalize this in Proposition \ref{di:prop:partial_resolutions_are_resolutions} below.
\end{example}

The guiding intuition behind partial resolutions is that if $\vvee$ represents a question-forming operator as in inquisitive logic, each formula in $\PPP\RRR_n(\phi)$ represents an attempt to resolve $n$ ``questions'' in $\phi$.\footnote{The notion of ``question'' here is such that each occurrence $\vvee$ corresponds to one question. The formula $a\vvee_0 (b\vvee_1 c)$ (which could be used to formalize ``Is Mary in Amsterdam or is she in Beijing or in Copenhagen?'') has two questions. Each of the partial resolutions $a$ (``Mary is in Amsterdam'') and $b\vvee_1 c$ (``Is Mary in Beijing or in Copenhagen?'') resolves question $0$. The first partial resolution $a$, since it itself contains no questions, also constitutes a (full) resolution to $\phi$.}  The formulas in $\RRR(\phi)=\PPP\RRR_{|\phi|_{\tiny\vvee}}(\phi)$, then, resolve all the questions in $\phi$.

We may now formulate the weak subformula property for $\mathsf{GT}^-$. Observe first that in the tree in Example \ref{di:ex:partial_resolutions}, the rule $\mathsf{R}\vvee$ takes a formula labelling a node to that labelling its parent node (that is, an application of $\mathsf{R}\vvee$ may have as its active formula the label of the child node and as its principal formula that of the parent), and the rule $\mathsf{L}\vvee$ takes the labels of the child nodes with the same edge-label ($0$ or $1$) to the label of their parent. We may accordingly reformulate these rules as follows, where we let $\chi[\tau]:=\chi[i_1/j_1]\ldots[i_{n-1}/j_{n-1}]$:\\

\noindent
\begin{tabular}{C{0.50\textwidth} C{0.44\textwidth}}
 \AxiomC{$\Ga,\chi[\tau][L/j_n]\seq \De$}
 \AxiomC{$\Ga,\chi[\tau][R/j_n]\seq \De$}
 \RightLabel{{\footnotesize $\mathsf{L}\vvee$}}
 \BinaryInfC{$\Ga,\chi[\tau]\seq \De$}
 \DisplayProof & 
 \AxiomC{$\Ga\seq \chi[\tau][i_n/j_n],\De$}
 \RightLabel{{\footnotesize $\mathsf{R}\vvee$}}
 \UnaryInfC{$\Ga\seq \chi[\tau],\De$}
 \DisplayProof\\
 \\
\end{tabular}

The following are therefore immediate from the rules of $\mathsf{GT}^-$:

\begin{lemma}\label{di:lemma:deep_inference_rules_with_partial_resolutions}\
    \begin{enumerate}
        \item[(i)] If $\Gamma\seq \chi[\tau][i_n/j_n],\Delta$, then $\Gamma\seq \chi[\tau],\Delta$.
        \item[(ii)] If $\Gamma,\chi[\tau][L/j_n]\seq \Delta$ and $\Gamma,\chi[\tau][R/j_n]\seq \Delta$, then $\Gamma, \chi[\tau]\seq\Delta$.
    \end{enumerate}
\end{lemma}

\sloppy
\begin{proposition}[Partial resolution subformula property]\label{di:prop:subformula_property} For any derivation $\mathcal{D}$ witnessing $\vdash_{\scriptsize \mathsf{GT}^-}\Gamma \seq \De$, each formula occurring in $\mathcal{D}$ is a subformula of a partial resolution of some formula occurring in $\Gamma,\Delta$.
\end{proposition}
\fussy

 There is an instructive alternative way of proving the cutfree completeness of $\mathsf{GT}$ that employs the reformulation above to show that completeness follows from the resolution theorem (Theorem \ref{di:theorem:resolution}) in conjunction with cutfree classical completeness (Corollary \ref{di:coro:cutfree_classical_completeness}). To explain this in more detail, we now also define partial resolutions for multisets.

We assume that each multiset $\Gamma$ (of $\vvee$-labelled formulas) comes with some canonical ordering of its elements and write $\vec{\Gamma}=\langle\phi_0,\ldots,\phi_q \rangle$ for the sequence corresponding to this canonical ordering. Informally, given $\vec{\Gamma}=\langle\phi_0,\ldots,\phi_q \rangle$, we write $\Gamma[i_1/j_1]_{k_1}\ldots[i_n/j_n]_{k_n}$ where $\{k_x\}_{1 \leq x \leq n}\subseteq  \{0,\ldots,q\}$ for the multiset resulting from applying each substitution  $[i_x/j_x]_{k_x}$ such that $k_x=y$ to the formula $\phi_y$ in $\Gamma$. More formally, $\Gamma[i_1/j_1]_{k_1}\ldots[i_n/j_n]_{k_n}$---where $\{i_x\}_{1 \leq x \leq n}\subseteq  \{L,R\}$; $\{k_x\}_{1 \leq x \leq n}\subseteq  \{0,\ldots,q\}$; for each $y\in \{0,\ldots,q\}$, $\{j_x\mid k_x=y\}_{1 \leq x \leq n}\subseteq \{0,\ldots, |\phi_y|_{\scriptsize\vvee}-1\}$; and the $j_x$ such that $k_x=y$ for a given $y\in \{0,\ldots,q\}$ are distinct---denotes the multiset consisting of the elements of the sequence $\langle\phi_0[\tau_0],\ldots ,\ldots,\phi_q[\tau_q] \rangle$, where $[\tau_y]$ is formed by taking the subsequence of $[i_1/j_1]_{k_1}\ldots[i_n/j_n]_{k_n}$ consisting of all elements with $k_x=y$ and removing the $k_x$-indices. E.g., for $\phi=p\vvee_0(q\vvee_1 r)$ and $\psi=p\vvee_0 r$ and $\vec{\Gamma}=\langle\phi,\phi,\psi\rangle$, we have $\Gamma[L/0]_2=\{\phi,\phi,p\}$, $\Gamma[L/1]_1=\{\phi,p\vvee_0 q,\psi\}$, $\Gamma[L/0]_2[L/1]_1=\{\phi,p\vvee_0 q,p\}$, and $\Gamma[L/0]_2[L/1]_1[R/0]_1=\{\phi, q,p\}$. Let $|\Gamma|_{\scriptsize\vvee}:=\sum_{\phi\in \Gamma}|\phi|_{\scriptsize\vvee}$.

\begin{definition}[Partial resolutions for multisets]
The set (of multisets) $\PPP\RRR_n (\Gamma)$ of \emph{partial resolutions of $\Gamma=\{\phi_0,\ldots,\phi_q\}$ of degree $n\in \{0,\ldots,|\Gamma|_{\scriptsize\vvee}\}$} is defined by
\begin{align*}
    \PPP\RRR_n(\Gamma):=\{\Gamma[i_1/j_1]_{k_1}\ldots[i_n/j_n]_{k_n} \mid\; &\{i_x\}_{1 \leq x \leq n}\subseteq  \{L,R\}, \{k_x\}_{1 \leq x \leq n}\subseteq  \{0,\ldots,q\},\\
    &\text{for each $y\in \{0,\ldots, q\}$: }\{j_x\mid k_x=y\}_{1 \leq x \leq n}\subseteq \{0,\ldots,|\phi_y|_{\scriptsize\vvee} -1\},\\
    &[x\neq z\text{ and }k_x=k_z=y] \text{ implies }j_x\neq j_z\}.
\end{align*}

The set (of multisets) $\PPP\RRR (\Gamma)$ of 
\emph{partial resolutions of $\Gamma$} is defined by $$ 
\PPP\RRR(\Gamma):=\bigcup_{0 \leq n\leq |\Gamma|_{\tiny\vvee}}\PPP\RRR_n(\Gamma).$$
\end{definition}

\begin{example} For $\Gamma=\{ \phi_0,\phi_1\}=\{ p\vvee_0 (q\vvee_1 r),s\vvee_0 r\}$, where $\vec{\Gamma}=\langle
\phi_0,\phi_1\rangle=\langle p\vvee_0 (q\vvee_1 r),s\vvee_0 r\rangle$, we have the following (where we do not display the entire partial resolution-tree):

{\footnotesize
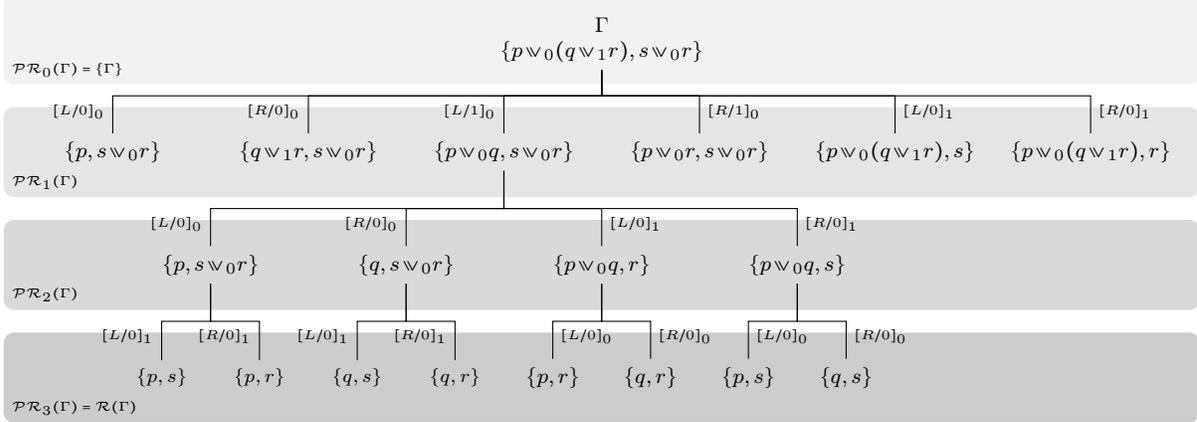
\begin{tikzpicture}[every text node part/.style={align=center},
level 1/.style={sibling distance=26mm},
level 3/.style={sibling distance=13mm}]

\fill[gray,rounded corners,fill opacity=0.1]  (-7.95,-0.6) node[black, above right,opacity=1] {\tiny{$\PPP\RRR_0(\Gamma)=\{\Gamma\}$}}   rectangle  (7.95,0.6)  ;

\fill[gray,rounded corners,fill opacity=0.2]  (-7.95,-2.1) node[black,above right,opacity=1] {\tiny{$\PPP\RRR_1(\Gamma)$}}   rectangle  (7.95,-0.9)  ;

\fill[gray,rounded corners,fill opacity=0.3]  (-7.95,-3.6) node[black,above right,opacity=1] {\tiny{$\PPP\RRR_2(\Gamma)$}}   rectangle  (7.95,-2.4)  ;

\fill[gray,rounded corners,fill opacity=0.4]  (-7.95,-5.1) node[black,above right,opacity=1] {\tiny{$\PPP\RRR_3(\Gamma)=\RRR(\Gamma)$}}   rectangle  (7.95,-3.9)  ;

  \node {$\Gamma$\\$\{ p\vvee_0 (q\vvee_1 r),s\vvee_0 r\}$} [edge from parent fork down]
    child {node (gaL00) {$\{ p,s\vvee_0 r\}$} [sibling distance=12mm]
    edge from parent node[below left,draw=none] {\tiny$[L/0]_0$} 
    }
    child {node (gaR00) {$\{ q\vvee_1 r,s\vvee_0 r\}$} [sibling distance=15mm]
    edge from parent node[below left,draw=none] {\tiny$[R/0]_0$}
    }
    child {node (gaL10) {$\{ p\vvee_0 q,s\vvee_0 r\}$}
        child {node (psr) {$\{p,s\vvee_0 r\}$}
            child {node (ps) {\scriptsize$\{p,s\}$} 
            edge from parent node[below left,draw=none] {\tiny$[L/0]_1$}
            } 
            child {node (pr) {\scriptsize$\{p,r\}$}
            edge from parent node[below left,draw=none] {\tiny$[R/0]_1$}
            } 
        edge from parent node[below left,draw=none] {\tiny$[L/0]_0$}
        } 
        child {node (qsr) {$\{q,s\vvee_0 r\}$}  
            child {node  (qs) {\scriptsize$\{q,s\}$} 
            edge from parent node[below left,draw=none] {\tiny$[L/0]_1$}
            } 
            child {node (qr) {\scriptsize$\{q,r\}$} 
            edge from parent node[below left,draw=none] {\tiny$[R/0]_1$}
            } 
        edge from parent node[below left,draw=none] {\tiny$[R/0]_0$}
        }
        child {node (pqr) {$\{p\vvee_0 q,r\}$} 
            child {node (pr2) {$\{p,r\}$} 
            edge from parent node[below right,draw=none] {\tiny$[L/0]_0$}
            } 
            child {node (qr2) {$\{q,r\}$} 
            edge from parent node[below right,draw=none] {\tiny$[R/0]_0$}
            } 
        edge from parent node[below right,draw=none] {\tiny$[L/0]_1$}
        }
        child {node (pqs) {$\{p\vvee_0 q,s\}$} 
            child {node (ps2) {$\{p,s\}$} 
            edge from parent node[below right,draw=none] {\tiny$[L/0]_0$}
            } 
            child {node (qs2) {$\{q,s\}$} 
            edge from parent node[below right,draw=none] {\tiny$[R/0]_0$}
            } 
        edge from parent node[below right,draw=none] {\tiny$[R/0]_1$}
        }
    edge from parent node[below left,draw=none] {\tiny$[L/1]_0$}
    }
    child {node (gaR10) {$\{ p\vvee_0 r,s\vvee_0 r\}$} [sibling distance=15mm]
    edge from parent node[below right,draw=none] {\tiny$[R/1]_0$}
    }
     child {node (gaL01) {$\{ p\vvee_0 (q\vvee_1 r),s\}$} [sibling distance=15mm]
    edge from parent node[below right,draw=none] {\tiny$[L/0]_1$}
    }
     child {node (gaR01) {$\{ p\vvee_0 (q\vvee_1 r),r\}$} [sibling distance=12mm]
    edge from parent node[below right,draw=none] {\tiny$[R/0]_1$}
    };
\end{tikzpicture}
}

(Note that $\PPP\RRR_3(\Gamma)=\{\{ p,s\}, \{ p,r\},\{ q,s\},\{ q,r\},\{ r,s\}, \{ r,r\}\}=\RRR(\Gamma)$---the part of the tree displayed above does not display all of $\RRR(\Gamma)$.)
\end{example}

One can check that $\PPP\RRR_n(\{\phi\})=\{\{\psi\}\mid \psi\in \PPP\RRR_n(\phi)\}$, so this notion generalizes that for formulas. It is also not difficult to see that we have the following generalization of the fact we observed above (we omit the proof).
\begin{proposition} \label{di:prop:partial_resolutions_are_resolutions}
    $\PPP\RRR_{|\Gamma|_{\scriptsize\vvee}}(\Gamma)=\RRR(\Gamma)$.
\end{proposition}

And we have the following generalization of Lemma \ref{di:lemma:deep_inference_rules_with_partial_resolutions} (and hence of the deep-inference rules): 

\begin{lemma} \label{di:lemma:partial_resolution_derivations}\
    \begin{enumerate}
    \item[(i)] \makeatletter\def\@currentlabel{(i)}\makeatother\label{di:lemma:partial_resolution_derivations_i}For any $n\in \{0,\ldots,|\Delta_{\scriptsize\vvee}|-1\}$, if there is a $\Delta_{n+1}\in \PPP\RRR_{n+1}(\Delta)$ such that $\vdash_{\scriptsize \mathsf{GT}^-}\Gamma\seq \Delta_{n+1}$, then there is a $\Delta_{n}\in \PPP\RRR_{n}(\Delta)$ such that $\vdash_{\scriptsize \mathsf{GT}^-}\Gamma\seq \Delta_{n}$.
    \item[(ii)] \makeatletter\def\@currentlabel{(ii)}\makeatother\label{di:lemma:partial_resolution_derivations_ii}For any $n\in \{0,\ldots,|\Gamma_{\scriptsize\vvee}|-1\}$, if for all $\Gamma_{n+1}\in \PPP\RRR_{n+1}(\Gamma)$ we have $\vdash_{\scriptsize \mathsf{GT}^-}\Gamma_{n+1}\seq \Delta$, then for all $\Gamma_{n}\in \PPP\RRR_{n}(\Gamma)$ we have $\vdash_{\scriptsize \mathsf{GT}^-}\Gamma_n\seq \Delta$.
    \end{enumerate}
\end{lemma}
\begin{proof}
    Follows by the definition of partial resolutions together with Lemma \ref{di:lemma:deep_inference_rules_with_partial_resolutions}.
\end{proof}

\begin{theorem}[Cutfree completeness] \label{di:theorem:cutfree_completeness}
 $ \Gamma \models \bigvee \Delta$ implies $\vdash_{\scriptsize \mathsf{GT}^-} \Gamma \seq \Delta$.
\end{theorem}
\begin{proof}
    By the resolution theorem (Theorem \ref{di:theorem:resolution}), for each $\Xi\in \RRR(\Gamma)$, there is some $\alpha\in \RRR(\bigvee \Delta)$ such that $\Xi \models \alpha$, i.e., there is some $\Lambda\in \RRR(\Delta)$ such that $\bigwedge \Xi \models \bigvee \Lambda$. By cutfree classical completeness (Corollary \ref{di:coro:cutfree_classical_completeness}), also
    $\vdash_{\scriptsize \mathsf{GT}^-}\Xi\seq \Lambda$. So by Proposition \ref{di:prop:partial_resolutions_are_resolutions}, for each $\Xi\in \PPP\RRR_{|\Gamma|_{\tiny\vvee}}(\Gamma)$, there is some $\Lambda \in \PPP\RRR_{|\Delta|_{\tiny\vvee}}( \Delta)$ such that $\vdash_{\scriptsize \mathsf{GT}^-}\Xi\seq \Lambda$. By repeated applications of Lemma \ref{di:lemma:partial_resolution_derivations} \ref{di:lemma:partial_resolution_derivations_i}, for each $\Xi\in \PPP\RRR_{|\Gamma|_{\tiny\vvee}}(\Gamma)$, $\vdash_{\scriptsize \mathsf{GT}^-}\Xi\seq \Delta$. By repeated applications of Lemma \ref{di:lemma:partial_resolution_derivations} \ref{di:lemma:partial_resolution_derivations_ii}, $\vdash_{\scriptsize \mathsf{GT}^-}\Gamma\seq \Delta$. 
\end{proof}

\section{Cut Elimination and Derivation Normal Form} \label{di:section:cut_elimination}
In this section, we prove a normal form result for cutfree derivations and use this normal form result to describe how the cut elimination procedure for $\mathsf{G3cp}$ yields one for $\mathsf{GT}$.

Let us introduce some standard terminology. The \emph{level} of a cut (i.e., an application of $\mathsf{Cut}$) is the sum of the heights of the deductions of its premises. The \emph{rank} of a cut on $\phi$ is $|\phi|$, the number of symbols in $\phi$. The \emph{cutrank} $cr(\mathcal{D})$ of a deduction $\mathcal{D}$ is the maximum of the ranks of the cutformulas occurring in $\mathcal{D}$.

The standard approach to cut elimination fails for our system due to the syntactic restrictions on the rules. Consider, for instance, the following step in a standard-style cut elimination attempt: the rightmost cut of maximal rank and maximal level among cuts of the same rank is such that neither premise is an axiom; the cutformula is not principal
on the left; the final rule applied to get the left premise is $\mathsf{L}\vee$. In the setting of $\mathsf{G3cp}$, this cut would look as follows:
\[ 
 \AxiomC{$\mathcal{D}_1$}
 \noLine
 \UnaryInfC{$\Xi,\beta_1 \seq \alpha,\Lambda$}
 \AxiomC{$\mathcal{D}_2$}
 \noLine
 \UnaryInfC{$\Xi,\beta_2 \seq \alpha,\Lambda$}
 \RightLabel{$\footnotesize\mathsf{L}\vee$}
 \BinaryInfC{$\Xi,\beta_1\vee\beta_2 \seq \alpha,\Lambda$}
 \AxiomC{$\mathcal{D}_3$}
 \noLine
 \UnaryInfC{$\Theta,\alpha\seq\Omega$}
 \RightLabel{$\footnotesize \mathsf{Cut}$}
 \BinaryInfC{$\Xi,\Theta,\beta_1\vee\beta_2 \seq \Lambda,\Omega$}
 \DisplayProof
\]
In the cut elimination procedure for $\mathsf{G3cp}$, we would commute the cut upwards to obtain a derivation with the same cutrank as before ($|\alpha|$), but with either a lower maximal level among cuts of rank $|\alpha|$, or one fewer cut of maximal level of rank $|\alpha|$:
\[
 \AxiomC{$\mathcal{D}_1$}
 \noLine
 \UnaryInfC{$\Xi,\beta_1 \seq \alpha,\Lambda$}
  \AxiomC{$\mathcal{D}_3$}
 \noLine
 \UnaryInfC{$\Theta,\alpha\seq\Omega$}
 \RightLabel{$\footnotesize \mathsf{Cut}$}
 \BinaryInfC{$\Xi,\Theta,\beta_1 \seq \Lambda,\Omega$}
 \AxiomC{$\mathcal{D}_2$}
 \noLine
 \UnaryInfC{$\Xi,\beta_2 \seq \alpha,\Lambda$}
  \AxiomC{$\mathcal{D}_3$}
 \noLine
 \UnaryInfC{$\Theta,\alpha\seq\Omega$}
 \RightLabel{$\footnotesize \mathsf{Cut}$}
 \BinaryInfC{$\Xi,\Theta,\beta_2 \seq \Lambda,\Omega$}
 \RightLabel{$\footnotesize\mathsf{L}\vee$}
 \BinaryInfC{$\Xi,\Theta,\beta_1\vee\beta_2 \seq \Lambda,\Omega$}
 \DisplayProof
\]

In the general setting of $\mathsf{GT}$, the cut would look as follows (in the case in which the cutformula $\alpha$ is not part of the right context set $\Delta$ generated by the implicit weakening):
\[
 \AxiomC{$\mathcal{D}_1$}
 \noLine
 \UnaryInfC{$\Gamma,\psi_1 \seq \alpha,\Lambda$}
 \AxiomC{$\mathcal{D}_2$}
 \noLine
 \UnaryInfC{$\Gamma,\psi_2 \seq \alpha,\Lambda$}
 \RightLabel{$\footnotesize\mathsf{L}\vee$}
 \BinaryInfC{$\Gamma,\psi_1\vee\psi_2 \seq \alpha,\Lambda,\Delta$}
 \AxiomC{$\mathcal{D}_3$}
 \noLine
 \UnaryInfC{$\Pi,\alpha\seq\Sigma$}
 \RightLabel{$\footnotesize \mathsf{Cut}$}
 \BinaryInfC{$\Gamma,\Pi,\psi_1\vee\psi_2 \seq \Lambda,\De,\Sigma$}
 \DisplayProof
\]
We cannot freely commute the cut upwards because $\Sigma$ might not be classical. Attempting to do so would yield the following:
\[
 \AxiomC{$\mathcal{D}_1$}
 \noLine
 \UnaryInfC{$\Gamma,\psi_1 \seq \alpha,\Lambda$}
  \AxiomC{$\mathcal{D}_3$}
 \noLine
 \UnaryInfC{$\Pi,\alpha\seq\Sigma$}
 \RightLabel{$\footnotesize \mathsf{Cut}$}
 \BinaryInfC{$\Gamma,\Pi,\psi_1 \seq \Lambda,\Sigma$}
 \AxiomC{$\mathcal{D}_2$}
 \noLine
 \UnaryInfC{$\Gamma,\psi_2 \seq \alpha,\Lambda$}
  \AxiomC{$\mathcal{D}_3$}
 \noLine
 \UnaryInfC{$\Pi,\alpha\seq\Sigma$}
 \RightLabel{$\footnotesize \mathsf{Cut}$}
 \BinaryInfC{$\Gamma,\Pi,\psi_2 \seq \Lambda,\Sigma$}
 \RightLabel{$\footnotesize\#\mathsf{L}\vee$}
 \BinaryInfC{$\Gamma,\Pi,\psi_1\vee\psi_2 \seq \Lambda,\De,\Sigma$}
 \DisplayProof
\]
Here the application of $\mathsf{L}\vee$ is not legitimate because the right context set $\Lambda,\Sigma$ of the premises might not be classical.

We show instead that each cut can be transformed into a cut on classical sequents. This follows from a derivation normal form theorem for the cutfree system: any cutfree derivation of $\Ga\seq \De$ can be transformed into a cutfree derivation in which one first derives classical sequents whose antecedents are the resolutions of $\Ga$ and whose succedents are resolutions of $\De$, and then applies the deep-inference rules to derive $\Ga\seq\De$. (Note the similarity between this theorem and \emph{decomposition theorems} in the literature on the calculus of structures \cite{guglielmi,guglielmistrassburger2001,brunnlerthesis}.) Before stating the theorem, we give an example.

\begin{example}
Theorem \ref{di:theorem:normal_form} transforms the derivation
\[
\AxiomC{$\mathcal{D}_1$}
\noLine
\UnaryInfC{$x , \lnot x \vee (\lnot q \vee p ),q\seq  p$}
 \RightLabel{\footnotesize $\mathsf{R}\vvee$}
\UnaryInfC{$x , \lnot x \vee (\lnot q \vee p ),q\seq  p\vvee r$}
\AxiomC{$\mathcal{D}_2$}
\noLine
\UnaryInfC{$x , \lnot x \vee (\lnot q \vee r),q\seq   r$}
 \RightLabel{\footnotesize $\mathsf{R}\vvee$}
\UnaryInfC{$x , \lnot x \vee (\lnot q \vee r),q\seq  p \vvee r$}
 \RightLabel{\footnotesize $\mathsf{L}\vvee$}
\BinaryInfC{$x , \lnot x \vee (\lnot q \vee(p\vvee r)),q\seq p\vvee r $}
 \RightLabel{\footnotesize $\mathsf{R}\lnot$}
\UnaryInfC{$x , \lnot x \vee (\lnot q \vee(p\vvee r))\seq p\vvee r, \lnot q $}
 \RightLabel{\footnotesize $\mathsf{R}\vee$}
\UnaryInfC{$x , \lnot x \vee (\lnot q \vee(p\vvee r))\seq (p\vvee r)\vee \lnot q $}
 \RightLabel{\footnotesize $\mathsf{L}\land$}
\UnaryInfC{$x \land( \lnot x \vee( \lnot q\vee (p\vvee r)))\seq (p\vvee r)\vee \lnot q$}
 \DisplayProof
\]
into the derivation
\[
\AxiomC{$\mathcal{D}_1$}
\noLine
\UnaryInfC{$x , \lnot x \vee (\lnot q \vee p ),q\seq  p$}
 \RightLabel{\footnotesize $\mathsf{R}\lnot$}
\UnaryInfC{$x , \lnot x \vee (\lnot q \vee p )\seq  p,\lnot q$}
 \RightLabel{\footnotesize $\mathsf{R}\vee$}
\UnaryInfC{$x , \lnot x \vee (\lnot q \vee p )\seq  p\vee\lnot q$}
 \RightLabel{\footnotesize $\mathsf{L}\land $}
\UnaryInfC{$x \land (\lnot x \vee (\lnot q \vee p ))\seq  p\vee\lnot q$}
 \RightLabel{\footnotesize $\mathsf{R}\vvee$}
\UnaryInfC{$x \land (\lnot x \vee (\lnot q \vee p ))\seq  (p\vvee r)\vee\lnot q$}
\AxiomC{$\mathcal{D}_2$}
\noLine
\UnaryInfC{$x , \lnot x \vee (\lnot q \vee r ),q\seq  r$}
 \RightLabel{\footnotesize $\mathsf{R}\lnot$}
\UnaryInfC{$x , \lnot x \vee (\lnot q \vee r )\seq  r,\lnot q$}
 \RightLabel{\footnotesize $\mathsf{R}\vee$}
\UnaryInfC{$x , \lnot x \vee (\lnot q \vee r )\seq  r\vee\lnot q$}
 \RightLabel{\footnotesize $\mathsf{L}\land $}
\UnaryInfC{$x \land (\lnot x \vee (\lnot q \vee r ))\seq  r\vee\lnot q$}
 \RightLabel{\footnotesize $\mathsf{R}\vvee$}
\UnaryInfC{$x \land (\lnot x \vee (\lnot q \vee r ))\seq  (p\vvee r)\vee\lnot q$}
 \RightLabel{\footnotesize $\mathsf{L}\vvee$}
\BinaryInfC{$x \land (\lnot x \vee (\lnot q \vee (p\vvee r) ))\seq  (p\vvee r)\vee\lnot q $}
 \DisplayProof
\]
where notice that $\mathcal{R}(x \land (\lnot x \vee (\lnot q \vee (p\vvee r) )))=\{x \land (\lnot x \vee (\lnot q \vee p )),x \land (\lnot x \vee (\lnot q \vee r ))\}$ and that $\mathcal{R}((p\vvee r)\vee\lnot q)=\{p\vee\lnot q, r\vee\lnot q\}$.
\end{example}

\begin{theorem}[Derivation normal form]\label{di:theorem:normal_form} There is an effective procedure (that is height-preserving) transforming any derivation witnessing $\vdash_{\scriptsize \mathsf{GT}^-}\Ga\seq \De$ into a derivation witnessing
\begin{align*}
&\vdash_{\scriptsize {\mathsf{G3cp}}^-}\bigdoublewedge_{\Xi\in \RRR(\Gamma)}(\Xi\seq f[\Xi])\vdash_{\scriptsize \mathsf{R}\vvee}\bigdoublewedge_{\Xi\in \RRR(\Gamma)}(\Xi\seq \De)\vdash_{\scriptsize \mathsf{L}\vvee}\Ga\seq \De,
\end{align*}
where $f:\RRR(\Ga)\to \RRR(\De)$. We say that a derivation of $\Ga\seq \De$ of the form above is in \emph{normal form}.
\end{theorem}
\begin{proof}
It suffices to show that applications of $\mathsf{R}\vvee$ can be commuted below applications of the logical rules of $\mathsf{G3cp}$, and that applications $\mathsf{L}\vvee$ can be commuted below applications of all other rules (in a way that preserves height); it is then easy to see that the rest follows by the definition of resolutions. We show this by induction on the height $n$ of derivations. If $n=1$, the derivation is already in normal form. We now assume the result holds for $n$ and prove that it holds for $n+1$. We only explicitly show some cases; the rest of the cases are similar. Note that in most of the cases we do not need to make use of the induction hypothesis. In most cases below we write only the relevant rule applications, and omit the rest of the derivation.

We show that the $\vvee$-rules can be commuted below $\mathsf{R}\land$.

$\mathsf{R}\vvee$: The principal formula of $\mathsf{R}\vvee$ cannot be a side formula of $\mathsf{R}\land$ since the former must be nonclassical and the latter classical. It is therefore an active formula of $\mathsf{R}\land$, and we may assume that the relevant part of $\mathcal{D}$ is of the form:
\[ 
     \AxiomC{$\Ga\seq \chi\{\phi_i\},\Lambda$}
 \RightLabel{{\footnotesize $\mathsf{R}\vvee$}}
 \UnaryInfC{$\Ga\seq\chi\{\phi_L\vvee\phi_R\},\Lambda$}
     \AxiomC{$\Ga\seq \psi,\Lambda$}
 \RightLabel{{\footnotesize $\mathsf{R}\land$}}
 \BinaryInfC{$\Ga\seq \chi\{\phi_L\vvee\phi_R\}\land \psi,\Lambda,\De$}
 \DisplayProof
\]
This is transformed into:
\[ 
     \AxiomC{$\Ga\seq \chi\{\phi_i\},\Lambda$}
     \AxiomC{$\Ga\seq \psi,\Lambda$}
 \RightLabel{{\footnotesize $\mathsf{R}\land$}}
 \BinaryInfC{$\Ga\seq \chi\{\phi_i\}\land \psi,\Lambda,\De$}
 \RightLabel{{\footnotesize $\mathsf{R}\vvee$}}
 \UnaryInfC{$\Ga\seq\chi\{\phi_L\vvee\phi_R\}\land \psi,\Lambda,\De$}
 \DisplayProof
\]

$\mathsf{L}\vvee$: The relevant part of $\mathcal{D}$ is of the form:
\[
     \AxiomC{$\mathcal{D}_1$}
     \noLine
     \UnaryInfC{$\Ga,\chi\{\phi_L\}\seq \phi,\Lambda$}
     \AxiomC{$\mathcal{D}_2$}
     \noLine
     \UnaryInfC{$\Ga,\chi\{\phi_R\}\seq \phi,\Lambda$}
 \RightLabel{{\footnotesize $\mathsf{L}\vvee$}}
 \BinaryInfC{$\Ga,\chi\{\phi_L\vvee\phi_R\}\seq \phi,\Lambda$}
     \AxiomC{$\mathcal{D}_3$}
     \noLine
     \UnaryInfC{$\Ga,\chi\{\phi_L\vvee\phi_R\}\seq \psi,\Lambda$}
 \RightLabel{{\footnotesize $\mathsf{R}\land$}}
 \BinaryInfC{$\Ga,\chi\{\phi_L\vvee\phi_R\}\seq \phi\land \psi,\Lambda,\De$}
 \DisplayProof
\]
By the induction hypothesis, we may assume that either $\mathcal{D}_3$ consists of a single axiom; or the final rule applied in $\mathcal{D}_3$ is $\mathsf{L}\vvee$, the principal formula of the final rule application is $\chi\{\phi_L\vvee\phi_R\}$, and the principal subformula is $\phi_L\vvee\phi_R$. In the former case, $\Ga,\chi\{\phi_L\}\seq \psi,\Lambda$ and $\Ga,\chi\{\phi_R\}\seq \psi,\Lambda$ are also axioms, and the above is transformed into:
\[
\AxiomC{$\mathcal{D}_1$}
     \noLine
     \UnaryInfC{$\Ga,\chi\{\phi_L\}\seq \phi,\Lambda$}
     \AxiomC{$\Ga,\chi\{\phi_L\}\seq \psi,\Lambda$}
 \RightLabel{{\footnotesize $\mathsf{R}\land$}}
 \BinaryInfC{$\Ga,\chi\{\phi_L\}\seq \phi \land\psi,\Lambda,\De$}
     \AxiomC{$\mathcal{D}_2$}
     \noLine
     \UnaryInfC{$\Ga,\chi\{\phi_R\}\seq \phi,\Lambda$}
     \AxiomC{$\Ga,\chi\{\phi_R\}\seq \psi,\Lambda$}
 \RightLabel{{\footnotesize $\mathsf{R}\land$}}
 \BinaryInfC{$\Ga,\chi\{\phi_R\} \seq \phi\land \psi,\Lambda,\De$}
 \RightLabel{{\footnotesize $\mathsf{L}\vvee$}}
 \BinaryInfC{$\Ga,\chi\{\phi_L\vvee\phi_R\}\seq \phi\land \psi,\Lambda,\De$}
 \DisplayProof
\]
In the latter case, we have:
\[
\AxiomC{$\mathcal{D}_1$}
     \noLine
     \UnaryInfC{$\Ga,\chi\{\phi_L\}\seq \phi,\Lambda$}
     \AxiomC{$\mathcal{D}_2$}
     \noLine
     \UnaryInfC{$\Ga,\chi\{\phi_R\}\seq \phi,\Lambda$}
 \RightLabel{{\footnotesize $\mathsf{L}\vvee$}}
 \BinaryInfC{$\Ga,\chi\{\phi_L\vvee\phi_R\}\seq \phi,\Lambda$}
     \AxiomC{$\mathcal{D}_4$}
     \noLine
     \UnaryInfC{$\Ga,\chi\{\phi_L\}\seq \psi,\Lambda$}
     \AxiomC{$\mathcal{D}_5$}
     \noLine
      \UnaryInfC{$\Ga,\chi\{\phi_R\}\seq \psi,\Lambda$}
 \RightLabel{{\footnotesize $\mathsf{L}\vvee$}}
 \BinaryInfC{$\Ga,\chi\{\phi_L\vvee\phi_R\}\seq \psi,\Lambda$}
 \RightLabel{{\footnotesize $\mathsf{R}\land$}}
 \BinaryInfC{$\Ga,\chi\{\phi_L\vvee\phi_R\}\seq \phi\land \psi,\Lambda,\De$}
 \DisplayProof
\]
This is transformed into:
\[
\AxiomC{$\mathcal{D}_1$}
     \noLine
     \UnaryInfC{$\Ga,\chi\{\phi_L\}\seq \phi,\Lambda$}
     \AxiomC{$\mathcal{D}_4$}
     \noLine
     \UnaryInfC{$\Ga,\chi\{\phi_L\}\seq \psi,\Lambda$}
 \RightLabel{{\footnotesize $\mathsf{R}\land$}}
 \BinaryInfC{$\Ga,\chi\{\phi_L\}\seq \phi \land\psi,\Lambda,\De$}
     \AxiomC{$\mathcal{D}_2$}
     \noLine
     \UnaryInfC{$\Ga,\chi\{\phi_R\}\seq \phi,\Lambda$}
     \AxiomC{$\mathcal{D}_5$}
     \noLine
      \UnaryInfC{$\Ga,\chi\{\phi_R\}\seq \psi,\Lambda$}
 \RightLabel{{\footnotesize $\mathsf{R}\land$}}
 \BinaryInfC{$\Ga,\chi\{\phi_R\} \seq \phi\land \psi,\Lambda,\De$}
 \RightLabel{{\footnotesize $\mathsf{L}\vvee$}}
 \BinaryInfC{$\Ga,\chi\{\phi_L\vvee\phi_R\}\seq \phi\land \psi,\Lambda,\De$}
 \DisplayProof
\]
We show that $\mathsf{L}\vvee$ can be commuted below $\mathsf{R}\vvee$. The relevant part of $\mathcal{D}$ is of the form:
    \[ 
     \AxiomC{$\Ga,\chi\{\phi_L\}\seq \eta\{\psi_i\},\De$}
     \AxiomC{$\Ga,\chi\{\phi_R\}\seq  \eta\{\psi_i\},\De$}
 \RightLabel{{\footnotesize $\mathsf{L}\vvee$}}
 \BinaryInfC{$\Ga,\chi\{\phi_L\vvee\phi_R\}\seq \eta\{\psi_i\},\De$}
 \RightLabel{{\footnotesize $\mathsf{R}\vvee$}}
 \UnaryInfC{$\Ga ,\chi\{\phi_L\vvee\phi_R\}\seq\eta\{\psi_L\vvee\psi_R\},\De$}
 \DisplayProof
\]
This is transformed into:
    \[ 
     \AxiomC{$\Ga,\chi\{\phi_L\}\seq \eta\{\psi_i\},\De$}
 \RightLabel{{\footnotesize $\mathsf{R}\vvee$}}
 \UnaryInfC{$\Ga ,\chi\{\phi_L\}\seq\eta\{\psi_L\vvee\psi_R\},\De$}
     \AxiomC{$\Ga,\chi\{\phi_R\}\seq  \eta\{\psi_i\},\De$}
 \RightLabel{{\footnotesize $\mathsf{R}\vvee$}}
 \UnaryInfC{$\Ga ,\chi\{\phi_R\}\seq\eta\{\psi_L\vvee\psi_R\},\De$}
 \RightLabel{{\footnotesize $\mathsf{L}\vvee$}}
 \BinaryInfC{$\Ga ,\chi\{\phi_L\vvee\phi_R\}\seq\eta\{\psi_L\vvee\psi_R\},\De$}
 \DisplayProof
\]

We show that applications of $\mathsf{L}\vvee$ can be commuted around one another.

Case 1: The principal formula of the first application is a side formula of the second application. We may assume that the relevant part of $\mathcal{D}$ is of the form:
    \[
     \AxiomC{$\Ga,\chi\{\phi_L\},\eta\{\psi_L\}\seq\De$}
     \AxiomC{$\Ga,\chi\{\phi_R\},\eta\{\psi_L\}\seq\De$}
 \RightLabel{{\footnotesize $\mathsf{L}\vvee$}}
 \BinaryInfC{$\Ga,\chi\{\phi_L\vvee\phi_R\},\eta\{\psi_L\}\seq \De$}
 \AxiomC{$\Ga,\chi\{\phi_L\vvee\phi_R\},\eta\{\psi_R\}\seq \De$}
 \RightLabel{{\footnotesize $\mathsf{L}\vvee$}}
 \BinaryInfC{$\Ga ,\chi\{\phi_L\vvee\phi_R\},\eta\{\psi_L\vvee\psi_R\}\seq\De$}
 \DisplayProof
\]
This case is clearly analogous to the $\mathsf{L}\vvee$-subcase of $\mathsf{R}\land$.

Case 2: The principal formula $\chi$ of the first application is an active formula of the second application (we assume without loss of generality that $\chi$ contains $\psi_L$, where $\psi_L\vvee \psi_R$ is the active subformula of the second application). There are three subcases.

Case 2.1: $\chi=\chi\{\phi_L\vvee\phi_R\}\{\psi_L\}$ (see Section \ref{di:section:inversion_etc} for the definition of this notation). The relevant part of $\mathcal{D}$ is of the form:
    \[
     \AxiomC{$\Ga,\chi\{\phi_L\}\{\psi_L\}\seq\De$}
     \AxiomC{$\Ga,\chi\{\phi_R\}\{\psi_L\}\seq\De$}
 \RightLabel{{\footnotesize $\mathsf{L}\vvee$}}
 \BinaryInfC{$\Ga,\chi\{\phi_L\vvee\phi_R\}\{\psi_L\}\seq \De$}
 \RightLabel{{\footnotesize $\mathsf{L}\vvee$}}
 \AxiomC{$\Ga,\chi\{\phi_L\vvee\phi_R\}\{\psi_R\}\seq \De$}
 \BinaryInfC{$\Ga ,\chi\{\phi_L\vvee\phi_R\}\{\psi_L\vvee\psi_R\}\seq\De$}
 \DisplayProof
\] 
This case is clearly analogous to the $\mathsf{L}\vvee$-subcase of $\mathsf{R}\land$.

Case 2.2: $\chi=\chi\{(\phi_L\vvee\phi_R)\{\psi_L\}\}$ (we assume without loss of generality that $\chi=\chi\{\phi_L\{\psi_L\}\vvee\phi_R\}$). The relevant part of $\mathcal{D}$ is of the form: 
   \[
     \AxiomC{$\Ga,\chi\{\phi_L\{\psi_L\}\}\seq\De$}
     \AxiomC{$\Ga,\chi\{\phi_R\}\seq\De$}
 \RightLabel{{\footnotesize $\mathsf{L}\vvee$}}
 \BinaryInfC{$\Ga,\chi\{\phi_L\{\psi_L\}\vvee\phi_R\}\seq \De$}
 \AxiomC{$\Ga,\chi\{\phi_L\{\psi_R\}\vvee\phi_R\}\seq \De$}
 \RightLabel{{\footnotesize $\mathsf{L}\vvee$}}
 \BinaryInfC{$\Ga ,\chi\{\phi_L\{\psi_L\vvee\psi_R\}\vvee\phi_R\}\seq\De$}
 \DisplayProof
\]
This case is similar to the $\mathsf{L}\vvee$-subcase of $\mathsf{R}\land$.

Case 2.3: $\chi=\chi\{\psi_L\{\phi_L\vvee\phi_R\}\}$. The relevant part of $\mathcal{D}$ is of the form: 
   \[
     \AxiomC{$\Ga,\chi\{\psi_L\{\phi_L\}\}\seq\De$}
     \AxiomC{$\Ga,\chi\{\psi_L\{\phi_R\}\}\seq\De$}
 \RightLabel{{\footnotesize $\mathsf{L}\vvee$}}
 \BinaryInfC{$\Ga,\chi\{\psi_L\{\phi_L\vvee\phi_R\}\}\seq \De$}
 \RightLabel{{\footnotesize $\mathsf{L}\vvee$}}
 \AxiomC{$\Ga,\chi\{\psi_R\}\seq \De$}
 \BinaryInfC{$\Ga ,\chi\{\psi_L\{\phi_L\vvee \phi_R\}\vvee\psi_R\}\seq\De$}
 \DisplayProof
\]
This is transformed into:
   \[
     \AxiomC{$\Ga,\chi\{\psi_L\{\phi_L\}\}\seq\De$}
 \AxiomC{$\Ga,\chi\{\psi_R\}\seq \De$}
 \RightLabel{{\footnotesize $\mathsf{L}\vvee$}}
 \BinaryInfC{$\Ga,\chi\{\psi_L\{\phi_L\}\vvee\psi_R\}\seq \De$}
     \AxiomC{$\Ga,\chi\{\psi_L\{\phi_R\}\}\seq\De$}
 \AxiomC{$\Ga,\chi\{\psi_R\}\seq \De$}
 \RightLabel{{\footnotesize $\mathsf{L}\vvee$}}
 \BinaryInfC{$\Ga,\chi\{\psi_L\{\phi_R\}\vvee\psi_R\}\seq \De$}
 \RightLabel{{\footnotesize $\mathsf{L}\vvee$}}
 \BinaryInfC{$\Ga ,\chi\{\psi_L\{\phi_L\vvee \phi_R\}\vvee\psi_R\}\seq\De$}
 \DisplayProof
\]
\end{proof}

    Observe that derivations in normal form have a stronger subformula property than that of cutfree derivations in general (Proposition \ref{di:prop:subformula_property}):

    \begin{proposition}[Normal form partial resolution subformula property] For any $\mathcal{D}$ in normal form witnessing $\vdash_{\scriptsize \mathsf{GT}^-}\Gamma \seq \De$, each formula occurring in $\mathcal{D}$ is a partial resolution or a subformula of a resolution of some formula occurring in $\Gamma,\Delta$.
\end{proposition}

We are now ready to prove our cut elimination theorem.

\begin{theorem}[Cut elimination] \label{di:theorem:cut_elimination}
 If $\mathcal{D}$ witnesses $\vdash_{\scriptsize \mathsf{GT}} \Ga \seq \De$, there is an effective procedure for transforming $\mathcal{D}$ into a derivation $\mathcal{D}'$ witnessing $\vdash_{\scriptsize \mathsf{GT}^-} \Ga \seq \De$.
\end{theorem}
\begin{proof}
       By induction on the cutrank, with a subinduction on the level of cuts. It suffices to show that we can transform a derivation $\mathcal{D}$ ending in a cut
    \[ 
 \AxiomC{$\mathcal{D}_1$}
 \noLine
 \UnaryInfC{$\Ga \seq \phi,\De$}
 \AxiomC{$\mathcal{D}_2$}
 \noLine
 \UnaryInfC{$\Pi,\phi \seq \Sig$}
 \RightLabel{\footnotesize $\mathsf{Cut}$}
 \BinaryInfC{$\Ga,\Pi\seq \De,\Sig$}
 \DisplayProof
\]
where $\mathcal{D}_1$ and $\mathcal{D}_2$ are cutfree into a cutfree derivation. By Theorem \ref{di:theorem:normal_form}, we can replace $\mathcal{D}_1$, $\mathcal{D}_2$ with cutfree $\mathcal{D}'_1$, $\mathcal{D}'_2$ in normal form; i.e., $\mathcal{D}'_1$ witnesses
\begin{align*}
    &\vdash_{\scriptsize \mathsf{G3cp}^-}\bigdoublewedge_{\Xi\in \RRR(\Gamma)}(\Xi\seq f(\Xi))\vdash_{\scriptsize \mathsf{R}\vvee}\bigdoublewedge_{\Xi\in \RRR(\Gamma)}(\Xi\seq \phi,\De)\vdash_{\scriptsize \mathsf{L}\vvee}\Ga\seq \phi,\De,
\end{align*} where $f:\RRR(\Ga)\to \RRR(\phi,\De)$, and $\mathcal{D}'_2$ witnesses\footnote{Below we write $[\Theta,\alpha]\in \RRR(\Pi,\phi)$ to indicate that $\Theta\in \RRR(\Pi)$ and $\alpha\in \RRR(\phi)$, whence the multiset $\Theta,\alpha$ is in $\RRR(\Pi,\phi)$; and similarly for $[\Xi,\Theta]\in \RRR(\Ga,\Pi)$.}
\begin{align*}
   &\vdash_{\scriptsize \mathsf{G3cp}^-}\bigdoublewedge_{[\Theta,\alpha]\in \RRR(\Pi,\phi)}(\Theta,\alpha \seq g(\Theta,\alpha))\vdash_{\scriptsize \mathsf{R}\vvee}\bigdoublewedge_{[\Theta,\alpha]\in \RRR(\Pi,\phi)}(\Theta,\alpha\seq \Sigma)\vdash_{\scriptsize \mathsf{L}\vvee}\Pi,\phi\seq \Sigma, 
\end{align*}
where $g:\RRR(\Pi,\phi)\to \RRR(\Sigma)$, each $\Theta\in \RRR(\Pi)$, and each $\alpha\in \RRR(\phi)$.

Write $f(\Xi)=\alpha^{\Xi},\Lambda^{\Xi}\in \RRR(\phi,\De)$. For each $[\Xi,\Theta]\in \RRR(\Ga,\Pi)$, we can find a subderivation $\mathcal{D_1}^{\Xi}$ of $\mathcal{D}'_1$ witnessing $\vdash_{\scriptsize \mathsf{G3cp}^-}\Xi\seq \Lambda^{\Xi},\alpha^{\Xi}$ and a subderivation $\mathcal{D}^{\Xi,\Theta}$ of $\mathcal{D}'_2$ witnessing $\vdash_{\scriptsize \mathsf{G3cp}^-}\Theta,\alpha^{\Xi}\seq g(\Theta,\alpha^{\Xi})$. We can then construct the following $\mathsf{G3cp}$-derivation:
    \[ 
 \AxiomC{$\mathcal{D}^{\Xi}$}
 \noLine
 \UnaryInfC{$\Xi\seq \Lambda^{\Xi},\alpha^{\Xi}$}
 \AxiomC{$\mathcal{D}^{\Xi,\Theta}$}
 \noLine
 \UnaryInfC{$\Theta,\alpha^{\Xi}\seq g(\Theta,\alpha^{\Xi})$}
 \RightLabel{\footnotesize $\mathsf{Cut}$}
 \BinaryInfC{$\Xi,\Theta,\seq \Lambda^{\Xi},g(\Theta,\alpha^{\Xi})$}
 \DisplayProof
\]
By classical cut elimination (Theorem \ref{di:theorem:classical_cut_elimination}), there is an effective procedure transforming this derivation into a derivation $\mathcal{D}_{\Xi,\Theta}$ witnessing $\vdash_{\scriptsize \mathsf{G3cp}^-}\Xi,\Theta,\seq \Lambda^{\Xi},g(\Theta,\alpha^{\Xi})$. 

Combining the derivations $\mathcal{D}_{\Xi,\Theta}$ with subderivations derived from the subderivations of $\mathcal{D}'_1$ and $\mathcal{D}_2'$ which apply $\mathsf{R}\vvee$ to partial resolutions of $\De,\Sigma$ (changing the contexts as appropriate), we can construct derivations witnessing
\begin{align*}
    &\vdash_{\scriptsize \mathsf{G3cp}^-}\bigdoublewedge_{[\Xi,\Theta]\in \RRR(\Ga,\Pi)}(\Xi,\Theta,\seq \Lambda^{\Xi},g(\Theta,\alpha^{\Xi}))\vdash_{\scriptsize \mathsf{R}\vvee}\bigdoublewedge_{[\Xi,\Theta]\in \RRR(\Ga,\Pi)}(\Xi,\Theta,\seq \De,\Sigma)
\end{align*}
Finally, combining these derivations with subderivations derived from the subderivations of $\mathcal{D}'_1$ and $\mathcal{D}_2'$  which apply $\mathsf{L}\vvee$ to partial resolutions of $\Ga,\Pi$ (changing the contexts as appropriate), we can construct a derivation witnessing
\reqnomode
\begin{align*}
    &\vdash_{\scriptsize \mathsf{G3cp}^-}\bigdoublewedge_{[\Xi,\Theta]\in \RRR(\Ga,\Pi)}(\Xi,\Theta,\seq \Lambda^{\Xi},g(\Theta,\alpha^{\Xi}))\vdash_{\scriptsize \mathsf{R}\vvee}\bigdoublewedge_{[\Xi,\Theta]\in \RRR(\Ga,\Pi)}(\Xi,\Theta,\seq \De,\Sigma)\vdash_{\scriptsize \mathsf{L}\vvee}\Ga,\Pi\seq \De,\Sigma 
\end{align*}
\end{proof}

\leqnomode

\begin{example}
We wish to eliminate the following cut from a derivation $\mathcal{D}$, where $\mathcal{D}_1$ and $\mathcal{D}_2$ are cutfree:
    \[
 \AxiomC{$\mathcal{D}_1$}
 \noLine
 \UnaryInfC{$a\vee (p\vvee q)\seq p\vvee q,a$}
 \AxiomC{$\mathcal{D}_2$}
 \noLine
 \UnaryInfC{$b,p\vvee q\seq (b\land p)\vvee (b\land q)$}
 \RightLabel{\footnotesize $\mathsf{Cut}$}
 \BinaryInfC{$a\vee (p\vvee q),b\seq a,(b\land p)\vvee (b\land q)$}
 \DisplayProof
\]
The normal form theorem transforms the derivations $\mathcal{D}_1$ and $\mathcal{D}_2$ into derivations in normal form $\mathcal{D}'_1$ and $\mathcal{D}'_2$:
    \[
\AxiomC{$\mathcal{D}^{a\vee p}$}
\noLine
\UnaryInfC{$a\vee p\seq p,a$}
 \RightLabel{\footnotesize $\mathsf{R}\vvee$}
\UnaryInfC{$a\vee p\seq p\vvee q,a$}
\AxiomC{$\mathcal{D}^{a\vee q}$}
\noLine
\UnaryInfC{$a\vee q\seq q,a$}
 \RightLabel{\footnotesize $\mathsf{R}\vvee$}
\UnaryInfC{$a\vee q\seq p\vvee q,a$}
 \RightLabel{\footnotesize $\mathsf{L}\vvee$}
\BinaryInfC{$a\vee (p\vvee q)\seq p\vvee q,a$}
 \DisplayProof \hspace{0.5cm}
\AxiomC{$\mathcal{D}^{a\vee p,b}$}
\noLine
\UnaryInfC{$b, p\seq b\land p$}
 \RightLabel{\footnotesize $\mathsf{R}\vvee$}
\UnaryInfC{$b, p\seq (b\land p)\vvee (b\land q)$}
\AxiomC{$\mathcal{D}^{a\vee q,b}$}
\noLine
\UnaryInfC{$b, q\seq b\land q$}
 \RightLabel{\footnotesize $\mathsf{R}\vvee$}
\UnaryInfC{$b, q\seq (b\land p)\vvee (b\land q)$}
 \RightLabel{\footnotesize $\mathsf{L}\vvee$}
\BinaryInfC{$b,(p\vvee q)\seq (b\land p)\vvee (b\land q)$}
 \DisplayProof
\]
Combining the classical subderivations of $\mathcal{D}'_1$ and $\mathcal{D}'_2$, we have the following $\mathsf{G3cp}$-derivations:
   \[
\AxiomC{$\mathcal{D}^{a\vee p}$}
\noLine
\UnaryInfC{$a\vee p\seq p,a$}
\AxiomC{$\mathcal{D}^{a\vee p,b}$}
\noLine
\UnaryInfC{$b, p\seq b\land p$}
 \RightLabel{\footnotesize $\mathsf{Cut}$}
 \BinaryInfC{$a\vee p,b\seq a,b\land p$}
 \DisplayProof \hspace{0.5cm}
 \AxiomC{$\mathcal{D}^{a\vee q}$}
\noLine
\UnaryInfC{$a\vee q\seq q,a$}
\AxiomC{$\mathcal{D}^{a\vee q,b}$}
\noLine
\UnaryInfC{$b, q\seq b\land q$}
 \RightLabel{\footnotesize $\mathsf{Cut}$}
 \BinaryInfC{$a\vee q,b\seq a,b\land q$}
 \DisplayProof
\]
By classical cut elimination, we may transform the first of these into a classical cutfree derivation $\mathcal{D}_{a\vee p, b}$ of $a\vee p,b\seq a,b\land p$, and similarly for the second. Combining these derivations with subderivations derived as appropriate from $\mathcal{D}'_1$ and $\mathcal{D}'_2$, we have the following cutfree derivation of $a\vee (p\vvee q),b\seq (b\land p)\vvee (b\land q)$: 
    \[
\AxiomC{$\mathcal{D}_{a\vee p,b}$}
\noLine
\UnaryInfC{$a\vee p,b\seq a,b\land p$}
 \RightLabel{\footnotesize $\mathsf{R}\vvee$}
\UnaryInfC{$a\vee p,b\seq a,(b\land p)\vvee (b\land q)$}
\AxiomC{$\mathcal{D}_{a\vee q,b}$}
\noLine
\UnaryInfC{$a\vee q,b\seq a,b\land q$}
 \RightLabel{\footnotesize $\mathsf{R}\vvee$}
\UnaryInfC{$a\vee q,b\seq a,(b\land p)\vvee (b\land q)$}
 \RightLabel{\footnotesize $\mathsf{L}\vvee$}
\BinaryInfC{$a\vee (p\vvee q),b\seq a, (b\land p)\vvee (b\land q)$}
 \DisplayProof
\]
\end{example}

Note that we have now the following counterpart to the resolution theorem (Theorem \ref{di:theorem:resolution}) in our system (cf. \cite[Theorem 1]{ciardelli2018}):

\begin{corollary}[Derivability resolution theorem] \label{coro:derivability_resolution}
There is an effective procedure transforming any derivation $\mathcal{D}$ witnessing $\vdash_{\scriptsize \mathsf{GT}}\Ga\seq \De$ into derivations $\mathcal{D}_\Xi$ ($\Xi\in \RRR(\Gamma)$) witnessing $\vdash_{\scriptsize \mathsf{GT}}\bigdoublewedge_{\Xi\in \RRR(\Gamma)}(\Xi\seq f[\Xi])$ where $f:\RRR(\Ga)\to \RRR(\De)$, and vice versa.
\end{corollary}
\begin{proof}
   For the first direction, Theorem \ref{di:theorem:cut_elimination} yields a derivation witnessing $\vdash_{\scriptsize \mathsf{GT}^-}\Ga\seq \De$, and the result then follows by Theorem \ref{di:theorem:normal_form}. For the second direction, we can clearly construct a derivation witnessing $\vdash_{\scriptsize \mathsf{GT}}\Ga\seq \De$ using the derivations $\mathcal{D}_\Xi$ and the rules $\mathsf{R}\vvee$ and $\mathsf{L}\vvee$. 
\end{proof}

\section{Interpolation} \label{di:section:interpolation}

The cutfree completeness of $\mathsf{GT}$ enables us to provide a proof of Craig's and Lyndon's interpolation for $\PLVVEE$ by adapting Maehara's method involving partition sequents (also known as split sequents) \cite{maehara}.

 We say that an occurrence of some propositional variable in $\phi$ occurs \emph{positively}/\emph{negatively} in $\phi$ if the occurrence is in the scope of an even/odd number of $\lnot$ in $\phi$. We write $\mathsf{P}^+(\phi)$/$\mathsf{P}^-(\phi)$ for the set of propositional variables occurring positively/negatively in $\phi$, and we let $\mathsf{P}^i(\Gamma):=\bigcup_{\phi\in \Gamma}\mathsf{P}^i(\phi)$, for $i\in \{+,-\}$.

A \emph{partition sequent} is an expression of the form $\Gamma_1;\Gamma_2\seq \De_1; \De_2$ such that $\Ga_1,\Ga_2\seq \De_1,\De_2$ is a sequent. We say that a formula $\phi$ is a \emph{sequent interpolant} of $\Gamma_1;\Gamma_2\seq \De_1; \De_2$, written $\Gamma_1;\Gamma_2\overset{\scriptsize \phi}{\seq} \De_1; \De_2$, if (i) $\vdash_{\scriptsize \mathsf{GT}^-}\Ga_1\seq \De_1,\phi$ and $\vdash_{\scriptsize \mathsf{GT}^-}\Ga_2,\phi\seq \De_2$; (ii) $\mathsf{P}(\phi)^i\subseteq (\mathsf{P}^i(\Ga_1)\cup \mathsf{P}^j(\De_1)) \cap (\mathsf{P}^j( \Ga_2)\cup \mathsf{P}^i(\De_2))$, for $i\in \{+,-\}$ and $j\in \{+,-\}\setminus \{i\}$.

Interpolation via Maehara's method for $\mathsf{G3cp}$ \cite[pp. 116--118]{troelstra} yields the result that for a sequent $\Ga \seq \De$ of $\mathsf{G3cp}$, for each partition sequent $\Ga_1;\Ga_2\seq \De_1;\De_2$ such that $\Ga_1,\Ga_2=\Ga$ and $\De_1,\De_2=\De$ (that is, for each pair of partitions of $\Ga$ into $\Ga_1$ and $\Ga_2$ and $\De$ into $\De_1$ and $\De_2$), there is a sequent interpolant. Interestingly, this result does not hold in general for $\PLVVEE$. Consider, for instance, the following valid sequent, where $?p:=p\vvee \lnot p$, etc.
$$?p, ?q\seq ?p\land ?q \land r , ?p \land ?q \land \lnot r$$
If it were the case that this form of interpolation held for $\PLVVEE$, there would be a formula $\phi$ such that $?p; ?q\overset{\scriptsize \phi}{\seq} ?p\land ?q \land r ; ?p \land ?q \land \lnot r$, and hence such that the following sequents are valid: $?p\seq ?p\land ?q \land r,\phi$ (whence $?p\models (?p\land ?q \land r)\vee \phi$) and $?q,\phi\seq ?p\land ?q \land \lnot r$ (whence $?q\land \phi \models ?p\land ?q \land \lnot r$). Now consider the team $t:=\{v_{pqr},v_{p\bar{q}r}\}$, where, e.g., $v_{p\bar{q}r}(p)=1$, $v_{p\bar{q}r}(q)=0$, $v_{p\bar{q}r}(r)=1$, etc. Since $t\models ?p$, we must have $t\models (?p\land ?q \land r)\vee \phi$, whence $t=t_1\cup t_2$, where $t_1\models ?p\land ?q \land r$ and $t_2\models \phi$. Clearly $v_{pqr}$ and $v_{p\bar{q}r}$ cannot both belong to $t_1$, whence either $v_{p\bar{q}r}\in t_2$ or $v_{pqr}\in t_2$. In the former case (the latter case is similar), by downward closure, $ \{v_{p \bar{q}r}\}\models \phi$. Clearly also $ \{v_{p\bar{q}r}\}\models ?q$, whence $ \{v_{p\bar{q}r}\}\models ?p\land ?q \land \lnot r$, but this clearly cannot be the case.

However, we can prove a version of this result restricted to partition sequents $\Ga_1;\Ga_2\seq \Lambda_1;\De_2$ where $\Lambda_1$ is classical.

\begin{theorem} \label{di:theorem:interpolation}
    If $\vdash_{\scriptsize \mathsf{GT}} \Ga\seq \De$, then for each pair of partitions $ \Ga_1 ,\Ga_2$ of $\Ga$ and $\Lambda_1 ,\De_2$ of $ \De$, there is a sequent interpolant $\phi$ of $ \Ga_1 ;\Ga_2\seq \Lambda_1 ;\De_2$, and if $\De_2$ is classical, then $\phi$ is classical.
\end{theorem}
\begin{proof}
By cutfree completeness, there is a derivation $\mathcal{D}$ that witnesses $\vdash_{\scriptsize \mathsf{GT^-}} \Ga\seq \De$. We prove the result by induction on the depth of $\mathcal{D}$. Given $\psi\in \Ga$, we often write $\Ga=\psi,\Ga'$. Most cases are as in the corresponding proof for $\mathsf{G3cp}$; we only provide the details for some interesting inductive cases, but summarize the algorithm for constructing interpolants in Table \ref{di:table:interpolants}. We omit mention of the constraints on $\mathsf{P}_i(\phi)$ ($i\in \{+,-\}$) in the proof; that they hold is easily confirmed by inspecting Table \ref{di:table:interpolants}.

\newcommand\Tstrut{\rule{0pt}{2.5ex}}
\newcommand\Tstrutt{\rule{0pt}{4.4ex}}
\newcommand\Tstruttt{\rule{0pt}{5.5ex}}
\newcommand\Tstrutttt{\rule{0pt}{3ex}}

\newcommand\Bstrut{\rule[-1.3ex]{0pt}{0pt}} 
\newcommand\Bstrutt{\rule[-3.1ex]{0pt}{0pt}}
\newcommand\Bstruttt{\rule[-4.1ex]{0pt}{0pt}}
\newcommand\Bstrutttt{\rule[-4.5ex]{0pt}{0pt}}
\setcounter{table}{0}
\begin{table}[ht]
\centering
\scriptsize
\begin{tabular}{ C{0.5cm} | C{3.45cm}  C{3.45cm} | C{3.5cm}  C{3.5cm}  }
 \Tstrut\Bstrut Rule & \multicolumn{2}{c |}{Principal formula in $\Ga_1$/$\Lambda_1$} & \multicolumn{2}{c}{Principal formula in $\Ga_2$/$\De_2$} \\
 \hline
 \Tstrut &  $p\in \Ga_1\cap\Lambda_1$ & \cellcolor{gray!10}  $p\in \Ga_1\cap \De_2$ & $p\in \Ga_2\cap\Lambda_1$ & \cellcolor{gray!10} $p\in \Ga_2\cap\De_2$ \\ 
 \multirow{-2}{0.2cm}{\vspace{-0.2cm}$\mathsf{Ax}$  } \Bstrut&   $\Ga_1',p;\Ga_2\overset{\tiny \bot}{\seq}p,\Lambda_1';\De_2$ & \cellcolor{gray!10} $\Ga_1',p;\Ga_2\overset{\tiny p}{\seq}\Lambda_1;p,\De_2'$ & $\Ga_1;\Ga_2',p\overset{\tiny \lnot p}{\seq}p,\Lambda'_1;\De_2$ &  \cellcolor{gray!10} $\Ga_1;\Ga_2',p\overset{\tiny \lnot \bot}{\seq}\Lambda_1;p,\De_2'$\\ 
 
 \rowcolor{gray!20}\Tstrut\Bstrut$\mathsf{L}\bot$ & \multicolumn{2}{c |}{ $\Ga_1',\bot;\Ga_2\overset{\tiny \bot}{\seq}\Lambda_1;\De_2$  } & \multicolumn{2}{c}{ $\Ga_1;\Ga_2',\bot\overset{\tiny \lnot \bot}{\seq}\Lambda_1;\De_2$  }\\ 
 
  \Tstruttt\Bstruttt$\mathsf{L}\lnot$ &  \multicolumn{2}{c |}{\AxiomC{$\Ga_1';\Ga_2\overset{\tiny \phi}{\seq}\alpha,\Lambda_1;\De_2$}
 \RightLabel{\tiny$\mathsf{L}\lnot$}
 \UnaryInfC{$\Ga_1',\lnot\alpha;\Ga_2\overset{\tiny \phi}{\seq}\Lambda_1;\De_2$ }
 \DisplayProof} & \multicolumn{2}{c}{\AxiomC{$\Ga_1;\Ga_2'\overset{\tiny \phi}{\seq}\Lambda_1;\alpha,\De_2$}
 \RightLabel{\tiny$\mathsf{L}\lnot$}
 \UnaryInfC{$\Ga_1;\Ga_2',\lnot\alpha\overset{\tiny \phi}{\seq}\Lambda_1;\De_2$ }
 \DisplayProof}    \\ 
 
  \rowcolor{gray!20}\Tstruttt\Bstruttt$\mathsf{R}\lnot$ &  \multicolumn{2}{c |}{\AxiomC{$\Ga_1,\alpha;\Ga_2\overset{\tiny \phi}{\seq}\Lambda_1';\De_2$}
 \RightLabel{\tiny$\mathsf{R}\lnot$}
 \UnaryInfC{ $\Ga_1;\Ga_2\overset{\tiny \phi}{\seq}\lnot\alpha,\Lambda'_1;\De_2$ }
 \DisplayProof }& \multicolumn{2}{c}{\AxiomC{$\Ga_1;\Ga_2,\alpha\overset{\tiny \phi}{\seq}\Lambda_1;\De_2'$}
 \RightLabel{\tiny$\mathsf{R}\lnot$}
 \UnaryInfC{ $\Ga_1;\Ga_2,\overset{\tiny \phi}{\seq}\Lambda_1;\lnot\alpha,\De'_2$ }
 \DisplayProof }   \\ 
   
  \Tstruttt\Bstruttt$\mathsf{L}\land$ &  \multicolumn{2}{c |}{ \AxiomC{$\Ga_1',\psi,\chi;\Ga_2\overset{\tiny \phi}{\seq}\Lambda_1;\De_2$}
 \RightLabel{\tiny$\mathsf{L}\land$}
 \UnaryInfC{$\Ga_1',\psi\land\chi;\Ga_2\overset{\tiny \phi}{\seq}\Lambda_1;\De_2$ }
 \DisplayProof} & \multicolumn{2}{c}{ \AxiomC{$\Ga_1;\Ga_2',\psi,\chi\overset{\tiny \phi}{\seq}\Lambda_1;\De_2$}
 \RightLabel{\tiny$\mathsf{L}\land$}
 \UnaryInfC{$\Ga_1;\Ga_2',\psi\land\chi\overset{\tiny \phi}{\seq}\Lambda_1;\De_2$ }
 \DisplayProof}   \\ 

\rowcolor{gray!20}\Tstruttt\Bstruttt$\mathsf{R}\land$ &  \multicolumn{2}{c |}{  \AxiomC{$\Ga_1;\Ga_2\overset{\tiny \alpha_1}{\seq} \psi,\Lambda_1';\Lambda_2$}
 \AxiomC{$\Ga_1;\Ga_2\overset{\tiny \alpha_2}{\seq} \chi,\Lambda_1';\Lambda_2$}
 \RightLabel{{\tiny$\mathsf{R}\land$}}
 \BinaryInfC{$\Ga_1;\Ga_2\overset{\tiny \alpha_1 \vee \alpha_2}{\seq} \psi\land \chi,\Lambda_1';\Lambda_2,\De_2'$}
 \DisplayProof} & \multicolumn{2}{c}{  \AxiomC{$\Ga_1;\Ga_2\overset{\tiny \phi_1}{\seq}\Lambda_1;\psi,\Lambda_2$}
 \AxiomC{$\Ga_1;\Ga_2\overset{\tiny \phi_2}{\seq}\Lambda_1;\chi,\Lambda_2$}
 \RightLabel{{\tiny$\mathsf{R}\land$}}
 \BinaryInfC{$\Ga_1;\Ga_2\overset{\tiny \phi_1\land\phi_2}{\seq}\Lambda_1;\psi\land\chi,\Lambda_2,\De'_2$}
 \DisplayProof}    \\ 

 \Tstruttt\Bstruttt$\mathsf{L}\vee$ & \multicolumn{2}{c |}{ \AxiomC{$\Ga_1;\Ga_2\overset{\tiny \alpha_1}{\seq} \psi,\Lambda_1';\Lambda_2$}
 \AxiomC{$\Ga_1',\chi;\Ga_2\overset{\tiny \alpha_2}{\seq}\Lambda_1;\Lambda_2$}
 \RightLabel{{ \tiny$\mathsf{L}\vee$}}
 \BinaryInfC{$\Ga_1',\psi\vee\chi;\Ga_2\overset{\tiny \alpha_1\vee\alpha_2}{\seq}\Lambda_1;\Lambda_2,\De_2'$}
 \DisplayProof  } & \multicolumn{2}{c}{ \AxiomC{$\Ga_1;\Ga_2',\psi\overset{\tiny \alpha_1}{\seq}\Lambda_1;\Lambda_2$}
 \AxiomC{$\Ga_1;\Ga_2',\chi\overset{\tiny \alpha_2}{\seq}\Lambda_1;\Lambda_2$}
 \RightLabel{{ \tiny$\mathsf{L}\vee$}}
 \BinaryInfC{$\Ga_1;\Ga_2',\psi\vee\chi\overset{\tiny \alpha_1\land\alpha_2}{\seq}\Lambda_1;\Lambda_2,\De'_2$ }
 \DisplayProof }  \\ 

\rowcolor{gray!20}\Tstruttt\Bstruttt$\mathsf{R}\vee$ &  \multicolumn{2}{c |}{ \AxiomC{$\Ga_1;\Ga_2\overset{\tiny \phi}{\seq}\psi,\chi,\Lambda_1';\De_2$}
 \RightLabel{\tiny$\mathsf{R}\vee$}
 \UnaryInfC{$\Ga_1;\Ga_2\overset{\tiny \phi}{\seq}\psi\vee\chi,\Lambda_1';\De_2$}
 \DisplayProof  } & \multicolumn{2}{c}{ \AxiomC{$\Ga_1;\Ga_2\overset{\tiny \phi}{\seq}\Lambda_1;\psi,\chi,\De_2'$}
 \RightLabel{\tiny$\mathsf{R}\vee$}
 \UnaryInfC{$\Ga_1;\Ga_2\overset{\tiny \phi}{\seq}\Lambda_1;\psi\vee\chi,\De_2'$}
 \DisplayProof  }    \\ 
  
 \Tstrutttt & \multicolumn{2}{c|}{$\De_2$ classical} & &   \\
\Bstrutttt& \multicolumn{2}{c|}{\hspace{-0.2cm}\AxiomC{$\Ga_1',\chi\{\phi_L\};\Ga_2\overset{\tiny \alpha_1}{\seq}\Lambda_1;\De_2$}
 \AxiomC{$\Ga_1',\chi\{\phi_R\};\Ga_2\overset{\tiny \alpha_2}{\seq}\Lambda_1;\De_2$}
 \RightLabel{{\tiny $\mathsf{L}\vvee$}}
 \BinaryInfC{$\Ga_1',\chi\{\phi_L\vvee\phi_R\};\Ga_2\overset{\tiny \alpha_1\vee\alpha_2}{\seq}\Lambda_1;\De_2$}
 \DisplayProof} & & \\
\Tstrut & \multicolumn{2}{c|}{\cellcolor{gray!10}$\De_2$ not classical} & \multicolumn{2}{c}{\multirow{-3}{*}{\vspace{-0.2cm}\hspace{-0.2cm}\AxiomC{$\Ga_1;\Ga_2',\chi\{\phi_L\}\overset{\tiny \phi_1}{\seq}\Lambda_1;\De_2$}
 \AxiomC{$\Ga_1;\Ga_2',\chi\{\phi_R\}\overset{\tiny \phi_2}{\seq}\Lambda_1;\De_2$}
 \RightLabel{{ \tiny$\mathsf{L}\vvee$}}
 \BinaryInfC{ $\Ga_1;\Ga_2',\chi\{\phi_L\vvee\phi_R\}\overset{\tiny \phi_1\land\phi_2}{\seq}\Lambda_1;\De_2$}
 \DisplayProof}} \\
 \Bstrutttt\multirow{-4}{*}{\vspace{0.85cm}$\mathsf{L}\vvee$}&\multicolumn{2}{c|}{\cellcolor{gray!10}\hspace{-0.2cm}\AxiomC{$\Ga_1',\chi\{\phi_L\};\Ga_2\overset{\tiny \phi_1}{\seq}\Lambda_1;\De_2$}
 \AxiomC{$\Ga_1',\chi\{\phi_R\};\Ga_2\overset{\tiny \phi_2}{\seq}\Lambda_1;\De_2$}
 \RightLabel{{ \tiny$\mathsf{L}\vvee$}}
 \BinaryInfC{$\Ga_1',\chi\{\phi_L\vvee\phi_R\};\Ga_2\overset{\tiny \phi_1\vvee\phi_2}{\seq}\Lambda_1;\De_2$}
 \DisplayProof}  & &     \\ 

 \rowcolor{gray!20}\Tstruttt\Bstruttt$\mathsf{R}\vvee$ &  \multicolumn{2}{c|}{ } & \multicolumn{2}{c}{ \AxiomC{$\Ga_1;\Ga_2\overset{\tiny \phi}{\seq}\Lambda_1;\chi\{\phi_i\},\De_2'$}
 \RightLabel{\tiny$\mathsf{R}\vee$}
 \UnaryInfC{$\Ga_1;\Ga_2\overset{\tiny \phi}{\seq}\Lambda_1;\chi\{\phi_L\vvee\phi_R\},\De_2'$}
 \DisplayProof  }    \\ 
\end{tabular}
\caption{Sequent interpolant algorithm. The conclusion of each rule application displays the relevant interpolation result, and the premise(s) display how the subformulas of the interpolant are obtained by applying the algorithm to partition sequent(s) corresponding to the premise(s) of the rule.}
\label{di:table:interpolants}
\end{table}

Suppose we have shown the result for all derivations of depth less than the depth of $\mathcal{D}$, and let $R$ denote the final rule applied in $\mathcal{D}$. Fix a pair of partitions $ \Ga_1 ,\Ga_2=\Gamma$ and $\Lambda_1 ,\De_2=\De$.

Suppose $R$ is $\mathsf{R}\land$.

 Case 1: The principal formula $\psi\land \chi$ of $R$ is in $\Lambda_1$ so that $R$ is of the form:
    \[
 \AxiomC{$\Ga_1,\Ga_2\seq \psi,\Lambda_1',\Lambda_2$}
 \AxiomC{$\Ga_1,\Ga_2\seq \chi,\Lambda_1',\Lambda_2$}
 \RightLabel{{\footnotesize $\mathsf{R}\land$}}
 \BinaryInfC{$\Ga_1,\Ga_2\seq \psi\land \chi,\Lambda_1',\Lambda_2,\De_2'$}
 \DisplayProof
\]
We apply the induction hypothesis to obtain $\alpha_1$ and $\alpha_2$ with $\Ga_1;\Ga_2\overset{\scriptsize \alpha_1}{\seq}\psi,\Lambda_1';\Lambda_2$ and $\Ga_1;\Ga_2\overset{\scriptsize \alpha_2}{\seq}\chi,\Lambda_1';\Lambda_2$, whence $\vdash\Ga_1\seq \psi,\Lambda_1',\alpha_1$, $\vdash\Ga_2,\alpha_1\seq \Lambda_2$, $\vdash\Ga_1\seq \chi,\Lambda_1',\alpha_2$, and $\vdash\Ga_2,\alpha_2\seq \Lambda_2$ (where, since $\Lambda_2$ is classical, $\alpha_1$ and $\alpha_2$ are classical). We show $\Ga_1;\Ga_2\overset{\scriptsize \alpha_1\vee\alpha_2}{\seq}\Lambda_1;\De_2$. By right weakening (Lemma \ref{di:lemma:weakening}), we have $\vdash\Ga_1\seq \psi,\Lambda_1',\alpha_1,\alpha_2$ and $\vdash\Ga_1\seq \chi,\Lambda_1',\alpha_1,\alpha_2$. Then by $\mathsf{R}\land $, given that $\Lambda_1',\alpha_1,\alpha_2$ is classical, $\vdash\Ga_1\seq \psi\land\chi,\Lambda_1',\alpha_1,\alpha_2$, whence by $\mathsf{R}\vee $, $\vdash\Ga_1\seq \psi\land\chi,\Lambda_1',\alpha_1\vee\alpha_2$. On the other hand, we have by $\mathsf{L}\vee$ that $\vdash\Ga_2,\alpha_1\vee \alpha_2\seq \Lambda_2,\De_2'$.

Case 2: The principal formula $\psi\land \chi$ of $R$ is in $\De_2=\psi\land \chi,\Lambda_2,\De'_2$ so that $R$ is of the form:
    \[
 \AxiomC{$\Ga_1,\Ga_2\seq \Lambda_1,\psi,\Lambda_2$}
 \AxiomC{$\Ga_1,\Ga_2\seq \Lambda_1,\chi,\Lambda_2$}
 \RightLabel{{\footnotesize $\mathsf{R}\land$}}
 \BinaryInfC{$\Ga_1,\Ga_2\seq \Lambda_1,\psi\land \chi,\Lambda_2, \De'_2$}
 \DisplayProof
\]
 We apply the induction hypothesis to obtain $\phi_1$ and $\phi_2$ with $\Ga_1;\Ga_2\overset{\tiny \phi_1}{\seq}\Lambda_1;\psi,\Lambda_2$ and $\Ga_1;\Ga_2\overset{\tiny \phi_2}{\seq}\Lambda_1;\chi,\Lambda_2$, whence $\vdash\Ga_1\seq \Lambda_1,\phi_1$, $\vdash\Ga_2,\phi_1\seq \psi,\Lambda_2$, $\vdash\Ga_1\seq \Lambda_1,\phi_2$, and $\vdash\Ga_2,\phi_2\seq \chi,  \Lambda_2$. We show $\Ga_1;\Ga_2\overset{\tiny \phi_1\land\phi_2}{\seq}\Lambda_1;\De_2$. By $\mathsf{R}\land$, we have that $\vdash\Ga_1\seq\Lambda_1, \phi_1\land \phi_2$. On the other hand, by left weakening (Lemma \ref{di:lemma:weakening}), we have $\vdash\Ga_2,\phi_1,\phi_2\seq \psi,  \Lambda_2$ and $\vdash\Ga_2,\phi_1,\phi_2\seq \chi,  \Lambda_2$. Then by $\mathsf{R}\land$, $\vdash\Ga_2,\phi_1,\phi_2\seq \psi\land \chi,  \Lambda_2,\De'_2$, whence by $\mathsf{L}\land $, $\vdash\Ga_2,\phi_1 \land \phi_2\seq \psi\land \chi,  \Lambda_2,\De'_2$. Moreover, if $\Delta_2=\psi\land \chi,\Lambda_2,\De'_2$ is classical, then so are $\psi,\Lambda_2$ and $\chi,\Lambda_2$, whence, by the induction hypothesis, so are the interpolants $\phi_1$ and $\phi_2$, so that, finally, $\phi_1 \land \phi_2$ is also classical.

 Suppose $R$ is $\mathsf{L}\vee$.

 Case 1: The principal formula $\psi\vee\chi$ of $R$ is in $\Ga_1$ so that $R$ is of the form:
   \[
 \AxiomC{$\Ga'_1,\psi,\Ga_2\seq \Lambda_1,\Lambda_2$}
 \AxiomC{$\Ga_1',\chi,\Ga_2\seq \Lambda_1,\Lambda_2$}
 \RightLabel{{\footnotesize $\mathsf{L}\vee$}}
 \BinaryInfC{$\Ga_1',\psi\vee\chi,\Ga_2\seq \Lambda_1,\Lambda_2,\De_2'$}
 \DisplayProof
\]
We apply the induction hypothesis to obtain $\alpha_1$ and $\alpha_2$ with $\Ga'_1,\psi;\Ga_2\overset{\tiny \alpha_1}{\seq}\Lambda_1;\Lambda_2$ and $\Ga_1',\chi;\Ga_2\overset{\tiny \alpha_2}{\seq}\Lambda_1;\Lambda_2$, whence $\vdash\Ga_1',\psi\seq \Lambda_1,\alpha_1$, $\vdash\Ga_2,\alpha_1\seq \Lambda_2$, $\vdash\Ga_1',\chi\seq \Lambda_1,\alpha_2$, and $\vdash\Ga_2,\alpha_2\seq \Lambda_2$ (where, since $\Lambda_2$ is classical, $\alpha_1$ and $\alpha_2$ are classical). We show $\Ga_1;\Ga_2\overset{\tiny \alpha_1\vee\alpha_2}{\seq}\Lambda_1;\De_2$. By right weakening (Lemma \ref{di:lemma:weakening}), we have $\vdash\Ga_1',\psi\seq \Lambda_1,\alpha_1,\alpha_2$ and $\vdash\Ga_1',\chi\seq \Lambda_1,\alpha_1,\alpha_2$. Then by $\mathsf{L}\vee $, given that $\Lambda_1,\alpha_1,\alpha_2$ is classical, $\vdash\Ga_1',\psi\vee\chi\seq \Lambda_1,\alpha_1,\alpha_2$, whence by $\mathsf{R}\vee $, $\vdash\Ga_1',\psi\vee\chi\seq \Lambda_1,\alpha_1\vee \alpha_2$. On the other hand, we have by $\mathsf{L}\vee$ that  $\vdash\Ga_2,\alpha_1\vee \alpha_2\seq \Lambda_2,\De_2'$.

Case 2: The principal formula $\psi\vee\chi$ of $R$ is in $\Ga_2$ so that $R$ is of the form:
    \[
 \AxiomC{$\Ga_1,\Ga'_2,\psi\seq \Lambda_1,\Lambda_2$}
 \AxiomC{$\Ga_1,\Ga'_2,\chi,\seq \Lambda_1,\Lambda_2$}
 \RightLabel{{\footnotesize $\mathsf{L}\vee$}}
 \BinaryInfC{$\Ga_1,\Ga_2',\psi\vee\chi,\seq \Lambda_1,\Lambda_2,\De_2'$}
 \DisplayProof
\]
We apply the induction hypothesis to obtain $\alpha_1$ and $\alpha_2$ with $\Ga_1;\Ga_2',\psi\overset{\tiny \alpha_1}{\seq}\Lambda_1;\Lambda_2$ and $\Ga_1;\Ga_2',\chi\overset{\tiny \alpha_2}{\seq}\Lambda_1;\Lambda_2$, whence $\vdash\Ga_1\seq \Lambda_1,\alpha_1$, $\vdash\Ga_2',\psi,\alpha_1\seq \Lambda_2$, $\vdash\Ga_1\seq \Lambda_1,\alpha_2$, and $\vdash\Ga_2',\chi,\alpha_2\seq \Lambda_2$ (where, since $\Lambda_2$ is classical, $\alpha_1$ and $\alpha_2$ are classical, though we do not make use of classicality in this case). We show $\Ga_1;\Ga_2\overset{\tiny \alpha_1\land\alpha_2}{\seq}\Lambda_1;\De_2$. By $\mathsf{R}\land$, we have $\Ga_1\seq \Lambda_1,\alpha_1\land \alpha_2$. On the other hand, by left weakening (Lemma \ref{di:lemma:weakening}), we have $\vdash\Ga_2',\psi,\alpha_1,\alpha_2\seq \Lambda_2$ and $\vdash\Ga_2',\chi,\alpha_1,\alpha_2\seq \Lambda_2$. Then by $\mathsf{L}\land$, $\vdash\Ga_2',\psi,\alpha_1\land \alpha_2\seq \Lambda_2$ and $\vdash\Ga_2',\chi,\alpha_1\land \alpha_2\seq \Lambda_2$, whence by $\mathsf{L}\vee$, $\vdash\Ga_2',\psi\vee \chi,\alpha_1\land \alpha_2\seq \Lambda_2,\De_2'$.

 Suppose $R$ is $\mathsf{L}\vvee$.

 Case 1: The principal formula $\chi\{\phi_L\vvee\phi_R\}$ of $R$ is in $\Ga_1$ so that $R$ is of the form:
  \[
 \AxiomC{$\Ga'_1,\chi\{\phi_L\},\Ga_2\seq \Lambda_1,\De_2$}
 \AxiomC{$\Ga'_1,\chi\{\phi_R\},\Ga_2\seq \Lambda_1,\De_2$}
 \RightLabel{{\footnotesize $\mathsf{L}\vvee$}}
 \BinaryInfC{$\Ga'_1,\chi\{\phi_L\vvee\phi_R\},\Ga_2\seq \Lambda_1,\De_2$}
 \DisplayProof
\]
We apply the induction hypothesis to obtain $\phi_1$ and $\phi_2$ with $\Ga'_1,\chi\{\phi_L\};\Ga_2\overset{\tiny \phi_1}{\seq}\Lambda_1;\De_2$ and $\Ga'_1,\chi\{\phi_R\};\Ga_2\overset{\tiny \phi_2}{\seq}\Lambda_1;\De_2$, whence $\vdash\Ga_1',\chi\{\phi_L\}\seq \Lambda_1,\phi_1$, $\vdash\Ga_2,\phi_1\seq \De_2$, $\vdash\Ga_1',\chi\{\phi_R\}\seq \Lambda_1,\phi_2$, and $\vdash\Ga_2,\phi_2\seq \De_2$, where if $\De_2$ is classical, $\phi_1$ and $\phi_2$ are classical. 

Case 1.1: $\De_2$ is classical. We show $\Ga_1;\Ga_2\overset{\tiny \phi_1\vee\phi_2}{\seq}\Lambda_1;\De_2$. Note that since $\phi_1$ and $\phi_2$ are classical, our new interpolant $\phi_1\vee \phi_2$ will also be classical. By right weakening (Lemma \ref{di:lemma:weakening}), we have $\vdash\Ga_1',\chi\{\phi_L\}\seq \Lambda_1,\phi_1,\phi_2$ and $\vdash\Ga_1',\chi\{\phi_R\}\seq \Lambda_1,\phi_1,\phi_2$ so by $\mathsf{L}\vvee$, $\vdash\Ga_1',\chi\{\phi_L\vvee\phi_R\}\seq \Lambda_1,\phi_1,\phi_2$ and then by $\mathsf{R}\vee$, $\vdash\Ga_1',\chi\{\phi_L\vvee\phi_R\}\seq \Lambda_1,\phi_1\vee \phi_2$. On the other hand, since $\Delta_2$ is classical, by $\mathsf{L}\vee$ we have $\vdash\Ga_2,\phi_1\vee \phi_2\seq \De_2$.

Case 1.2:  $\Delta_2$ is not classical. We show $\Ga_1;\Ga_2\overset{\tiny \phi_1\vvee\phi_2}{\seq}\Lambda_1;\De_2$. By $\mathsf{R}\vvee$, we have $\vdash\Ga_1',\chi\{\phi_L\}\seq \Lambda_1,\phi_1\vvee \phi_2$ and $\vdash\Ga_1',\chi\{\phi_R\}\seq \Lambda_1,\phi_1\vvee \phi_2$ so by $L\vvee$, $\vdash\Ga_1',\chi\{\phi_L\vvee\phi_R\}\seq \Lambda_1,\phi_1 \vvee \phi_2$. On the other hand, by $\mathsf{L}\vvee$, $\vdash\Ga_2,\phi_1\vvee \phi_2\seq \De_2$.

Case 2: The principal formula $\chi\{\phi_L\vvee\phi_R\}$ of $R$ is in $\Ga_2$ so that $R$ is of the form:
 \[
 \AxiomC{$\Ga_1,\Ga_2',\chi\{\phi_L\},\seq \Lambda_1,\De_2$}
 \AxiomC{$\Ga_1,\Ga_2',\chi\{\phi_R\},\seq \Lambda_1,\De_2$}
 \RightLabel{{\footnotesize $\mathsf{L}\vvee$}}
 \BinaryInfC{$\Ga_1,\Ga_2', \chi\{\phi_L\vvee\phi_R\},\seq \Lambda_1,\De_2$}
 \DisplayProof
\]
We apply the induction hypothesis to obtain $\phi_1$ and $\phi_2$ with $\Ga_1;\Ga_2',\chi\{\phi_L\}\overset{\tiny \phi_1}{\seq}\Lambda_1;\De_2$ and $\Ga_1;\Ga_2',\chi\{\phi_R\}\overset{\tiny \phi_2}{\seq}\Lambda_1;\Lambda_2$, whence $\vdash \Ga_1,\seq \Lambda_1,\phi_1$, $\vdash \Ga_2',\chi\{\phi_L\},\phi_1\seq \De_2$, $\vdash \Ga_1,\seq \Lambda_1,\phi_2$, and $\vdash \Ga_2',\chi\{\phi_R\},\phi_2\seq \De_2$. We show $\Ga_1;\Ga_2\overset{\tiny \phi_1\land\phi_2}{\seq}\Lambda_1;\De_2$. By $\mathsf{R}\land $, since $\Lambda_1$ is classical, $\vdash \Ga_1,\seq \Lambda_1,\phi_1\land \phi_2$. On the other hand, by left weakening (Lemma \ref{di:lemma:weakening}), $\vdash \Ga_2',\chi\{\phi_L\},\phi_1,\phi_2\seq \De_2$, and $\vdash \Ga_2',\chi\{\phi_R\},\phi_1,\phi_2\seq \De_2$ so by $\mathsf{L}\land $, $\vdash \Ga_2',\chi\{\phi_L\},\phi_1\land \phi_2\seq \De_2$ and $\vdash \Ga_2',\chi\{\phi_R\},\phi_1 \land \phi_2\seq \De_2$ whence by $\mathsf{L}\vvee$, $\vdash \Ga_2',\chi\{\phi_L\vvee\phi_R\},\phi_1 \land \phi_2\seq \De_2$. Moreover, if $\Delta_2$ is classical, then so are $\phi_1$ and $\phi_2$, and hence also $\phi_1 \land \phi_2$.

  Suppose $R$ is $\mathsf{R}\vvee$.

 The principal formula $\chi\{\phi_L\vvee\phi_R\}$ of $R$ must be in $\De_2$ so that $R$ is of the form:
    \[
 \AxiomC{$\Ga_1,\Ga_2\seq \Lambda_1,\chi\{\psi_i\},\De_2'$}
 \RightLabel{{\footnotesize $R\vvee$}}
 \UnaryInfC{$\Ga_1,\Ga_2\seq  \Lambda_1,\chi\{\phi_L\vvee\phi_R\},\De_2'$}
 \DisplayProof
\]
We apply the induction hypothesis to obtain a $\phi$ with $\Ga_1;\Ga_2\overset{\tiny \phi}{\seq}\Lambda_1;\chi\{\psi_i\},\De_2'$, whence $\Ga_1\seq \Lambda_1,\phi$ and $\Ga_2,\phi\seq \chi\{\psi_i\},\De_2'$. We show $\Ga_1;\Ga_2\overset{\tiny \phi}{\seq}\Lambda_1;\De_2$. We already have $\Ga_1\seq \Lambda_1,\phi$, and by an application of $\mathsf{R}\vvee$, $\Ga_2,\phi\seq \chi\{\psi_L\vvee \phi_R\},\De_2'$.
\end{proof}

\begin{corollary}[Craig's and Lyndon's interpolation] \label{di:coro:interpolation} If $\phi\models \psi$, then there is a $\theta$ such that $\phi\models \theta\models \psi$, and $\mathsf{P}^i(\theta)\subseteq \mathsf{P}^i(\phi)\cap \mathsf{P}^i(\psi)$, for any $i\in \{+,-\}$.
\end{corollary}
\begin{proof}
    By completeness, $\vdash_{\scriptsize \mathsf{GT}}\phi\seq \psi$, whence by Theorem \ref{di:theorem:interpolation}, there is a $\theta$ such that $\phi;\overset{\scriptsize \theta}{\seq};\psi$, meaning that $\phi\models \theta$, $\theta\models \psi$, and $\mathsf{P}^i(\theta)\subseteq \mathsf{P}^i(\phi)\cap \mathsf{P}^i(\psi)$.
\end{proof}

\begin{example} The following is a deduction annotated with partitions as well as the corresponding sequent interpolants produced by our procedure for the partition sequent $(p\vvee q)\vee r ; \lnot p\seq  r\vee s;q\vvee x$:
\[
\AxiomC{$p ; \overset{\scriptsize p}{\seq}   r, s;q ,p$}
\AxiomC{$ r ; \overset{\scriptsize \bot}{\seq}   r, s;q,p$}
\RightLabel{\footnotesize $\mathsf{L}\vee$}
\BinaryInfC{$p\vee r ; \overset{\scriptsize p\vee\bot}{\seq}   r, s;q,p$}
\RightLabel{\footnotesize $\mathsf{L}\lnot$}
\UnaryInfC{$p\vee r ; \lnot p\overset{\scriptsize p\vee \bot}{\seq}   r, s;q$}
 \RightLabel{\footnotesize $\mathsf{R}\vee $}
\UnaryInfC{$p\vee r ; \lnot p\overset{\scriptsize p \vee\bot}{\seq}   r\vee s;q$}
 \RightLabel{\footnotesize $\mathsf{R}\vvee$}
\UnaryInfC{$p\vee r ;\lnot p\overset{\scriptsize p\vee\bot}{\seq} r\vee s;q\vvee x$ }
\AxiomC{$q ; \overset{\scriptsize q}{\seq}   r, s;q,p $}
\AxiomC{$ r ; \overset{\scriptsize \bot}{\seq}   r, s;q,p$}
\RightLabel{\footnotesize $\mathsf{L}\vee$}
\BinaryInfC{$q\vee r ; \overset{\scriptsize q\vee\bot}{\seq}   r, s;q,p$}
\RightLabel{\footnotesize $\mathsf{L}\lnot$}
\UnaryInfC{$q\vee r ;\lnot p\overset{\scriptsize q\vee \bot}{\seq}   r, s;q$}
 \RightLabel{\footnotesize $\mathsf{R}\vee $}
\UnaryInfC{$q\vee r ;\lnot p\overset{\scriptsize q \vee\bot}{\seq}   r\vee s;q$}
 \RightLabel{\footnotesize $\mathsf{R}\vvee$}
\UnaryInfC{$q\vee r ;\lnot p\overset{\scriptsize q\vee\bot}{\seq} r\vee s;q\vvee x$ }
 \RightLabel{\footnotesize $\mathsf{L}\vvee$}
\BinaryInfC{$(p\vvee q)\vee r ;\lnot p \overset{\scriptsize (p\vee\bot)\vvee(q\vee \bot) }{\seq}   r\vee s;q\vvee x $}
 \DisplayProof
 \]
\end{example}

As mentioned in Section \ref{di:section:introduction}, Corollary \ref{di:coro:interpolation} is novel in providing a Lyndon's interpolation theorem for $\PLVVEE$, as well as a Craig's interpolation theorem that does not rely on the construction of uniform interpolants; note also that the modified Maehara's method we have used is constructive, and that the size of the constructed interpolant is linear in the size of the relevant cutfree derivation (cf. the uniform interpolation proof making use of expressive completeness and locality in \cite{dagostino,yang2022}). We conjecture that one can also use our system to provide a proof of sequent uniform interpolation along the lines of \cite{Pitts_1992,iemhoff2019apal,iemhoff2019aml}, but we leave this for future work.

\section{Variant System with Independent Contexts} \label{di:section:structural_rule_variant}

Our system $\mathsf{GT}$ features shared contexts for the two-premise rules, and, as in $\mathsf{G3cp}$, the structural rules of contraction are implicit, having been ``absorbed'' into these shared-context rules. In this section, we briefly present a variant of $\mathsf{GT}$ with independent contexts for $\mathsf{R}\land$ and $\mathsf{L}\vee$ as well as explicit contraction rules, and show that it is equivalent to $\mathsf{GT}$. This variant features no syntactic restrictions on the rules $\mathsf{R}\land$ and $\mathsf{L}\vee$; the rules of the variant therefore make it explicit that the structural effect of the syntactic restrictions in conjunction with the implicit weakening in $\mathsf{R}\land$ and $\mathsf{L}\vee$ is to allow for right weakening for all formulas while allowing for right contraction only for classical formulas.

\begin{definition}[The sequent calculus $\mathsf{GT}'$]\
The rules for $\mathsf{GT}'$ are as for $\mathsf{GT}$ except we remove the rules $\mathsf{R}\land$ and $\mathsf{L}\vee$, and add:
    \\
    
\noindent
\begin{tabular}{|C{0.50\textwidth} C{0.44\textwidth} |}
\hline
&\\
 \multicolumn{2}{|c|}{\it Logical rules} \\
 &\\
 \AxiomC{$\Ga_1,\phi\seq \De_1$}
 \AxiomC{$\Ga_2,\psi\seq \De_2$}
 \RightLabel{{\footnotesize $\mathsf{L}\lor'$}}
 \BinaryInfC{$\Ga_1,\Ga_2,\phi\lor \psi\seq \De_1,\De_2$}
 \DisplayProof 
 &\AxiomC{$\Ga_1\seq \phi,\De_1$}
 \AxiomC{$\Ga_2\seq \psi,\De_2$}
 \RightLabel{{\footnotesize $\mathsf{R}\land'$}}
 \BinaryInfC{$\Ga_1,\Ga_2 \seq \phi \land \psi,\De_1,\De_2$}
 \DisplayProof \\
  &\\
 &\\
 \multicolumn{2}{|c|}{\it Structural rules} \\
 &\\
 \AxiomC{$\Ga,\phi,\phi\seq \De$}
 \RightLabel{{\footnotesize $\mathsf{LC}$}}
 \UnaryInfC{$\Ga,\phi\seq \De$}
 \DisplayProof 
 &  \AxiomC{$\Ga\seq \alpha,\alpha,\De$}
 \RightLabel{{\footnotesize $\mathsf{RC}$}}
 \UnaryInfC{$\Ga\seq \alpha,\De$}
 \DisplayProof \\
 &\\
 \hline
\end{tabular}
\end{definition}

It is easy to verify that these rules are sound; we omit the proof.

Compare, e.g., with Dummett's system for intuitionistic logic \cite[p. 97]{dummett} and with multiplicative linear logic \cite{girard}. Cf. \cite{abramsky}, in which a connection is drawn between the semantics of the split disjunction $\vee$ and the multiplicative conjunction (whereas here the rules for the split disjunction coincide with those for the multiplicate disjunction).

It is easy to see that the systems $\mathsf{GT}$ and $\mathsf{GT}'$ are equivalent:

\begin{proposition}
    $\vdash_{\scriptsize \mathsf{GT}^-}\Ga\seq \De \iff \vdash_{\scriptsize \mathsf{GT}'^-}\Ga\seq \De$.
\end{proposition}
\begin{proof}
The direction $\Longleftarrow$ follows by soundness and the cutfree completeness of $\mathsf{GT}$. The direction $\Longrightarrow$ follows by the admissibility of $\mathsf{R}\land$ and $\mathsf{L}\vee$ in $\mathsf{GT}'^-$. For instance, we have:
\[\AxiomC{$\Ga,\phi\seq \Lambda$}
 \AxiomC{$\Ga,\psi\seq \Lambda$}
 \RightLabel{{\footnotesize $\mathsf{L}\lor'$}}
 \BinaryInfC{$\Ga,\Ga,\phi\lor \psi\seq \Lambda,\Lambda$}
 \DisplayProof \]
 Clearly $\Ga,\Ga,\phi\lor \psi\seq \Lambda,\Lambda\vdash_{\scriptsize \mathsf{LC}}\Ga,\phi\lor \psi\seq \Lambda,\Lambda\vdash_{\scriptsize \mathsf{RC}}\Ga,\phi\lor \psi\seq \Lambda$. As with $\mathsf{GT}$ (Lemma \ref{di:lemma:weakening}), it is easy to show that weakening is admissible in $\mathsf{GT}'^-$, whence $\Ga,\phi\lor \psi\seq \Lambda\vdash_{\scriptsize \mathsf{GT}'^-}\Ga,\phi\lor \psi\seq \Lambda,\De$.
\end{proof}

\section{Concluding Remarks} \label{di:section:conclusion}

In this paper, we presented a sequent calculus $\mathsf{GT}$ for basic propositional team logic $\PLVVEE$ consisting of a Gentzen-style subsystem $\mathsf{G3cp}$ (without the implication rules) for classical propositional logic $\PL$ together with two deep-inference rules for the nonclassical inquisitive disjunction $\vvee$. Our results on $\mathsf{GT}$ demonstrate that the logic $\PLVVEE$ is amenable to standard proof-theoretic techniques and results, adjusted to account for the failure of universal applicability. Commas in the succedents of sequents are interpreted using the split disjunction in $\mathsf{GT}$; we saw that this leads to interesting correspondences between the structural rules of our calculus and team-semantic closure properties. We showed that $\mathsf{GT}^-$ admits height-preserving weakening, contraction, and inversion; with the caveats that right contraction is only admissible for classical formulas (which is as expected, given the connection we observed between right contraction and union closure), and that $\mathsf{R}\vvee$ is only ``invertible'' in the sense of the split property. We proved the cutfree completeness of $\mathsf{GT}$ by providing a procedure for constructing cutfree derivations and countermodels that is similar to the analogous procedure for $\mathsf{G3cp}$. We generalized the notion of resolutions from inquisitive logic to define the notion of partial resolutions, and used partial resolutions to define a weak subformula property for $\mathsf{GT}^-$ as well as to provide a second proof of the cutfree completeness of $\mathsf{GT}$. We proved a normal form theorem for cutfree derivations, using which we provided a cut elimination procedure for $\mathsf{GT}$. We proved a sequent interpolation theorem via an adaptation of Maehara's method, and derived two novel results as corollaries: Lyndon's interpolation for $\PLVVEE$, and Craig's interpolation for $\PLVVEE$ that does not rely on uniform interpolants. Finally, we defined a variant of $\mathsf{GT}$ with independent-context rules instead of shared-context rules with restricted contexts for the connectives $\vee$ and $\land$; this variant makes it explicit that the structural effect of the contextual restrictions is to disallow right contraction for non-classical formulas.

We conclude by discussing the prospects of applying our deep-inference approach to provide sequent calculi for other team logics, as well as possible modifications of the approach.

Recall first that in Section \ref{di:section:introduction} we noted that $\PLVVEE$ is an extension of propositional dependence logic (the extension of $\PL$ with dependence atoms) in the sense that dependence atoms are definable in $\PLVVEE$. The system $\mathsf{GT}$ therefore also serves as a system for propositional dependence logic, albeit with dependence atoms handled implicitly via definitions rather than explicitly with their own dedicated rules. It is not clear to us how smoothly our approach would extend to a (cutfree complete) system for propositional dependence logic in which dependence atoms are given their own rules.

It is natural ask whether the calculus $\mathsf{GT}$ could be extended in a modular way with rules for the first-order quantifiers or for the modalities to provide calculi for the first-order or modal extensions of $\PLVVEE$. There are multiple different versions of the first-order quantifiers and the modalities in team semantics in the literature (see, e.g., \cite{kontinen2009,galliani2012,kontinen2023} for different quantifiers, and \cite{vaananen2008,hella20192,ciardellibook,anttila2025} for different modalities), but among the most prominent are the quantifiers adopted in first-order dependence logic \cite{vaananen2007} and the modalities used in modal dependence logic \cite{vaananen2008}. These quantifiers and modalities distribute over the global disjunction $\vvee$, so we expect the deep-inference rules to remain sound and hence in principle we see no obstacle to modular extensions of $\mathsf{GT}$ with rules for these quantifiers/modalities. However, it is not immediately obvious whether some extension with standard rules for the quantifiers/modalities would be cutfree complete.

A similar natural question is whether $\mathsf{GT}$ can be modularly extended with implication rules to provide cutfree complete calculi for extensions of $\PLVVEE$ with implication connectives. The most natural of such extensions is the extension of $\PLVVEE$ with the \emph{intuitionistic implication}  also used in inquisitive logic. $\PLVVEE$ with this implication is essentially the logic InqB$_\vee$, the extension of propositional inquisitive logic InqB with the split disjunction $\vee$ (see, for instance, \cite{ciardelli2016dependency,ciardelliIemhoffYang}), and so extending $\mathsf{GT}$ in a modular way with rules for this implication would also yield a calculus for InqB. Given that resolutions are adapted from inquisitive logic, and that analogues of the semantic results on which the functioning of our system depends (see Section \ref{di:section:resolutions}) hold for InqB and InqB$_\vee$, we expect our approach to carry over smoothly. However, unlike the binary connectives of $\PLVVEE$, the intuitionistic implication does not distribute over $\vvee$ in a straightforward way; this has the consequence that resolutions in InqB and InqB$_\vee$ have a more complex structure than in $\PLVVEE$. The more complex resolutions, in turn, appear to necessitate a more involved and less intuitive set of deep-inference rules. Moving from standard Gentzen-style sequents to a slightly more general type of structure may allow for a simpler set of rules; we aim to investigate this in future work.

The logic $\PLVVEE$ extends classical propositional logic $\PL$ with the inquisitive disjunction $\vvee$. One can similarly extend intuitionistic propositional logic with $\vvee$---the intuitionistic versions of InqB and InqB$_\vee$, for instance, are studied in \cite{ciardelliIemhoffYang,Muller}. It is thus natural to ask whether replacing the classical rules in the $\mathsf{G3cp}$-base of $\mathsf{GT}$ with their intuitionistic counterparts would yield a cutfree complete calculus for the fragment of propositional intuitionistic logic with the connectives $\bot,\lnot,\land,\vee$, extended with $\vvee$. (Note that the set ${\bot, \lnot, \land, \lor}$ is not functionally complete for propositional intuitionistic logic \cite{mckinsey}, so the logic is indeed a proper fragment.) Preliminary investigation suggests that this is the case, but we leave the detailed investigation of intuitionistic team logics for future work.

As for other team logics, given our strong focus on the properties of a single team-semantic connective---the inquisitive disjunction $\vvee$---as well as the connections between the cutfree completeness of our system and the resolution theorem, which depends, in turn, on the downward closure of $\PLVVEE$ as well as the union closure of the $\vvee$-free fragment, we expect any application of this approach to logics which do not involve $\vvee$ and which do not also share these additional features to involve some further innovations and/or stronger proof-theoretic machinery. There are, however, some other interesting team logics which do incorporate $\vvee$ and do share these features (for instance, the extension of \emph{Team Linear Temporal Logic} with $\vvee$ in \cite{kontinen2023setsemanticsasynchronousteamltl}).

We chose to employ standard Gentzen-style sequents and only one set of deep-inference rules in order to keep our system as close to a Gentzen-style system as possible. However, it might be interesting to develop a system for $\PLVVEE$ making more extensive use of deep-inference rules and techniques, and perhaps operating on a more general type of sequent/structure, in order to make comparisons with other systems in the deep-inference literature, as well as to make use of results in said literature. In particular, it appears that one can formulate an appealing system similar to $\mathsf{GT}$ in the \emph{calculus of structures} \cite{guglielmi,guglielmistrassburger2001,brunnler2006}. Another approach to which team logics may prove amenable is that of \emph{nested sequents} \cite{bull,kashima,brunnler2009,poggiolesi2009,poggiolesi}. More extensive use of deep-inference rules and a more general notion of sequent/structure might also allow for a smoother and more uniform and modular treatment of other team logics, at the expense of losing the close correspondence with Gentzen-style systems.

To summarize, our approach was tailored for $\PLVVEE$, and the resulting system, with its simplicity and clarity, appears to be a good fit for this logic. Whether this approach can be generalised to other team logics remains to be explored. The intriguing links in our setting between structural properties and team-semantic closure properties do suggest that further investigation of team logics which differ in their closure properties from $\PLVVEE$ may prove fruitful.

\bibliographystyle{elsarticle-harv} 
\bibliography{bibb.bib}

\begin{thebibliography}{56}
\expandafter\ifx\csname natexlab\endcsname\relax\def\natexlab#1{#1}\fi
\providecommand{\url}[1]{\texttt{#1}}
\providecommand{\href}[2]{#2}
\providecommand{\path}[1]{#1}
\providecommand{\DOIprefix}{doi:}
\providecommand{\ArXivprefix}{arXiv:}
\providecommand{\URLprefix}{URL: }
\providecommand{\Pubmedprefix}{pmid:}
\providecommand{\doi}[1]{\href{http://dx.doi.org/#1}{\path{#1}}}
\providecommand{\Pubmed}[1]{\href{pmid:#1}{\path{#1}}}
\providecommand{\bibinfo}[2]{#2}
\ifx\xfnm\relax \def\xfnm[#1]{\unskip,\space#1}\fi
\bibitem[{Abramsky and Väänänen(2009)}]{abramsky}
\bibinfo{author}{Abramsky, S.}, \bibinfo{author}{Väänänen, J.}, \bibinfo{year}{2009}.
\newblock \bibinfo{title}{From {IF} to {BI}}.
\newblock \bibinfo{journal}{Synthese} \bibinfo{volume}{167}, \bibinfo{pages}{207--230}.
\newblock \URLprefix \url{https://doi.org/10.1007/s11229-008-9415-6}, \DOIprefix\doi{10.1007/s11229-008-9415-6}.
\bibitem[{Anttila(2025)}]{anttila2025}
\bibinfo{author}{Anttila, A.}, \bibinfo{year}{2025}.
\newblock \bibinfo{title}{Not Nothing: Nonemptiness in Team Semantics}.
\newblock Ph.D. thesis. University of Amsterdam.
\newblock \URLprefix \url{https://hdl.handle.net/11245.1/2fe80942-29ff-41a8-a584-3b6545d1fbd7}.
\bibitem[{Br{\"u}nnler(2003)}]{brunnlerthesis}
\bibinfo{author}{Br{\"u}nnler, K.}, \bibinfo{year}{2003}.
\newblock \bibinfo{title}{Deep Inference and Symmetry in Classical Proofs}.
\newblock Ph.D. thesis. Technische Universität Dresden.
\newblock \URLprefix \url{https://people.bath.ac.uk/ag248/kai/phd.pdf}.
\bibitem[{Br\"{u}nnler(2006)}]{brunnler2006}
\bibinfo{author}{Br\"{u}nnler, K.}, \bibinfo{year}{2006}.
\newblock \bibinfo{title}{Locality for classical logic}.
\newblock \bibinfo{journal}{Notre Dame J. Form. Log.} \bibinfo{volume}{47}, \bibinfo{pages}{557--580}.
\newblock \URLprefix \url{https://doi.org/10.1305/ndjfl/1168352668}, \DOIprefix\doi{10.1305/ndjfl/1168352668}.
\bibitem[{Br\"{u}nnler(2009)}]{brunnler2009}
\bibinfo{author}{Br\"{u}nnler, K.}, \bibinfo{year}{2009}.
\newblock \bibinfo{title}{Deep sequent systems for modal logic}.
\newblock \bibinfo{journal}{Arch. Math. Logic} \bibinfo{volume}{48}, \bibinfo{pages}{551--577}.
\newblock \URLprefix \url{https://10.1007/s00153-009-0137-3}, \DOIprefix\doi{10.1007/s00153-009-0137-3}.
\bibitem[{Brünnler(2010)}]{brunnler2010nestedsequents}
\bibinfo{author}{Brünnler, K.}, \bibinfo{year}{2010}.
\newblock \bibinfo{title}{Nested Sequents}.
\newblock \bibinfo{type}{Habilitationsschrift}. University of Bern.
\newblock \URLprefix \url{https://arxiv.org/abs/1004.1845}, \DOIprefix\doi{10.48550/arXiv.1004.1845}, \href{http://arxiv.org/abs/1004.1845}{{\tt arXiv:1004.1845}}.
\bibitem[{Bull(1992)}]{bull}
\bibinfo{author}{Bull, R.A.}, \bibinfo{year}{1992}.
\newblock \bibinfo{title}{Cut elimination for propositional dynamic logic without *}.
\newblock \bibinfo{journal}{Mathematical Logic Quarterly} \bibinfo{volume}{38}, \bibinfo{pages}{85--100}.
\newblock \URLprefix \url{https://onlinelibrary.wiley.com/doi/abs/10.1002/malq.19920380107}, \DOIprefix\doi{10.1002/malq.19920380107}.
\bibitem[{Chen and Ma(2017)}]{ChenMa}
\bibinfo{author}{Chen, J.}, \bibinfo{author}{Ma, M.}, \bibinfo{year}{2017}.
\newblock \bibinfo{title}{Labelled sequent calculus for inquisitive logic}, in: \bibinfo{editor}{Alexandru~Baltag, J.S.}, \bibinfo{editor}{Yamada, T.} (Eds.), \bibinfo{booktitle}{Proceedings of the 6th International Workshop on Logic, Rationality, and Interaction (LORI)}. \bibinfo{publisher}{Berlin: Springer}. volume \bibinfo{volume}{10455} of \textit{\bibinfo{series}{Lecture {Notes} in {Computer} {Science}}}, pp. \bibinfo{pages}{526--540}.
\newblock \URLprefix \url{https://doi.org/10.1007/978-3-662-55665-8_36}, \DOIprefix\doi{10.1007/978-3-662-55665-8_36}.
\bibitem[{Ciardelli(2016)}]{ciardelli2016dependency}
\bibinfo{author}{Ciardelli, I.}, \bibinfo{year}{2016}.
\newblock \bibinfo{title}{Dependency as question entailment}, in: \bibinfo{editor}{Abramsky, S.}, \bibinfo{editor}{Kontinen, J.}, \bibinfo{editor}{V{\"a}{\"a}n{\"a}nen, J.}, \bibinfo{editor}{Vollmer, H.} (Eds.), \bibinfo{booktitle}{Dependence Logic: Theory and Applications}. \bibinfo{publisher}{Springer International Publishing}, \bibinfo{address}{Cham}, pp. \bibinfo{pages}{129--181}.
\newblock \URLprefix \url{https://doi.org/10.1007/978-3-319-31803-5_8}, \DOIprefix\doi{10.1007/978-3-319-31803-5_8}.
\bibitem[{Ciardelli(2018)}]{ciardelli2018}
\bibinfo{author}{Ciardelli, I.}, \bibinfo{year}{2018}.
\newblock \bibinfo{title}{Questions as information types}.
\newblock \bibinfo{journal}{Synthese} \bibinfo{volume}{195}, \bibinfo{pages}{321--365}.
\newblock \URLprefix \url{https://doi.org/10.1007/s11229-016-1221-y}, \DOIprefix\doi{10.1007/s11229-016-1221-y}.
\bibitem[{Ciardelli(2022)}]{ciardellibook}
\bibinfo{author}{Ciardelli, I.}, \bibinfo{year}{2022}.
\newblock \bibinfo{title}{Inquisitive Logic---Consequence and Inference in the Realm of Questions}. volume~\bibinfo{volume}{60} of \textit{\bibinfo{series}{Trends in Logic---Studia Logica Library}}.
\newblock \bibinfo{publisher}{Springer, Cham}.
\newblock \URLprefix \url{https://doi.org/10.1007/978-3-031-09706-5}, \DOIprefix\doi{10.1007/978-3-031-09706-5}.
\bibitem[{Ciardelli et~al.(2018)Ciardelli, Groenendijk and Roelofsen}]{inqsembook}
\bibinfo{author}{Ciardelli, I.}, \bibinfo{author}{Groenendijk, J.}, \bibinfo{author}{Roelofsen, F.}, \bibinfo{year}{2018}.
\newblock \bibinfo{title}{{Inquisitive Semantics}}.
\newblock \bibinfo{publisher}{Oxford University Press}.
\newblock \URLprefix \url{https://doi.org/10.1093/oso/9780198814788.001.0001}, \DOIprefix\doi{10.1093/oso/9780198814788.001.0001}.
\bibitem[{Ciardelli et~al.(2020)Ciardelli, Iemhoff and Yang}]{ciardelliIemhoffYang}
\bibinfo{author}{Ciardelli, I.}, \bibinfo{author}{Iemhoff, R.}, \bibinfo{author}{Yang, F.}, \bibinfo{year}{2020}.
\newblock \bibinfo{title}{Questions and dependency in intuitionistic logic}.
\newblock \bibinfo{journal}{Notre Dame J. Form. Log.} \bibinfo{volume}{61}, \bibinfo{pages}{75--115}.
\newblock \URLprefix \url{https://doi.org/10.1215/00294527-2019-0033}, \DOIprefix\doi{10.1215/00294527-2019-0033}.
\bibitem[{Ciardelli and Roelofsen(2011)}]{CiardelliRoelofsen2011}
\bibinfo{author}{Ciardelli, I.}, \bibinfo{author}{Roelofsen, F.}, \bibinfo{year}{2011}.
\newblock \bibinfo{title}{Inquisitive logic}.
\newblock \bibinfo{journal}{J. Philos. Logic} \bibinfo{volume}{40}, \bibinfo{pages}{55--94}.
\newblock \URLprefix \url{https://doi.org/10.1007/s10992-010-9142-6}, \DOIprefix\doi{10.1007/s10992-010-9142-6}.
\bibitem[{D'Agostino(2019)}]{dagostino}
\bibinfo{author}{D'Agostino, G.}, \bibinfo{year}{2019}.
\newblock \bibinfo{title}{Uniform interpolation for propositional and modal team logics}.
\newblock \bibinfo{journal}{J. Logic Comput.} \bibinfo{volume}{29}, \bibinfo{pages}{785--802}.
\newblock \URLprefix \url{https://doi.org/10.1093/logcom/exz006}, \DOIprefix\doi{10.1093/logcom/exz006}.
\bibitem[{Dummett(2000)}]{dummett}
\bibinfo{author}{Dummett, M.}, \bibinfo{year}{2000}.
\newblock \bibinfo{title}{Elements of Intuitionism}.
\newblock \bibinfo{publisher}{Oxford University Press}.
\newblock \URLprefix \url{https://doi.org/10.1093/oso/9780198505242.001.0001}, \DOIprefix\doi{10.1093/oso/9780198505242.001.0001}.
\bibitem[{Frittella et~al.(2016)Frittella, Greco, Palmigiano and Yang}]{frittella2016}
\bibinfo{author}{Frittella, S.}, \bibinfo{author}{Greco, G.}, \bibinfo{author}{Palmigiano, A.}, \bibinfo{author}{Yang, F.}, \bibinfo{year}{2016}.
\newblock \bibinfo{title}{A multi-type calculus for inquisitive logic}, in: \bibinfo{editor}{Väänänen, J.}, \bibinfo{editor}{Hirvonen, {\AA}.}, \bibinfo{editor}{de~Queiroz, R.} (Eds.), \bibinfo{booktitle}{Logic, {Language}, {Information}, and {Computation}}, \bibinfo{publisher}{Springer}, \bibinfo{address}{Berlin, Heidelberg}. pp. \bibinfo{pages}{215--233}.
\newblock \URLprefix \url{https://doi.org/10.1007/978-3-662-52921-8_14}, \DOIprefix\doi{10.1007/978-3-662-52921-8_14}.
\bibitem[{Galliani(2012)}]{galliani2012}
\bibinfo{author}{Galliani, P.}, \bibinfo{year}{2012}.
\newblock \bibinfo{title}{Inclusion and exclusion dependencies in team semantics — {On} some logics of imperfect information}.
\newblock \bibinfo{journal}{Ann. Pure Appl. Logic} \bibinfo{volume}{163}, \bibinfo{pages}{68--84}.
\newblock \URLprefix \url{https://doi.org/10.1016/j.apal.2011.08.005}.
\bibitem[{Girard(1987)}]{girard}
\bibinfo{author}{Girard, J.Y.}, \bibinfo{year}{1987}.
\newblock \bibinfo{title}{Linear logic}.
\newblock \bibinfo{journal}{Theoretical Computer Science} \bibinfo{volume}{50}, \bibinfo{pages}{1--101}.
\newblock \URLprefix \url{https://www.sciencedirect.com/science/article/pii/0304397587900454}, \DOIprefix\doi{10.1016/0304-3975(87)90045-4}.
\bibitem[{Groenendijk(2009)}]{groenendijk}
\bibinfo{author}{Groenendijk, J.}, \bibinfo{year}{2009}.
\newblock \bibinfo{title}{Inquisitive semantics: Two possibilities for disjunction}, in: \bibinfo{editor}{Bosch, P.}, \bibinfo{editor}{Gabelaia, D.}, \bibinfo{editor}{Lang, J.} (Eds.), \bibinfo{booktitle}{Logic, Language, and Computation}, \bibinfo{publisher}{Springer Berlin Heidelberg}, \bibinfo{address}{Berlin, Heidelberg}. pp. \bibinfo{pages}{80--94}.
\newblock \URLprefix \url{https://doi.org/10.1007/978-3-642-00665-4_8}, \DOIprefix\doi{10.1007/978-3-642-00665-4_8}.
\bibitem[{Guglielmi(2007)}]{guglielmi}
\bibinfo{author}{Guglielmi, A.}, \bibinfo{year}{2007}.
\newblock \bibinfo{title}{A system of interaction and structure}.
\newblock \bibinfo{journal}{ACM Trans. Comput. Logic} \bibinfo{volume}{8}, \bibinfo{pages}{1–es}.
\newblock \URLprefix \url{https://doi.org/10.1145/1182613.1182614}, \DOIprefix\doi{10.1145/1182613.1182614}.
\bibitem[{Guglielmi and Stra{\ss}burger(2001)}]{guglielmistrassburger2001}
\bibinfo{author}{Guglielmi, A.}, \bibinfo{author}{Stra{\ss}burger, L.}, \bibinfo{year}{2001}.
\newblock \bibinfo{title}{Non-commutativity and {MELL} in the calculus of structures}, in: \bibinfo{editor}{Fribourg, L.} (Ed.), \bibinfo{booktitle}{Computer Science Logic}, \bibinfo{publisher}{Springer Berlin Heidelberg}, \bibinfo{address}{Berlin, Heidelberg}. pp. \bibinfo{pages}{54--68}.
\newblock \URLprefix \url{https://doi.org/10.1007/3-540-44802-0_5}, \DOIprefix\doi{10.1007/3-540-44802-0_5}.
\bibitem[{Hella et~al.(2019)Hella, Kuusisto, Meier and Vollmer}]{hella20192}
\bibinfo{author}{Hella, L.}, \bibinfo{author}{Kuusisto, A.}, \bibinfo{author}{Meier, A.}, \bibinfo{author}{Vollmer, H.}, \bibinfo{year}{2019}.
\newblock \bibinfo{title}{Satisfiability of modal inclusion logic: Lax and strict semantics}.
\newblock \bibinfo{journal}{ACM Trans. Comput. Logic} \bibinfo{volume}{21}.
\newblock \URLprefix \url{https://doi.org/10.1145/3356043}, \DOIprefix\doi{10.1145/3356043}.
\bibitem[{Hintikka(1996)}]{hintikka1996}
\bibinfo{author}{Hintikka, J.}, \bibinfo{year}{1996}.
\newblock \bibinfo{title}{The Principles of Mathematics Revisited}.
\newblock \bibinfo{publisher}{Cambridge University Press, Cambridge}.
\newblock \URLprefix \url{https://doi.org/10.1017/CBO9780511624919}, \DOIprefix\doi{10.1017/CBO9780511624919}.
\bibitem[{Hintikka and Sandu(1989)}]{hintikka1989}
\bibinfo{author}{Hintikka, J.}, \bibinfo{author}{Sandu, G.}, \bibinfo{year}{1989}.
\newblock \bibinfo{title}{Informational independence as a semantical phenomenon}, in: \bibinfo{booktitle}{Logic, methodology and philosophy of science, {VIII} ({M}oscow, 1987)}. \bibinfo{publisher}{North-Holland, Amsterdam}. volume \bibinfo{volume}{126} of \textit{\bibinfo{series}{Stud. Logic Found. Math.}}, pp. \bibinfo{pages}{571--589}.
\newblock \URLprefix \url{https://doi.org/10.1016/S0049-237X(08)70066-1}, \DOIprefix\doi{10.1016/S0049-237X(08)70066-1}.
\bibitem[{Hodges(1997a)}]{hodges1997}
\bibinfo{author}{Hodges, W.}, \bibinfo{year}{1997}a.
\newblock \bibinfo{title}{Compositional semantics for a language of imperfect information}.
\newblock \bibinfo{journal}{Log. J. IGPL} \bibinfo{volume}{5}, \bibinfo{pages}{539--563}.
\newblock \URLprefix \url{https://doi.org/10.1093/jigpal/5.4.539}, \DOIprefix\doi{10.1093/jigpal/5.4.539}.
\bibitem[{Hodges(1997b)}]{hodges1997b}
\bibinfo{author}{Hodges, W.}, \bibinfo{year}{1997}b.
\newblock \bibinfo{title}{Some strange quantifiers}, in: \bibinfo{booktitle}{Structures in logic and computer science}. \bibinfo{publisher}{Springer, Berlin}. volume \bibinfo{volume}{1261} of \textit{\bibinfo{series}{Lecture Notes in Comput. Sci.}}, pp. \bibinfo{pages}{51--65}.
\newblock \URLprefix \url{https://doi.org/10.1007/3-540-63246-8_4}, \DOIprefix\doi{10.1007/3-540-63246-8\_4}.
\bibitem[{Iemhoff(2019a)}]{iemhoff2019aml}
\bibinfo{author}{Iemhoff, R.}, \bibinfo{year}{2019}a.
\newblock \bibinfo{title}{Uniform interpolation and sequent calculi in modal logic}.
\newblock \bibinfo{journal}{Arch. Math. Logic} \bibinfo{volume}{58}, \bibinfo{pages}{155--181}.
\newblock \URLprefix \url{https://doi.org/10.1007/s00153-018-0629-0}, \DOIprefix\doi{10.1007/s00153-018-0629-0}.
\bibitem[{Iemhoff(2019b)}]{iemhoff2019apal}
\bibinfo{author}{Iemhoff, R.}, \bibinfo{year}{2019}b.
\newblock \bibinfo{title}{Uniform interpolation and the existence of sequent calculi}.
\newblock \bibinfo{journal}{Ann. Pure Appl. Logic} \bibinfo{volume}{170}, \bibinfo{pages}{102711}.
\newblock \URLprefix \url{https://doi.org/10.1016/j.apal.2019.05.008}, \DOIprefix\doi{10.1016/j.apal.2019.05.008}.
\bibitem[{Kashima(1994)}]{kashima}
\bibinfo{author}{Kashima, R.}, \bibinfo{year}{1994}.
\newblock \bibinfo{title}{Cut-free sequent calculi for some tense logics}.
\newblock \bibinfo{journal}{Studia Logica: An International Journal for Symbolic Logic} \bibinfo{volume}{53}, \bibinfo{pages}{119--135}.
\newblock \URLprefix \url{https://doi.org/10.1007/BF01053026}, \DOIprefix\doi{10.1007/BF01053026}.
\bibitem[{Kontinen et~al.(2023)Kontinen, Sandstr\"{o}m and Virtema}]{kontinen2023setsemanticsasynchronousteamltl}
\bibinfo{author}{Kontinen, J.}, \bibinfo{author}{Sandstr\"{o}m, M.}, \bibinfo{author}{Virtema, J.}, \bibinfo{year}{2023}.
\newblock \bibinfo{title}{Set semantics for asynchronous team{LTL}: Expressivity and complexity}, in: \bibinfo{editor}{Leroux, J.}, \bibinfo{editor}{Lombardy, S.}, \bibinfo{editor}{Peleg, D.} (Eds.), \bibinfo{booktitle}{48th International Symposium on Mathematical Foundations of Computer Science (MFCS 2023)}, \bibinfo{publisher}{Schloss Dagstuhl -- Leibniz-Zentrum f{\"u}r Informatik}, \bibinfo{address}{Dagstuhl, Germany}. pp. \bibinfo{pages}{60:1--60:14}.
\newblock \URLprefix \url{https://drops.dagstuhl.de/entities/document/10.4230/LIPIcs.MFCS.2023.60}, \DOIprefix\doi{10.4230/LIPIcs.MFCS.2023.60}.
\bibitem[{Kontinen and V\"a\"an\"anen(2009)}]{kontinen2009}
\bibinfo{author}{Kontinen, J.}, \bibinfo{author}{V\"a\"an\"anen, J.}, \bibinfo{year}{2009}.
\newblock \bibinfo{title}{On definability in dependence logic}.
\newblock \bibinfo{journal}{J. Log. Lang. Inf.} \bibinfo{volume}{18}, \bibinfo{pages}{317--332}.
\newblock \URLprefix \url{https://doi.org/10.1007/s10849-009-9082-0}, \DOIprefix\doi{10.1007/s10849-009-9082-0}.
\bibitem[{Kontinen and Yang(2023)}]{kontinen2023}
\bibinfo{author}{Kontinen, J.}, \bibinfo{author}{Yang, F.}, \bibinfo{year}{2023}.
\newblock \bibinfo{title}{Complete logics for elementary team properties}.
\newblock \bibinfo{journal}{J. Symb. Log.} \bibinfo{volume}{88}, \bibinfo{pages}{579–619}.
\newblock \URLprefix \url{https://doi.org/10.1017/jsl.2022.80}, \DOIprefix\doi{10.1017/jsl.2022.80}.
\bibitem[{Litak and Sano(2025)}]{litak2025boundedinquisitivelogicssequent}
\bibinfo{author}{Litak, T.}, \bibinfo{author}{Sano, K.}, \bibinfo{year}{2025}.
\newblock \bibinfo{title}{Bounded inquisitive logics: Sequent calculi and schematic validity}.
\newblock \URLprefix \url{https://arxiv.org/abs/2507.13946}, \DOIprefix\doi{10.48550/arXiv.2507.13946}, \href{http://arxiv.org/abs/2507.13946}{{\tt arXiv:2507.13946}}.
\bibitem[{Lück(2018)}]{luck2018}
\bibinfo{author}{Lück, M.}, \bibinfo{year}{2018}.
\newblock \bibinfo{title}{Axiomatizations of team logics}.
\newblock \bibinfo{journal}{Ann. Pure Appl. Logic} \bibinfo{volume}{169}, \bibinfo{pages}{928--969}.
\newblock \URLprefix \url{https://doi.org/10.1016/j.apal.2018.04.010}, \DOIprefix\doi{10.1016/j.apal.2018.04.010}.
\bibitem[{Maehara(1961)}]{maehara}
\bibinfo{author}{Maehara, S.}, \bibinfo{year}{1961}.
\newblock \bibinfo{title}{{C}raig no interpolation theorem [{C}raig's interpolation theorem] [in {J}apanese]}.
\newblock \bibinfo{journal}{Sūgaku} \bibinfo{volume}{12}, \bibinfo{pages}{235--237}.
\newblock \URLprefix \url{https://doi.org/10.11429/sugaku1947.12.235}, \DOIprefix\doi{10.11429/sugaku1947.12.235}.
\bibitem[{McKinsey(1939)}]{mckinsey}
\bibinfo{author}{McKinsey, J.C.C.}, \bibinfo{year}{1939}.
\newblock \bibinfo{title}{Proof of the independence of the primitive symbols of {H}eyting's calculus of propositions}.
\newblock \bibinfo{journal}{J. Symb. Log.} \bibinfo{volume}{4}, \bibinfo{pages}{155--158}.
\newblock \URLprefix \url{https://doi.org/10.2307/2268715}, \DOIprefix\doi{10.2307/2268715}.
\bibitem[{M\"{u}ller(2022)}]{Muller}
\bibinfo{author}{M\"{u}ller, V.}, \bibinfo{year}{2022}.
\newblock \bibinfo{title}{On the Proof Theory of Inquisitive Logic}.
\newblock Master's thesis. University of Amsterdam, The Netherlands.
\newblock \URLprefix \url{https://eprints.illc.uva.nl/id/eprint/2278/1/MoL-2023-26.text.pdf}.
\bibitem[{M{\"u}ller(2024)}]{muller2024}
\bibinfo{author}{M{\"u}ller, V.}, \bibinfo{year}{2024}.
\newblock \bibinfo{title}{Labelled sequent calculi for inquisitive modal logics}, in: \bibinfo{editor}{Metcalfe, G.}, \bibinfo{editor}{Studer, T.}, \bibinfo{editor}{de~Queiroz, R.} (Eds.), \bibinfo{booktitle}{Logic, Language, Information, and Computation}, \bibinfo{publisher}{Springer Nature Switzerland}, \bibinfo{address}{Cham}. pp. \bibinfo{pages}{122--139}.
\newblock \URLprefix \url{https://doi.org/10.1007/978-3-031-62687-6_9}, \DOIprefix\doi{10.1007/978-3-031-62687-6_9}.
\bibitem[{Pitts(1992)}]{Pitts_1992}
\bibinfo{author}{Pitts, A.M.}, \bibinfo{year}{1992}.
\newblock \bibinfo{title}{On an interpretation of second order quantification in first order intuitionistic propositional logic}.
\newblock \bibinfo{journal}{J. Symb. Log.} \bibinfo{volume}{57}, \bibinfo{pages}{33--52}.
\newblock \URLprefix \url{https://doi.org/10.2307/2275175}, \DOIprefix\doi{10.2307/2275175}.
\bibitem[{Poggiolesi(2009)}]{poggiolesi2009}
\bibinfo{author}{Poggiolesi, F.}, \bibinfo{year}{2009}.
\newblock \bibinfo{title}{The method of tree-hypersequents for modal propositional logic}, in: \bibinfo{editor}{Makinson, D.}, \bibinfo{editor}{Malinowski, J.}, \bibinfo{editor}{Wansing, H.} (Eds.), \bibinfo{booktitle}{Towards Mathematical Philosophy}. \bibinfo{publisher}{Springer}. volume~\bibinfo{volume}{28} of \textit{\bibinfo{series}{Trends in Logic}}, pp. \bibinfo{pages}{31--51}.
\newblock \URLprefix \url{https://doi.org/10.1007/978-1-4020-9084-4}, \DOIprefix\doi{10.1007/978-1-4020-9084-4}.
\bibitem[{Poggiolesi(2010)}]{poggiolesi}
\bibinfo{author}{Poggiolesi, F.}, \bibinfo{year}{2010}.
\newblock \bibinfo{title}{Gentzen Calculi for Modal Propositional Logic}. volume~\bibinfo{volume}{32} of \textit{\bibinfo{series}{Trends in Logic}}.
\newblock \bibinfo{publisher}{Springer}.
\newblock \URLprefix \url{https://doi.org/10.1007/978-90-481-9670-8}, \DOIprefix\doi{10.1007/978-90-481-9670-8}.
\bibitem[{Punčochář(2015)}]{puncochar}
\bibinfo{author}{Punčochář, V.}, \bibinfo{year}{2015}.
\newblock \bibinfo{title}{Weak negation in inquisitive semantics}.
\newblock \bibinfo{journal}{J. Log. Lang. Inf.} \bibinfo{volume}{24}, \bibinfo{pages}{323--355}.
\newblock \URLprefix \url{https://doi.org/10.1007/s10849-015-9219-2}, \DOIprefix\doi{10.1007/s10849-015-9219-2}.
\bibitem[{Pym(2002)}]{pym}
\bibinfo{author}{Pym, D.J.}, \bibinfo{year}{2002}.
\newblock \bibinfo{title}{The Semantics and Proof Theory of the Logic of Bunched Implications}.
\newblock \bibinfo{publisher}{Springer}.
\newblock \URLprefix \url{https://doi.org/10.1007/978-94-017-0091-7}, \DOIprefix\doi{10.1007/978-94-017-0091-7}.
\bibitem[{Sano(2009)}]{sano}
\bibinfo{author}{Sano, K.}, \bibinfo{year}{2009}.
\newblock \bibinfo{title}{Sound and complete tree-sequent calculus for inquisitive logic}, in: \bibinfo{editor}{Ono, H.}, \bibinfo{editor}{Kanazawa, M.}, \bibinfo{editor}{de~Queiroz, R.} (Eds.), \bibinfo{booktitle}{Logic, Language, Information and Computation}, \bibinfo{publisher}{Springer Berlin Heidelberg}, \bibinfo{address}{Berlin, Heidelberg}. pp. \bibinfo{pages}{365--378}.
\newblock \URLprefix \url{https://doi.org/10.1007/978-3-642-02261-6_29}, \DOIprefix\doi{10.1007/978-3-642-02261-6_29}.
\bibitem[{Sano and Virtema(2015)}]{sanovirtema}
\bibinfo{author}{Sano, K.}, \bibinfo{author}{Virtema, J.}, \bibinfo{year}{2015}.
\newblock \bibinfo{title}{Axiomatizing propositional dependence logics}, in: \bibinfo{editor}{Kreutzer, S.} (Ed.), \bibinfo{booktitle}{24th EACSL Annual Conference on Computer Science Logic (CSL 2015)}, \bibinfo{publisher}{Schloss Dagstuhl -- Leibniz-Zentrum f{\"u}r Informatik}, \bibinfo{address}{Dagstuhl, Germany}. pp. \bibinfo{pages}{292--307}.
\newblock \URLprefix \url{https://drops.dagstuhl.de/entities/document/10.4230/LIPIcs.CSL.2015.292}, \DOIprefix\doi{10.4230/LIPIcs.CSL.2015.292}.
\bibitem[{Schütte(1977)}]{schutte}
\bibinfo{author}{Schütte, K.}, \bibinfo{year}{1977}.
\newblock \bibinfo{title}{Proof Theory}.
\newblock \bibinfo{publisher}{Springer-Verlag}.
\bibitem[{Troelstra and Schwichtenberg(2000)}]{troelstra}
\bibinfo{author}{Troelstra, A.S.}, \bibinfo{author}{Schwichtenberg, H.}, \bibinfo{year}{2000}.
\newblock \bibinfo{title}{Basic Proof Theory}.
\newblock \bibinfo{edition}{2} ed., \bibinfo{publisher}{Cambridge University Press}, \bibinfo{address}{New York}.
\newblock \URLprefix \url{https://doi.org/10.1017/CBO9781139168717}, \DOIprefix\doi{10.1017/CBO9781139168717}.
\bibitem[{Väänänen(2007)}]{vaananen2007}
\bibinfo{author}{Väänänen, J.}, \bibinfo{year}{2007}.
\newblock \bibinfo{title}{Dependence Logic: A New Approach to Independence Friendly Logic}. volume~\bibinfo{volume}{70} of \textit{\bibinfo{series}{London Mathematical Society Student Texts}}.
\newblock \bibinfo{publisher}{Cambridge University Press, Cambridge}.
\newblock \URLprefix \url{https://doi.org/10.1017/CBO9780511611193}, \DOIprefix\doi{10.1017/CBO9780511611193}.
\bibitem[{Väänänen(2008)}]{vaananen2008}
\bibinfo{author}{Väänänen, J.}, \bibinfo{year}{2008}.
\newblock \bibinfo{title}{Modal dependence logic}, in: \bibinfo{booktitle}{New Perspectives on Games and Interaction}. \bibinfo{publisher}{Amsterdam Univ. Press, Amsterdam}. volume~\bibinfo{volume}{4} of \textit{\bibinfo{series}{Texts Log. Games}}, pp. \bibinfo{pages}{237--254}.
\newblock \URLprefix \url{https://hdl.handle.net/11245/1.300509}.
\bibitem[{Väänänen and Hodges(2010)}]{vaananen2010}
\bibinfo{author}{Väänänen, J.}, \bibinfo{author}{Hodges, W.}, \bibinfo{year}{2010}.
\newblock \bibinfo{title}{Dependence of variables construed as an atomic formula}.
\newblock \bibinfo{journal}{Ann. Pure Appl. Logic} \bibinfo{volume}{161}, \bibinfo{pages}{817--828}.
\newblock \URLprefix \url{https://www.sciencedirect.com/science/article/pii/S0168007209001195}, \DOIprefix\doi{10.1016/j.apal.2009.06.009}.
\bibitem[{Yang(2014)}]{yang2014}
\bibinfo{author}{Yang, F.}, \bibinfo{year}{2014}.
\newblock \bibinfo{title}{On Extensions and Variants of Dependence Logic: A study of intuitionistic connectives in the team semantics setting}.
\newblock Ph.D. thesis. University of Helsinki.
\newblock \URLprefix \url{https://helda.helsinki.fi/items/0b455380-1a6b-47e5-99bc-910c78e06f37}.
\bibitem[{Yang(2022)}]{yang2022}
\bibinfo{author}{Yang, F.}, \bibinfo{year}{2022}.
\newblock \bibinfo{title}{Propositional union closed team logics}.
\newblock \bibinfo{journal}{Ann. Pure Appl. Logic} \bibinfo{volume}{173}, \bibinfo{pages}{Paper No. 103102, 35}.
\newblock \URLprefix \url{https://doi.org/10.1016/j.apal.2022.103102}, \DOIprefix\doi{10.1016/j.apal.2022.103102}.
\bibitem[{Yang(2025)}]{yang2024negation}
\bibinfo{author}{Yang, F.}, \bibinfo{year}{2025}.
\newblock \bibinfo{title}{There are (other) ways to negate in propositional team semantics}, in: \bibinfo{booktitle}{Exploring Negation, Modality and Proof (to appear)}. Logic in Asia: Studia Logica Library.
\newblock \URLprefix \url{https://arxiv.org/abs/2410.08413}, \DOIprefix\doi{10.48550/arXiv.2410.08413}, \href{http://arxiv.org/abs/2410.08413}{{\tt arXiv:2410.08413}}.
\bibitem[{Yang and V\"{a}\"{a}n\"{a}nen(2016)}]{yangvaananen2016}
\bibinfo{author}{Yang, F.}, \bibinfo{author}{V\"{a}\"{a}n\"{a}nen, J.}, \bibinfo{year}{2016}.
\newblock \bibinfo{title}{Propositional logics of dependence}.
\newblock \bibinfo{journal}{Ann. Pure Appl. Logic} \bibinfo{volume}{167}, \bibinfo{pages}{557--589}.
\newblock \URLprefix \url{https://doi.org/10.1016/j.apal.2016.03.003}, \DOIprefix\doi{10.1016/j.apal.2016.03.003}.
\bibitem[{Yang and V\"{a}\"{a}n\"{a}nen(2017)}]{yang2017}
\bibinfo{author}{Yang, F.}, \bibinfo{author}{V\"{a}\"{a}n\"{a}nen, J.}, \bibinfo{year}{2017}.
\newblock \bibinfo{title}{Propositional team logics}.
\newblock \bibinfo{journal}{Ann. Pure Appl. Logic} \bibinfo{volume}{168}, \bibinfo{pages}{1406--1441}.
\newblock \URLprefix \url{https://doi.org/10.1016/j.apal.2017.01.007}, \DOIprefix\doi{10.1016/j.apal.2017.01.007}.

\end{thebibliography}

\end{document}